\documentclass{article}
\usepackage[a4paper]{geometry}

\title{Finite length for unramified $GL_2$: beyond multiplicity one}
\author{Lucrezia Bertoletti}
\date{\today}

\usepackage{amsmath, amssymb, amsthm, mathrsfs, bbm}
\usepackage{mathtools}
\usepackage{tikz-cd} 
\usepackage{enumerate}
\usepackage{verbatim, hyperref, cleveref}
\usepackage[shortlabels]{enumitem}
\usepackage[
backend=biber,
style=numeric,
sorting=nyt,
minnames=5,
maxnames=5
]{biblatex}
\usepackage{graphicx}

\theoremstyle{plain}
\newtheorem{thm}{Theorem}[subsection]
\newtheorem{prop}[thm]{Proposition}
\newtheorem{lemma}[thm]{Lemma}
\newtheorem{deflem}[thm]{Definition-lemma}
\newtheorem{corollary}[thm]{Corollary}

\newtheorem{newthm}{Theorem}[section]

\theoremstyle{remark}
\newtheorem{example}[thm]{Example}
\newtheorem{remark}[thm]{Remark}

\newtheorem{newremark}[newthm]{Remark}

\theoremstyle{definition}
\newtheorem{definition}[thm]{Definition}

\newtheorem{newdefinition}[newthm]{Definition}

\newcommand{\defeq}{:=}

\DeclareMathOperator{\coker}{coker} 
\DeclareMathOperator{\id}{id} 
\DeclareMathOperator{\im}{im} 
\DeclareMathOperator{\gr}{gr} 
\DeclareMathOperator{\grm}{gr_{\mathfrak{m}} } 
\DeclareMathOperator{\Gt}{\widetilde{\Gamma}} 
\DeclareMathOperator{\grL}{gr(\Lambda)}

\DeclareMathOperator{\radGt}{rad_{\widetilde{\Gamma}}} 
\DeclareMathOperator{\soc}{soc} 
\DeclareMathOperator{\socG}{soc_{\Gamma}} 
\DeclareMathOperator{\socGt}{soc_{\Gt}} 
\DeclareMathOperator{\socR}{soc_{GL_2(\mathcal{O}_K)} } 
\DeclareMathOperator{\cosoc}{cosoc} 
\DeclareMathOperator{\cosocG}{cosoc_{\Gamma} } 
\DeclareMathOperator{\cosocGt}{cosoc_{\Gt} } 
\DeclareMathOperator{\cosocR}{cosoc_{GL_2(\mathcal{O}_K)} } 

\DeclareMathOperator{\injGt}{Inj_{\Gt}} 

\DeclareMathOperator{\GL}{GL} 
\DeclareMathOperator{\Gal}{Gal} 
\DeclareMathOperator{\Hom}{Hom} 
\DeclareMathOperator{\Ext}{Ext} 
\DeclareMathOperator{\End}{End} 
 
\newcommand{\repr}{\overline{\rho}} 
\newcommand{\glrepr}{\overline{r}} 
\newcommand{\mgl}{\mathfrak{m}_{\overline{r}}}
\DeclareMathOperator{\F}{\mathbb{F}} 
\DeclareMathOperator{\JH}{JH} 
\DeclareMathOperator{\Proj}{Proj} 
 
\newcommand{\ProjI}{\Proj_{I/Z_1}} 
\DeclareMathOperator{\ProjG}{Proj_{\Gamma}} 
\DeclareMathOperator{\ProjGt}{Proj_{\widetilde{\Gamma}}} 

\DeclareMathOperator{\Frob}{Frob} 
\DeclareMathOperator{\tr}{tr} 
\DeclareMathOperator{\Norm}{Norm} 

\DeclareMathOperator{\Tr}{Tr} 
\DeclareMathOperator{\Ind}{Ind} 
\newcommand{\IndI}{\mathrm{Ind}^{\mathrm{GL}_2(\mathcal{O}_K)}_I} 
 
\DeclareMathOperator{\cInd}{c-\Ind} 
\DeclareMathOperator{\Fr}{Fr} 
\DeclareMathOperator{\Sym}{Sym} 

\newcommand{\mytwist}{\mathrm{det}(\overline{\rho})\omega^{-1}}

\newcommand{\mI}{\mathfrak{m}} 
\newcommand{\mK}{\mathfrak{m}_{K_1} } 
\newcommand{\barmI}{\overline{\mI}}

\newcommand{\PhG}{\Phi\Gamma^{\text{ét} }_{\mathbb{F} } }
\newcommand{\PhGhat}{\widehat{\Phi\Gamma}^{\text{ét} }_{\mathbb{F} }} 

\newcommand{\V}{V}
\newcommand{\E}[1][2f]{\mathrm{E}^{#1}_{\Lambda} }
\newcommand{\Eg}[1][2f]{\mathrm{E}^{#1}_{\grL} }

\newcommand{\Dvee}[1][\pi]{D^\vee_\xi(#1)}
\newcommand{\Nell}[1][\ell]{N(\overline{\rho},{#1})}
\newcommand{\Nnull}{N(\overline{\rho})}

\newcommand{\rwt}[1][\pi',\sigma]{r({#1})}
\newcommand{\rell}[1][\pi',\ell]{r({#1})}
\newcommand{\rnull}[1][\pi']{r({#1})}
\newcommand{\Vwt}[1][\pi',\sigma]{V({#1})}
\newcommand{\Vell}[1][\pi',\ell]{V({#1})}
\newcommand{\Vnull}[1][\pi']{V(#1)}
\newcommand{\Wbar}{\overline{W}_{\chi,3}}
\newcommand{\W}[1]{W_{\chi,#1}}

\newcommand{\IW}{\mathrm{Ind}_{I} ^{\GLring}(\Wbar)}

\newcommand{\GLring}[1][K]{\GL_2(\mathcal{O} _{#1})} 
\newcommand{\GLfield}{\GL_2(K)}
\newcommand{\GLres}{\GL_2(\F_q)}
\newcommand{\inertia}[1][K]{I(\overline{#1}/#1)} 
\newcommand{\KK}{\GL_2(\mathcal{O} _K)K^\times } 
\newcommand{\Mring}[1][K]{\mathrm{M}_2(\mathcal{O}_{#1})} 

\newcommand{\agal}[1][F]{\Gal(\overline{#1}/#1)} 

\newcommand{\lgR}{\lg_{\GLring}} 

\newcommand{\Diagr}{D(\overline{\rho})} 
\newcommand{\DZero}{D_0(\overline{\rho})} 
\newcommand{\DOne}{D_1(\overline{\rho})} 
\newcommand{\DZeroEll}[1][\ell]{D_0 (\overline{\rho})_{#1} } 
\newcommand{\DOneEll}[1][\ell]{D_1 (\overline{\rho})_{#1} } 
\newcommand{\Dx}[1][\sigma]{D_{0,{#1}} (\overline{\rho})}
\newcommand{\DOnex}[1][\sigma]{D_{1,{#1}} (\overline{\rho})}
\newcommand{\DEll}[1][\ell]{D (\overline{\rho})_{#1} }

\newcommand{\diagr}{{D}^{\prime (j)}}
\newcommand{\dZero}{{D_0}^{\prime (j)}}
\newcommand{\dOne}{{D_1}^{\prime(j)}}

\newcommand{\tilDZero}{\widetilde{D}_0 (\overline{\rho})} 
\newcommand{\tilDZeroEll}[1][\ell]{\widetilde{D}_0 (\overline{\rho})_{#1} } 
 
\newcommand{\tilDx}[1][\sigma]{\widetilde{D}_{0,{#1}} (\overline{\rho})}

\newcommand{\xtwoheadrightarrow}[2][]{%
  \xrightarrow[#1]{#2}\mathrel{\mkern-14mu}\rightarrow
}

\NewBibliographyString{toappear}
\NewBibliographyString{inprep}
\DefineBibliographyStrings{english}{%
  toappear = {to appear},
  inprep = {in preparation}
}

\renewbibmacro*{in:}{%
  \iffieldundef{pubstate}
    {}
    {\printfield{pubstate}%
     \setunit{\addspace}%
     \clearfield{pubstate}}%
  \printtext{%
    \bibstring{in}\intitlepunct}}

\begin{document}
\maketitle

\begin{abstract}
Let $p$ be a prime number and 
$K$ a finite unramified extension of 
$\mathbb{Q}_p$. 
Building on recent work of 
Breuil, Herzig, Hu, Morra and Schraen,
we study the smooth mod $p$ representations 
of $\mathrm{GL}_2(K)$ appearing 
in a tower of mod $p$ Hecke eigenspaces 
of the cohomology of Shimura curves,
under mild genericity assumptions but notably 
no multiplicity one assumption at tame level,
and prove that these representations 
are of finite length, thereby extending 
the results of \cite{BHHMS4}.
\end{abstract}

\tableofcontents

\section{Introduction}
	\label{sec:introduction}
\subsection{Main results}
	\label{sec:main-results} 

Fix a prime number $p$,
a totally real number field $F$ unramified
at places above $p$, and a quaternion algebra
$D$ with center $F$ which is split at all places
above $p$ and at exactly one infinite place.
We denote by $\mathbb{A}_F^{ \infty }$
the ring of finite adèles of $F$.
Fix one place $v\mid p$, and a compact open
subgroup $V^v$ of 
$(D \otimes_{F}\mathbb{A}_F^{\infty,v})^{\times}$. 
If $V_v$ is any compact open subgroup
of $(D \otimes_{F}F_v)^{\times} \cong \GL_2(F_v)$,
then $V^vV_v$ is a compact open subgroup of
$(D \otimes_{F}\mathbb{A}^\infty_F)^{\times}$,
and we can consider the associated smooth
projective Shimura curve $X_{V^vV_v}$, 
which is defined over $F$, with the conventions of
\cite[§3.1]{BD14}.

If $\F$ is a finite extension of $\F$,
which we assume to be sufficiently large,
and if
$\overline{r} \colon \Gal( \overline{F}/F)
\to \GL_2(\F)$ is a continuous absolutely
irreducible Galois representation,
then we study
the admissible smooth representation
\begin{equation}
	\label{eq:piShi-def}
\pi(V^{v}) \defeq \varinjlim_{V_v} \Hom_{\Gal(\overline{F}/F)}(\overline{r}, H^1_{ \text{ét}} (X_{V^vV_v} \times_{F}\overline{F}, \F))
\end{equation}
of $\GL_2(F_v)$ over $\F$,
where we let $V_v$ vary in the set of
compact open subgroups of  
$(D \otimes_{F}F_v)^{\times} \cong \GL_2(F_v)$.
The reader should note that,
despite our notation, $\pi(V^{v})$
depends on all global choices, and not only on
the prime-to-$v$ level.
We restrict ourselves to those
$\overline{r}$ that give a nonzero $\pi$.
These representations have seen some
attention in recent years
(e.g.\ in \cite{HW22}, \cite{BHHMS1}, \cite{BHHMS2}, \cite{Wan23}), 
though for the most part they
remain mysterious when $F_v \neq \mathbb{Q}_p$.
An important contribution to
 our understanding comes from 
the recent paper \cite{BHHMS4},
establishing that \eqref{eq:piShi-def} has
finite length, under
some genericity assumption on $\overline{r}$,
and a multiplicity one assumption on 
$\pi$ (which is called the 
\emph{minimal case}).

In this paper
we drop the multiplicity one assumption,
and try to extend the results 
of \cite{BHHMS4}.
Unlike \emph{loc.\ cit.}\ we 
assume from now on that the restriction
$\glrepr|_{I(\overline{F}_v/F_v)} $
of $\glrepr$ to the inertia subgroup at $v$ 
is semisimple.
In the rest of the introduction we 
present our results.

We set $K \defeq F_v$, $f \defeq [K:\mathbb{Q}_p]$,
and $q \defeq p^{f},$
and we denote by $\omega$ the 
mod $p$ cyclotomic character of
$\Gal(\overline{K}/K)$,
which we regard as a character of $K^{\times}$via local class field theory,
 with the convention that uniformisers 
are sent to geometric
Frobenius elements.
We denote by $\omega_f,$ $\omega_{2f} $
Serre's fundamental characters of the 
inertia subgroup $\inertia$ of 
$ \Gal(\overline{K}/K)$ of level
$f$, $2f$ respectively.
Following \cite{BHHMS4}, we say that $\overline{r}$
is \emph{generic} if the following conditions
are satisfied
\begin{enumerate}[(i)]
\item 
\label{point:generic(i)} 
the restriction 
$\overline{r}|_{\Gal(\overline{F}/F(\mu_p))} $ is absolutely irreducible;
\item 
\label{point:generic(ii)} 
	if $w\nmid p$ is a place where either
$D$ or $\overline{r}$ is ramified, then
the framed deformation ring of 
$\overline{r}|_{\Gal(\overline{F}_w/F_w)} $
over the Witt vectors $W(\F)$ is formally smooth;
\item 
\label{point:generic(iii)} 
	the restriction 
$\overline{r}|_{\inertia} $ of $\overline{r}$
to the inertia subgroup at $v$
is, up to twist, of the form
\[
\begin{pmatrix}
\omega_{f}^{\sum_{j=0} ^{f-1}(r_j+1)p^j }	
&  0\\
0 & 1
\end{pmatrix}, 
\]
with $\max \{12,2f+1\} <r_j<p- \max \{15,2f+4\} $,
or
\[
\begin{pmatrix}
\omega_{2f}^{\sum_{j=0} ^{f-1}(r_j+1)p^j }
& 0 \\
0 & 
\omega_{2f}^{\sum_{j=0} ^{f-1}(r_j+1)p^{j+f} }	
\end{pmatrix},
\]
with $\max \{12,2f+1\} \le  r_i\le p-\max \{15,2f+4\}$ for $j>0$,
and with
$\max \{13,2f+2\} \le r_0\le p-\max \{14,2f+3\}$.
\end{enumerate}
The bounds on the $r_i$ of 
\ref{point:generic(iii)} all come
from \cite{BHHMS4}: we don't need any stronger
bound to make our arguments work.
Notice however that (iii) implies that
$\glrepr|_{\inertia}$ is semisimple,
which on the contrary 
is not assumed in \emph{loc.\ cit.}

The following is our main result.
\begin{thm}[\Cref{thm:global-part-fl}]
	\label{thm:pre-main-0}
Assume $\glrepr$ generic.
Then, the representation $\pi(V^{v})$
has finite length as a $\GLfield$-representation.
\end{thm}
Observe that we don't need any hypothesis
on the prime-to-$v$ level $V^{v}$.

The proof of \Cref{thm:pre-main-0}
reduces ``by dévissage''
to the study of some subquotients of \eqref{eq:piShi-def},
which we now describe.
For every $w \neq v$ lying over $p$,
choose an (absolutely) irreducible
$\GLring[w]$-representation
$\sigma_w$,
and set $\sigma_p^{v} \defeq 
\bigotimes_{ \begin{substack}
w\mid p, \\ w\neq v
\end{substack} } \sigma_w $.
Then, our strategy 
(which mirrors the proof of 
\cite[Corollary~8.4.6]{BHHMS1})
is to show that $\pi(V^{v})$ admits 
a \emph{finite} filtration whose subquotients 
are submodules
of $\GLfield$-representations
of the form
\begin{equation}
	\label{eq:piShi-dev} 
\pi(V^{v},\sigma_p^{v})\defeq \varinjlim_{V_v} 
\Hom_{\prod_{
w\mid p, w\neq v
} \GL_2(\mathcal{O}_{F_w} )}
\left(
\sigma_p^{v},
\Hom_{\Gal(\overline{F}/F)}
\left(\overline{r}, H^{1}_{\text{ét}}(X_{V^vV_v} \times_{F}\overline{F}, \F)\right)\right).
\end{equation}
Then, we show that every such
$\pi(V^{v},\sigma_p^{v})$
is of finite length, concluding the proof.

The two theorems below describe in more precise terms
the properties of $\pi(V^{v},\sigma_p^{v})$.
Remember that, by \cite[Theorem~1.9]{BHHMS1},
there is a unique integer
$r \ge 1$, called the \emph{multiplicity},
such that $\dim_{\F} \Hom_{\GLring}(\sigma, \pi) \in \{0,r\}$
for every (absolutely) irreducible representation
$\sigma$ of $\GLring$ over $\F$.

From now on we let $\repr \defeq \overline{r}^\vee |_{\Gal(\overline{K}/K)},$
where $\overline{r}^\vee $ is the 
dual of $\overline{r}$.
While \cite[Theorem~1.1.1]{BHHMS4}
assumes that $r=1$,
the following is a generalisation for arbitrary
$r \ge 1$.
\begin{thm}
	\label{thm:pre-main-1} 
Assume $\overline{r}$ generic,
and let $\pi$ be a $\GLfield$-representation
of the form \eqref{eq:piShi-dev}.
\begin{enumerate}[(i)]
\item 
\label{point:pre-main-1(i)} 
	If $\repr$ is reducible split,
say of the form $\repr \cong 
\begin{pmatrix} 
\chi_1& 0 \\
0 & \chi_2 \\
\end{pmatrix}$,
 then
\begin{equation}
	\label{eq:pi-red-decomposition}
	\pi \cong \Ind_{B(K)} ^{\GLfield}(\chi_2 \otimes \chi_1 \omega^{-1} )^{ \oplus r} \oplus \pi' \oplus 
	\Ind_{B(K)} ^{\GLfield}(\chi_1 \otimes \chi_2 \omega^{-1} )^{ \oplus r},
\end{equation}
where $\pi'$ has finite length, bounded 
above by $r\cdot (f-1)$.
The Jordan-H\"older constituents of
$\pi'$ are supersingular.
\item 
\label{point:pre-main-1(ii)} 
If $\repr$ is irreducible then 
$\pi$ has finite length, bounded above by $r$.
The Jordan-H\"older constituents of
$\pi$ are supersingular.
\end{enumerate}
\end{thm}
Part \ref{point:pre-main-1(i)} follows
from \Cref{thm:finite-length-yay}(ii),
\Cref{thm:principal-series}(i)
and \Cref{cor:supersingularity-for-free}(i),
while \ref{point:pre-main-1(ii)} follows
from \Cref{thm:finite-length-yay}(ii)
and \Cref{cor:supersingularity-for-free}(ii).
Ideally, we should expect that
for every irreducible subquotient $\overline{\pi}$
of $\pi$ we have $[\pi : \overline{\pi}]=r$.
However, such a result is currently out of our reach.

Let $K_1\unlhd \GLring $ 
be the set of matrices with trivial reduction mod $p $, 
$I\le \GLring $ 
the set of matrices which are upper triangular mod $p $,
and $I_1\unlhd I $ 
the set of matrices which are upper unipotent mod $p $.
Define $Z_1 := (1+ p \mathcal{O} _{K} )\cdot \id $
to be the centre of $I_1 $,
and
$\Lambda \defeq \F [\![I_1/Z_1]\!]  $
to be the completed group algebra of $I_1/Z_1 $,
which is a noetherian noncommutative local ring of Krull
dimension $3f$,
with maximal ideal that
we denote by $\mI$.
By \cite[Theorem~1.9]{BHHMS1} 
and by \cite[Theorem~6.3(ii)]{Wan23},
$\pi$ has a central character.
In particular, for all subquotients
$\pi'$ of $\pi$, 
the linear dual $ \Hom_{\F}(\pi', \F)$
is a $\Lambda$-module,
which is moreover finitely generated
since $\pi$ is admissible.
Recall that a nonzero finitely generated 
$\Lambda$-module $M$ is Cohen-Macaulay of
grade $c \ge 0$ if
$ \Ext^{i}_{\Lambda}(M, \Lambda)\neq 0$ if
and only if $i = c$.

The following result generalises
\cite[Theorem~1.1.2]{BHHMS4} to the case
$r \ge 1$.
\begin{thm}[]
	\label{thm:pre-main-2} 
Assume $\overline{r}$ generic, and 
$\repr$ semisimple.
Let $\pi$ be a $\GLfield$-representation
of the form \eqref{eq:piShi-dev}.
\begin{enumerate}[(i)]
\item 
\label{point:pre-main-2(i)} 
	If $\pi'$ is a subquotient of $\pi$,
then its linear dual $ \Hom_{\F}(\pi', \F)$
is a Cohen-Macaulay $\Lambda$-module of 
grade $2f$.
\item 
\label{point:pre-main-2(ii)} 
	Any subquotient of $\pi$ is generated
	by its $\GLring$-socle.
\item
\label{point:pre-main-2(iii)} 
	For any subquotient $\pi'$ of $\pi$,
we have 
\[
\dim_{\F (\!(X)\!) } \Dvee[\pi']= 
\lgR(\socR(\pi')),
\]
where $\Dvee[\pi']$ 
is the cyclotomic $(\varphi,\Gamma)$-module
associated to $\pi'$ in \cite[§2.2.1]{BHHMS2},
and where $\lgR $ denotes the length of a $\GLring$-representation.
\item
\label{point:pre-main-2(iv)} 
	For any subrepresentations $\pi_1 \subseteq \pi_2$ of $\pi$, there is a split
short exact sequence of $\GLring$-representations
\[
0 \to \socR (\pi_1) \to
\socR (\pi_2) \to
\socR (\pi_2/\pi_1) \to 0.
\]
\item
\label{point:pre-main-2(v)} 
	For any subrepresentations $\pi_1 \subseteq \pi_2$ of $\pi$,
and for any $n \ge 1$, there is a
short exact sequence of $\GLring$-representations
\[
0 \to \pi_1[\mI^n] \to \pi_2[\mI^n] \to
(\pi_2/\pi_1)[\mI^n] \to 0,
\]
which is split for $n \le \max\{6,f+1\}$.
\item
\label{point:pre-main-2(vi)} 
	For any subrepresentations $\pi_1 \subseteq \pi_2$ of $\pi$, there is a
short exact sequence of $\GLring$-representations
\[
	0 \to  \pi_1^{K_1} \to
\pi_2^{K_1} \to
(\pi_2/\pi_1)^{K_1} \to 0.
\]
\end{enumerate}
\end{thm}
Part \ref{point:pre-main-2(i)} comes
from \Cref{cor:final-corollary}(iii),
\ref{point:pre-main-2(ii)} is
\Cref{cor:final-corollary}(iv),
\ref{point:pre-main-2(iii)} is
\Cref{cor:final-corollary}(ii),
\ref{point:pre-main-2(iv)} is
\Cref{lemma:soc-ex},
\ref{point:pre-main-2(v)} is
\Cref{cor:pi1-pi2-exact},
\ref{point:pre-main-2(vi)} is
\Cref{prop:K1-exact}.

\subsection{Outline of the article}
We begin with the study of \eqref{eq:piShi-dev},
which we denote simply by $\pi$.
Crucially, if $\repr$ is generic then
there is a unique positive integer $r \ge 1$
such that the following four hypotheses hold:
\begin{enumerate}[(i)]
	\item \label{hypothesis:pre-i}
(\cite[Theorem~8.4.2]{BHHMS1})
If $\DZero $ is the $\GLring$-representation  
defined in \cite[Proposition~13.1]{BP12},
on which we make $K^{\times}$
act by the central character of $\pi$,
then there is an integer $r\ge 1$ such that $\pi^{K_1}\cong D_0(\overline{\rho})^{\oplus r}$ as representations of $\KK$.
	\item \label{hypothesis:pre-ii}
(\cite[Proposition~6.4.6]{BHHMS1} and \cite[Theorem~8.4.2]{BHHMS1})
		If 
$\chi:I\to \F^\times  $ is a smooth character
appearing in 
$\pi^{I_1}=\pi[\mI]$,
then we have an equality of multiplicities
		\[
			[\pi[\mI^3]\colon \chi] = [\pi[\mI]\colon \chi]=r.
		\]

	\item \label{hypothesis:pre-iii}
		(\cite[Theorem~8.2]{BHHMS4} 
and \cite[Theorem~8.4.1]{BHHMS1})
		The linear dual $\pi^\vee
\defeq \Hom_{\F}(\pi, \F)$ is \textit{essentially self-dual} of grade $2f$,
	i.e.\ there exists a $\GL_2(K)$-equivariant isomorphism of $\Lambda$-modules
	\[
		\Ext^{2f}_{\Lambda}(\pi^\vee,\Lambda)\cong \pi^\vee \otimes \chi_\pi.
	\]	

	\item \label{hypothesis:pre-iv}
(\cite[Proposition~2.6.2]{BHHMS4})
If $\chi \colon I \to \F ^\times $ is a 
smooth character and $i \ge 0$, 
then $\Ext^i_{I/Z_1} (\chi,\pi)\neq 0$
if and only if $[\pi[\mI]:\chi]\neq 0$,
in which case
\[
	\dim_{\F} \Ext^i_{I/Z_1} (\chi,\pi)=
	\binom{2f}{i} r.
\]
\end{enumerate}

In \cite{BHHMS2} and in \cite{BHHMS4}
the authors keep to the case where $r=1$;
we extend their arguments to the general case.

Although the representation \eqref{eq:piShi-dev}
is a global object, our arguments are local:
we take any 
admissible smooth representation of $\GLfield$
with a central character
satisfying \ref{hypothesis:pre-i} to \ref{hypothesis:pre-iv},
and deduce some of its properties, 
culminating in \Cref{thm:pre-main-1} 
and \Cref{thm:pre-main-2}.

In \Cref{sec:structure-pi},
we study the subrepresentations of $\pi$,
mainly in terms of their diagrams,
in the sense of \cite{Pas04}.
Remember that to a Galois representation $\repr $ as in \Cref{sec:main-results}.
one can associate a diagram 
$\Diagr=(\DOne \xhookrightarrow{r} \DZero)$,
which we take to be the diagram $\mathcal{D} (\pi_{\text{glob}} (\repr)) $
of \cite[Theorem~1.3]{DL21}
(or equivalently the one of
\cite[Theorem~1.3.2]{BHHMS2}).

With this choice,
the isomorphism of \ref{hypothesis:pre-i}
promotes to an isomorphism
\begin{equation}
	\label{eq:pre-1}
\iota_{\pi} \colon
(\pi^{I_1} \hookrightarrow  \pi^{K_1})
\cong
\Diagr ^{\oplus r} 
\end{equation}
of diagrams:
this is 
\cite[Theorem~1.3.2]{BHHMS2}.

If $\pi'$ is a subrepresentation of $\pi$,
then $\iota_\pi(\pi^{\prime I_1} \hookrightarrow \pi^{\prime K_1})$
is a subdiagram of 
$\Diagr^{\oplus r}\cong \F^{r} \otimes_{\F} \Diagr$,
and in \Cref{sec:structure-srep}
we determine what subdiagrams can occur.
More precisely, remember that
if $\repr$ is irreducible, 
then $\Diagr$ is indecomposable
(as a diagram)
by \cite[Theorem~15.4(i)]{BP12},
and that if $\repr$ is reducible split, 
then $\Diagr$ decomposes as 
$ \bigoplus_{ \ell=0 }^{f} \DEll$
by \cite[Theorem~15.4(ii)]{BP12},
for some subdiagrams $\DEll$
of $\Diagr$
(for their definition, see
\eqref{eq:ell-diagr-def}).
Then, we show in \Cref{cor:direction}
that the subdiagram
$\iota_\pi(\pi^{\prime I_1} \hookrightarrow \pi^{\prime K_1}) \subseteq \F^{r} \otimes_{\F} \Diagr$ 
is always of this form:
if $\repr$ is reducible split, then
\[
	\iota_\pi(\pi^{\prime I_1} \hookrightarrow \pi^{\prime K_1}) =
	\bigoplus_{ \ell=0 }^{ f} V'_{\ell} \otimes_{\F} \DEll \subseteq \F^{r} \otimes_{\F}\Diagr, 
\]
for some $\F$-vector subspaces 
$V'_\ell \subseteq \F^r$,
for  $0 \le \ell \le f$,
and if $\repr$ is irreducible, then
\[
	\iota_\pi(\pi^{\prime I_1} \hookrightarrow \pi^{\prime K_1}) =
	V' \otimes_{\F} \DEll \subseteq 
\F^{r} \otimes_{\F} \Diagr, 
\]
for some $\F$-vector subspace
$V'\subseteq \F^r$.
The technical difficulty that emerges when
$r>1$ 
is then to coordinate all our constructions
to make sure that these $\F$-vector subspaces 
are respected.

\Cref{sec:generation} concludes with
\Cref{thm:principal-series},
where we show the decomposition
\eqref{eq:pi-red-decomposition} of
\Cref{thm:pre-main-1}.

In \Cref{sec:CM-fl},
we prove some of the properties that
a subquotient $\pi'$ of $\pi$ satisfies.
First, in \Cref{sec:CM},
we determine the structure of
$\grm(\pi^{ \prime \vee}) \defeq 
\bigoplus_{ n \ge 0} \mI^{n}\pi^{ \prime\vee}/\mI^{n+1}\pi^{ \prime\vee}$
as a module over
$\grL \defeq
\bigoplus_{ n \ge 0} \mI^{n}/\mI^{n+1}$.
In particular, we find that it is Cohen-Macaulay
of grade $2f$, and 
\Cref{thm:pre-main-2}(i)
follows from
\cite[Proposition~III.2.2.4]{LvO}.

In \Cref{sec:fl}, we establish
the upper bounds on the length of $\pi'$
(and in particular of $\pi$), 
which concludes the proof of
\Cref{thm:pre-main-1}(i),(ii).

In \Cref{sec:mK-torsion}, 
we 
conclude the proof of
\Cref{thm:pre-main-2}.

Finally, in \Cref{sec:global}, we prove
the finite length of \eqref{eq:piShi-def}
passing from the finite length of
\eqref{eq:piShi-dev}, by 
 ``dévissage,'' thus concluding the proof of
\Cref{thm:pre-main-0}.

\paragraph{Acknowledgements:}
The author wishes to thank her advisor, Prof.\ Christophe Breuil, and her co-advisor, Prof.\ Stefano Morra,
for suggesting this topic
and for guiding the learning process.
Owing to their help and their encouragement, this has been a challenging
and rewarding experience
for the author.

\section{Preliminaries}

In this section we introduce 
the notation and the material
that will be used throughout the text.
We fix $K/\mathbb{Q}_p $ an unramified extension of (inertial) degree $f$, and let $q \defeq  p^f $.
We let
 $I \le \GLring$ be the subgroup of matrices
 that are upper triangular modulo $p$,
$I_1\unlhd I$ be its maximal pro-$p$ subgroup, consisting of those matrices that are unipotent modulo $p$,
and we let $K_1\le I_1$ be the subgroup of matrices that are sent to the identity modulo $p$,
which is also the pro-$p$ radical of $\GLring $.
In particular, $\GLring/K_1 \cong \GLres$,
which we also denote by $\Gamma$.

We let $N$ be the normaliser of $I$ in $\GL_2(K)$,
which is generated by $I$,  $K^\times $, and the matrix $\Pi \defeq \begin{pmatrix} 0 & 1 \\ p & 0 \end{pmatrix} $.

The following completed group algebra
will play an important role.
\begin{newdefinition}
If $Z_1 := (1+ p \mathcal{O} _{K} )\cdot \id $ denotes the centre of $I_1 $,
then we let
$\Lambda$
be the completed group algebra 
$\Lambda \defeq \mathbb{F} [\![I_1/Z_1]\!]$,
whose unique maximal ideal we call $\mI$,
and which we endow with the $\mI$-adic topology.
\end{newdefinition}
We know from
\cite{Clozel17}
and from
\cite[Theorem 5.3.4]{BHHMS1}, 
while keeping the notation of \cite[paragraph~before~Remark~3.1.2.8]{BHHMS2},
 that the graded ring $\grL \defeq \gr_{\mI}(\Lambda )$
 associated to $\Lambda $ is an 
Auslander-regular
 noncommutative algebra over $\mathbb{F} $
generated by the variables $\{y_i, z_i, h_i\mid 0\le i <f\}$,
which commute for different indices,
and are subject to the commutator relations 
\begin{align}
&[y_i,z_i]=h_i, \nonumber\\ 
	\label{eq:grL-commutator-relations}
&[h_i, y_i]=0, \\
&[h_i, z_i]=0. \nonumber
\end{align}
For the definition of Auslander-regularity,
see \cite[Definition~III.2.1.3]{LvO},
\cite[Definition~III.2.1.7]{LvO}.
In particular, $\Lambda$ is Auslander-regular
by \cite[Theorem~III.2.2.5]{LvO}.
Then, we define
\begin{align}
	\label{eq:R-definition}
	R  \defeq& \grL  /(h_0,\dots, h_{f-1} ) \cong \F [y_i,z_i\mid 0\le i < f], \\
	\label{eq:R-bar-definition}
 \overline{R}  \defeq& \grL/J = R/(y_iz_i \mid 0\le i < f)\cong \\ 
\nonumber \cong&
\F [y_i,z_i\mid 0\le i < f]/(y_iz_i\mid 0\le i < f),
\end{align}
	 where $J:=(h_i,y_iz_i\mid 0\le i < f)$.
	We are interested in the following invariant, which is defined in \cite[paragraph before Lemma~3.1.4.3]{BHHMS2}:
\begin{newdefinition}[]
	\label{def:multiplicity}
		Let $N$ be a finitely generated module over $\grL$, 
annihilated by some power $J^n$ of $J$.
We define the \emph{multiplicity} of $N$ at a minimal prime $\mathfrak{q} $
of $\overline{R}$ by
		\[
			m_{\mathfrak{q} }(N) \defeq  \sum_{i=0} ^{n-1} \lg_{\overline{R}_{\mathfrak{q} } } (J^iN/J^{i+1} N)_{\mathfrak{q} } . 
		\]
	Here, $\lg_{\overline{R}_{\mathfrak{q}}}$ denotes the length of an $\overline{R}_{\mathfrak{q}}$-module.	
	\end{newdefinition}
	Of the $2^f$ minimal prime ideals of $\overline{R} $,
	we fix in particular $\mathfrak{p} _0 \defeq  (z_i\mid 0\le i < f)$,
	and we remind the reader that $m_{\mathfrak{q} }  $
	is additive,
	cf.\ \cite[Lemma~3.1.4.3]{BHHMS2}.

\subsection{Diagrams}
In this section, we review the theory of diagrams
and in particular the construction of the diagram
$\Diagr$ of
\cite[Theorem~13.8]{BP12}.

If $\chi:I \to \mathbb{F}^\times$ is a smooth character, we define
$\chi^s(\cdot) \defeq \chi(\Pi \cdot \Pi^{-1} )$.
Remember that
$\chi$ factors through a character of 
the finite torus $H \defeq I/I_1$.

\begin{definition}[{\cite[Definition~9.7]{BP12}}]
	A \emph{diagram} is a triple $(D_0, D_1, r)$,
where $D_0$ is a smooth representation of $\GLring K^ \times $ over $\mathbb{F}$, $D_1$ is a smooth representation of $N$ over $\mathbb{F}$,
		and $r:D_1\to D_0$ is an $IK^ \times $-equivariant map.

A \emph{basic diagram} is a diagram $(D_0,D_1, r)$ such that $\begin{pmatrix} p & \\ & p \end{pmatrix} $ acts trivially
and such that
$r$ induces an isomorphism $D_1 \xrightarrow{\sim}
 D_0^{I_1} \hookrightarrow D_0 $.
\end{definition}

\begin{example}
 If $\pi $ is a smooth representation of $\GLring $, then it induces the diagram
		\[
			(\pi^{I_1}
\xhookrightarrow{r} \pi^{K_1}),
		\]
	where $r $ is the canonical inclusion.	
\end{example}

We are also interested in another diagram:
the diagram 
$\Diagr $ of 
 \cite[Proposition~13.1]{BP12}. We now recall its construction.
Fix once and for all an embedding
$\mathbb{F} _{q^2} \hookrightarrow \mathbb{F}  $.
If $\inertia $ denotes the inertia subgroup of $\Gal(\overline{K}/K) $,
we denote by 
\begin{align*}
	\omega_f: &I(\overline{K}/ K)\twoheadrightarrow \mathbb{F}_q^\times  \hookrightarrow \mathbb{F}^\times , \\
	\omega_{2f} : &I(\overline{K}/ K)\twoheadrightarrow \mathbb{F}_{q^2}^\times  \hookrightarrow \mathbb{F}^\times 
\end{align*}
Serre's fundamental characters of level
$f$, $2f$ respectively.

Following the second paragraph of
\cite[§1.3]{BHHMS4},
we say that 
 a continuous Galois representation 
 $ \repr:\Gal(\overline{K}/K)\to \GL_2(\mathbb{F} )$ 
is $n$-\emph{generic},
where $n \ge 0$ is a nonnegative integer,
if, up to twist, 
$\repr|_{\inertia} 
$
is not isomorphic to $\omega \oplus 1$,
and if one of the following holds:
\begin{enumerate}
	\item \emph{(the irreducible case)}   the restriction $\repr|_{\inertia}  $ is isomorphic, up to twist, to
		\[
			\begin{pmatrix}
				\omega_{f}^{\sum_{j=0} ^{f-1}(r_j+1)p^j }	
			& \ast \\
			0 & 1
			\end{pmatrix} ,
		\]
for some integers $r_j $ with
$ n\le r_j \le p-3-n$,
	\item \emph{(the reducible case)}
		the restriction $\repr|_{\inertia}  $ is isomorphic, up to twist, to
		\[
			\begin{pmatrix}
				\omega_{2f}^{\sum_{j=0} ^{f-1}(r_j+1)p^j }	
			& 0 \\
			0 & 
		\omega_{2f}^{\sum_{j=0} ^{f-1}(r_j+1)p^{j+f} }	
			\end{pmatrix} ,
		\]
	for some integers $r_j $ with
	$n+1\le r_0 \le p-2-n$,
	and with
	$n\le r_j \le p-3-n$
	for $j>0 $.
\end{enumerate}
When $\repr$ is irreducible,
note that the condition
$\repr|_{\inertia} 
\not \cong 
\omega \oplus 1$
up to inertia is equivalent to
$(r_0,\dots,r_{f-1} )\notin
\{(0,\dots,0),(p-3,\dots,p-3) \} $.
In particular, $\repr$ is $0$-generic
precisely when it is generic in the 
sense of \cite[Definition~11.7]{BP12}.

To a $0$-generic $\repr$ 
we can associate a set of Serre weights $W(\repr) $,
following
\cite[§§3.1-2]{BDJ10}.
Remember that a Serre weight
$\sigma $ is an isomorphism class of irreducible representations of $\GLring $ over $\mathbb{F} $,
or equivalently of $\Gamma$.

A classification of Serre weights
can be found in \cite[Proposition~2.17]{NewYork}:
it states that
Serre weights can be uniquely written as 
\begin{equation}
\Sym^{r_0} \F^2 \otimes_{\F} 
(\Sym^{r_1} \F^2)^{\Fr} \otimes_{\F} 
\cdots \otimes_{\F} 
(\Sym^{r_{f-1}} \F^2)^{\Fr^{f-1}} \otimes_{\F} 
{\det}^m,
\label{eq:weight-characterisation}
\end{equation}
for $0\le r_i \le p-1 $ and for $0\le m < q-1 $.
Here, the notation 
$(\ )^{\Fr^{i}}$
denotes the restriction of scalars
via the group homomorphism
$\Gamma \xrightarrow{\Fr^{i}} \Gamma$
that sends 
$ \left( \begin{smallmatrix}
a & b \\
c & d
\end{smallmatrix} \right) \mapsto 
\left( \begin{smallmatrix}
a^{p^i} & b^{p^{i}} \\
c^{p^i} & d^{p^{i}}
\end{smallmatrix} \right)$.
We denote the Serre weight
\eqref{eq:weight-characterisation}
by $(r_0,\dots, r_{f-1} )\otimes \det^m
$.
From this characterisation one can also see
that 
$\sigma^{I_1} $ is always 1-dimensional,
and that if the character $\chi:I\to \F^\times  $ describes the action of $I $ on $\sigma^{I_1} $, 
then we can recover $\sigma $ from $\chi $
as long as $\chi\neq \chi^s $.
Note that the genericity assumption on 
$\repr $ makes sure that for all $\sigma\in W(\repr) $
we have $\chi^s \neq \chi$.

Following \cite[§2]{HW22}, we say that
a Serre weight is $n$-\emph{generic},
where $n \ge 0$ is a nonnegative integer,
if the integers $r_i$ of \eqref{eq:weight-characterisation},
for $i \in \{0, \dots, f-1\}$,
satisfy
$n \le r_i \le p-2-n$.
We say that a smooth character 
$ \chi \colon I \to \F^{\times}$
is $n$-generic if $\chi = \sigma^{I_1}$
for an $n$-generic Serre weight $\sigma$.
If $\repr$ is $n$-generic, then 
we remark that any $\sigma \in W(\repr)$
is $n$-generic, and any $\lambda \in \mathscr{P}$
is $(n-1)$-generic if $n \ge 1$.

Suppose that $\repr$ is $0$-generic.
In order to  construct the diagram $\Diagr $, 
we begin with the following definition, coming from \cite[p.~65, last paragraph]{BP12}:
\begin{deflem}[{\cite[Corollary~3.12]{BP12}}]
	\label{deflem:I}
If $\sigma,\tau$ 
are two Serre weights,
then there exists at most one (up to isomorphism) indecomposable $\Gamma $-representation 
with socle $\sigma $ and cosocle $\tau $,
and such that $\sigma $ appears with multiplicity one.
We let $I(\sigma, \tau) $ be this representation, if it exists, and we let $I(\sigma, \tau)=0  $ otherwise.

\end{deflem}
Recall that the socle filtration 
$(\soc_i(M))_{i \ge 0} $ of a $\Gamma$-module
$M$ is defined inductively as follows:
$\soc_{0}(M)=0$, and for $i >  0$,
$\soc_i(M)$ is the preimage
of $\socG (M/\soc_{i-1} (M))$ in $M$.
If there exists an $\ell \ge 0$ such that
$\soc_\ell(M)=M$, then we call the
smallest such $\ell$ the 
\emph{Loewy length} of $M$.

Then, for any $I(\sigma, \tau)$,
we let $\ell(\sigma, \tau)\in \mathbb{Z} _{>0} \cup \infty $ 
 be $\infty $ if $I(\sigma, \tau)=0$,
and otherwise the Loewy length of $I(\sigma,\tau)$.
Finally, we set $\ell(\repr, \tau) \defeq  \min \left\{ \ell(\sigma,\tau)\mid \sigma\in W(\repr) \right\}  $.
If $\ell(\repr,\tau)<\infty $, this minimum 
is realised by exactly one $\sigma \in W(\repr) $, 
cf.\ \cite[Lemma~12.8]{BP12}.
In this case,
following \cite[Paragraph before Lemma~13.3]{BP12}
we set $I(\repr,\tau) \defeq  I(\sigma,\tau) $.
\label{def:I-for-inductive}

Define $\Dx $ as the direct limit 
\begin{equation}
\label{eq:diagr-from-I}
\Dx \defeq \varinjlim_{\le} I(\repr, \tau) ,
\end{equation}
where the direct limit is taken over
the set of $I(\repr, \tau) $ with socle $\sigma$,
to which we give the partial order defined in 
\cite[Paragraph after Lemma~13.3]{BP12}.
It is clear by construction that 
$\socG \Dx[\sigma] = \sigma$.
Then, according to \cite[Proposition~13.4]{BP12}, we can define $D_0(\repr) $ as the direct sum 
\begin{equation} \label{eq:diagr-def} D_0(\repr) \defeq  \bigoplus_{\sigma \in W(\repr)} \Dx[\sigma],  \end{equation}
which is multiplicity free
by \cite[Corollary~13.5]{BP12}.
We regard all these as representations of $\GLring $, by inflation of scalars.

Next, we define 
\begin{equation}
	\label{eq:diagr1-def} 
\begin{array}{ll}
\DOnex[\sigma] \defeq \Dx[\rho]^{I_1}, & 
\quad \DOne\defeq 
\DZero^{I_1}= \bigoplus_{ \sigma \in W(\repr )} \DOnex[\sigma].
\end{array}
\end{equation}
This determines the diagram $\Diagr$
up to the choice of certain parameters,
which we take to be the ones of
\cite[Theorem~1.3]{DL21}.
In any case, notice that $\Pi$ will send
an eigenvector $v_\chi \in \DOne$
 with eigencharacter $\chi$ 
to an eigenvector $\Pi v_{\chi} $
with eigencharacter $\chi^{s}$.

Finally, we have $K^\times$ act on $\DZero $ in the following way:
we consider the character 
$\mytwist$ of the Galois group,
where $\omega $ is the mod $p $ cyclotomic character,
and we see this as a character $K^\times \to  \F^\times $ via local class field theory, which
we normalise so that uniformisers correspond to geometric Frobeniuses.
Then, we let $K^{\times}$ act by
$\mytwist$,
and since
$\mytwist \colon\agal[K]\to 
\F^{\times}$
is tame (wild inertia is a 
pro-$p$ subgroup of $\agal[K]$),
we see that $p$ acts trivially on $\DZero$.

As always, suppose that $\repr$
is $0$-generic.
Following \cite[§4]{Bre14},
if $\lambda = (\lambda_0(x_0), \dots,\lambda_{f-1}(x_{f-1}))$
is an $f$-tuple, with $\lambda_{i} (x_i)
\in  \mathbb{Z} \pm x_i$
for $i \in \{0, \dots, f-1\}$,
we define
\begin{equation}
	\label{eq:e-lambda-def}
\begin{aligned}
e(\lambda) &\defeq \frac{1}{2} 
\left( \sum _{i=0}^{f-1} p^{i}(x_i-\lambda_i(x_i))\right)
\text{ if }\lambda_{f-1}(x_{f-1})
\in \mathbb{Z}+x_{f-1} , \\
e(\lambda) &\defeq \frac{1}{2} 
\left(p^{f}-1+ \sum _{i=0}^{f-1} p^{i}(x_i-\lambda_i(x_i))\right)
\text{ if }\lambda_{f-1}(x_{f-1})
\in \mathbb{Z}-x_{f-1}.
\end{aligned}
\end{equation}

Then, we let $\mathscr{P}$ be the set of 
$f$-tuples 
defined in \emph{loc.\ cit.}\ and denoted there by $\mathscr{PD}$
if $\repr$ is reducible,
and by $\mathscr{PID}$ if $\repr$
is irreducible.
To $\lambda \in \mathscr{P}$
(which is of the form $ \lambda_{i} (x_i)
\in  \mathbb{Z} \pm x_i$
for $i \in \{0, \dots, f-1\}$)
we associate the Serre weight
$\sigma_{\lambda} \defeq(\lambda(r_0), \dots,\lambda(r_{f-1}))
\otimes \det^{e(\lambda)(r_0,\dots,r_{f-1})}$,
and let 
$\chi_{\lambda} \colon I \to \F^{\times}$
be the character acting on 
$\sigma_\lambda^{I_1}$.
It follows from Proposition 4.2 of 
\emph{loc.\ cit.}
that $\lambda \mapsto \chi_\lambda$
is a bijection between $\mathscr{P}$
and the set of eigencharacters
of $\DOne$.

We also let $\mathscr{D}$
be the subset of $\mathscr{P}$
defined in \cite[§4]{Bre14}
and denoted there by $ \mathscr{D}$
if $\repr$ is reducible,
and by $\mathscr{ID}$ if $\repr$ 
is irreducible.
By \cite[Lemma~11.2]{BP12}
and \cite[Lemma~11.4]{BP12},
$\mathscr{D}$ is in bijection with
the Serre weights of $W(\repr)$
via $\lambda \mapsto \sigma_{\lambda}$.
We report here the definition
of $\mathscr{D}$ when $\repr$ is \emph{reducible},
for later use.
\begin{definition}[]
	\label{def:D-red-irr}
If $\repr$ is reducible, we define 
$\mathscr{D}$ as
$\{x_0, p-3-x_0\}$ for $f=1$,
while for $f>1$ we let
$(\lambda_0(x_0), \dots,\lambda_{f-1}(x_{f-1})) \in \mathscr{D}$
if and only if the following three
conditions are satisfied, for 
$i \in \{0, \dots, f-1\}$:
\begin{enumerate}[(i)]
\item 
	\label{condition:D-red-i}
$\lambda_i(x_i) \in \{x_i,x_i+1,p-2-x_i,p-3-x_i\}$,
\item 
	\label{condition:D-red-ii}
$\lambda_i(x_i) \in \{x_i, x_i+1\} \implies
\lambda_{i+1}(x_{i+1}) \in \{x_{i+1} ,p-2-x_{i+1}\}$,
\item 
	\label{condition:D-red-iii}
$\lambda_i(x_i) \in \{p-2-x_i, p-3-x_i-1\} \implies
\lambda_{i+1}(x_{i+1})\in \{p-3-x_{i+1},x_{i+1}+1\}$,
\end{enumerate}
with the conventions 
$x_f = x_0$, $\lambda_f(x_f)=\lambda_0(x_0)$.
\end{definition}

To a Serre weight $\sigma\in W(\repr)$ 
we can associate a positive integer $\ell(\sigma)$
as follows. 
\begin{definition}
	\label{def:J-and-ell}
If $\lambda \in \mathscr{P}$,
following \cite[Eq.~(10)]{BHHMS4} we define 
$J_{\lambda} \defeq \{j \in \{0, \dots, f-1\}
\mid
\lambda_j(x_j) \in \{x_j + 1, x_j + 2,
p-3-x_j\}\}$,
and we set the length $\ell(\lambda) $ 
of $\lambda$ as 
$\ell(\lambda) \defeq |J_{\lambda}|$.
By \Cite[§11]{BP12}),
the assignment $\lambda \mapsto J_{\lambda}$
gives a bijection between $\mathscr{D}
$
and the set of subsets of $\{0, \dots, f-1\}$.

For $0 \le \ell \le f$, we set 
$\mathscr{P}_{\ell} \defeq 
\{\lambda \in \mathscr{P} \mid \ell(\lambda)=\ell\} $.
For a Serre weight $\sigma\in W(\repr) $, if 
$\lambda\in \mathscr{D} \subseteq \mathscr{P}$
is the unique
$f$-tuple such that $\sigma^{I_1}= \chi_\lambda	,$
then we define the length 
$\ell(\sigma) $ to be $\ell(\lambda)$,
and for $0 \le \ell \le f$ we set 
$W(\repr)_{\ell} \defeq 
\{\sigma\in W(\repr) \mid \ell(\sigma)=\ell\} $.

\end{definition}
In the case where $\repr $ is reducible split,
it is shown in \cite[Theorem~15.4(ii)]{BP12} that $\Diagr $ decomposes into a direct sum
\begin{equation}
	\label{eq:ell-decomposition}
\Diagr = \bigoplus_{\ell=0} ^f \DEll	
\end{equation}
of diagrams, where the subdiagram
$\DEll = (\DOneEll[\ell]
\hookrightarrow \DZeroEll[\ell])$
of $\Diagr$ is defined as follows:
\begin{equation}
\label{eq:ell-diagr-def}
\begin{array}{ll}
\DZeroEll[\ell] \defeq \displaystyle \bigoplus_{\sigma \in W(\repr)_{\ell}}
\Dx[\sigma], \quad & 
\DOneEll[\ell] \defeq \DZeroEll[\ell]^{I_1}, \nonumber
\end{array}
\end{equation}
with the unique diagram structure that makes 
$\DEll$ into a subdiagram of $\Diagr$.

We conclude this section by introducing
weight cycling on $W(\repr)$.
First, a notation coming from \cite[p.~9]{BP12}.
\begin{definition}
\label{def:s-at-D0-level}
Given a Serre weight $\sigma $,
if the character $\chi:I\to \overline{\mathbb{F}} _p ^\times  $
describes the action of $I $ on $\sigma^{I_1} $,
then there exists a unique Serre weight, which we denote by $\sigma^{[s]} $,
such that $\sigma^{[s]} $ is distinct from $\sigma $
and such that $\chi^s$ describes the action of $I $ on $(\sigma^{[s]})^{I_1} $.
\end{definition}
Like before, if
$\sigma\in W(\repr) $,
and
if $\chi $
is the character given by the action of $I $ on $\sigma^{I_1} $,
then $\chi^s \neq \chi$ 
follows from the genericity assumption on $\repr $.
In particular,
the condition $\sigma^{[s]}\neq \sigma$
is automatic.

Now we define weight cycling, 
for which we employ the notation of \cite[§15, second paragraph]{BP12}.
For simplicity, we always write $\delta $ instead of $\delta_r $ or $\delta_i $.
\begin{definition}
For a given Serre weight $\sigma \in W(\repr) $,
which we regard inside $\socR \DZero  $,
consider
$\chi $ the character of $I $ given by $\sigma^{I_1} $,
then consider
$\chi^s $,
which also occurs inside
$\DOne $ 
(indeed, as the image of $\sigma^{I_1}$
under the action of $\Pi$).
We define $\delta(\sigma) $ to be the unique Serre weight $\delta(\sigma) \in W(\repr)$
such that $\chi^s \subseteq D_{1,\delta(\sigma)} (\repr) $,
or equivalently such that $\ell(\delta(\sigma), \sigma^{[s ] } ) = \ell(\repr, \sigma^{[s]} ) $, cf.\ \eqref{eq:diagr-from-I}.
Then, we define
	$\delta^n(\sigma) $ by iterating this procedure.
	\label{def:weight-cycling}
	\end{definition}
The fact that this definition agrees with the one given in \emph{loc.\!\ cit.\ }follows from \cite[Lemma~15.2]{BP12}
applied to the case $\tau=\sigma $
(note that in this case
we have $\mathcal{S} ^+=\mathcal{S} ^-=\emptyset $,
with the notation of the paragraph before
\cite[Lemma~15.1]{BP12}).

\subsection{$(\varphi, \Gamma)$-modules}
In this section we review the functor $\Dvee$
of \cite{Breuil15}
taking values in (pro-)étale 
$(\varphi,\Gamma)$-modules.
Let 
\begin{equation}
	\label{eq:N0-def}
N_0 \defeq  \left\{ \begin{pmatrix} 1 & \mathcal{O} _K \\ 
& 1\end{pmatrix}  \right\} \cong \mathcal{O} _K,
\end{equation}
which comes with the trace map 
$N_0 \cong \mathcal{O} _K 
\xtwoheadrightarrow{\Tr}
\mathbb{Z} _p,$
and let $N_1$ be its kernel.
\begin{definition}
	\label{def:Colmez_generalised}	
If $T_{2,K} \subseteq \GL_{2,K}  $
denotes the maximal torus of diagonal matrices, then we define the morphisms of algebraic groups
\begin{align*}
	\xi: \mathbb{G} _m 
	& \to T_{2,K}   \\
	x & \mapsto \begin{pmatrix} x & \\ & 1 \end{pmatrix} 
\end{align*}
and
\begin{align*}
	\theta: T_{2,K} 
	& \to \mathbb{G} _m \\
	\begin{pmatrix} x_1 & \\ & x_2 \end{pmatrix} & \mapsto x_1 ,
\end{align*}
which coincide with 
\cite[example~2.1.1.3]{BHHMS2} for $n=2$.
Note that, for $x\in \mathbb{Z} _p\setminus \{0\} $,
$\xi(x)N_1\xi(x^{-1} )\subseteq N_1 $.
\end{definition}

Now fix a field of coefficients $\mathbb{F}$,
which we assume to be a (big enough) finite field of coefficients, of characteristic $p$.
\begin{remark}
	\label{rem:rep-to-mod}	
If $\pi$ is a smooth representation  of $\GLfield$
over $\F $,
we can take its $N_1$-invariants
and we can see $\pi^{N_1} $ as a smooth representation of $N_0/N_1 \cong \mathbb{Z} _p$.
\end{remark}
With this in mind, from now on we fix the identification
\[
	\mathbb{F}[\![ N_0/N_1]\!] \xrightarrow{\sim} \mathbb{F} [\![ \mathbb{Z} _p ]\!] \cong \mathbb{F} [\![ X ]\!],
\]
where the last isomorphism sends $X$ to $[1] - 1$.

Define $\mathbb{F} [\![X]\!] [F]$ to be the noncommutative polynomial ring over $\mathbb{F} [\![X]\!] $ with the relation $FS(X)= S(X^p)F$ for $S(X)\in \mathbb{F} [\![X]\!] $.
Then, 
we enrich the structure of  $\mathbb{F} [\![X]\!] $-module on
$\pi^{N_1} $
to an $\mathbb{F} [\![X]\!] [F]$-module structure
by making $F$ act as 
$F(v) = \sum_{n_1\in N_1/{\xi(p)} N_1 {\xi(p)}^{-1}} 
n_1\xi(p)v \in \pi^{N_1 } $ on $v\in \pi^{N_1 }  $.

We can also endow $\pi^{N_1} $ with an action of $\mathbb{Z} _p ^ \times  $, 
by making $x \in \mathbb{Z} _p ^ \times  $
act by $\xi(x) $.
The fact that, for 
$x\in \mathbb{Z} _p ^ \times   $,
$\xi(x) $ normalises $N_1$ implies that
the action of $\mathbb{Z} _p ^ \times   $ commutes with $F$, while
the identity
\[
	\begin{pmatrix} x & 0 \\0 & 1 \end{pmatrix} 	
	\begin{pmatrix} 1 & 1 \\0 & 1 \end{pmatrix} 
	\begin{pmatrix} x ^{-1} &0 \\0 & 1 \end{pmatrix} 
	=
	\begin{pmatrix} 1 & x \\0 & 1  \end{pmatrix},
\]
implies the commutation relation
\[
	\xi(x) \circ (1+X) = (1+X)^x \circ \xi(x)
\]
of endomorphisms of $\pi^{N_1} $.

We adopt the convention of
\cite{Breuil15}
of denoting by $\PhG$ the category of 
finite-dimensional étale $(\varphi, \Gamma) $-modules
over $\mathbb{F} (\!(X)\!)  $,
and of denoting by $\PhGhat $ its pro-completion. 
Again following \emph{loc.\!\ cit.},
we now define a contravariant functor $\Dvee [-]$.
\begin{definition}
	\label{def:Dvee}
Given $\pi$ a smooth representation of $\GLfield $
over $\F $,
we can consider the set of all finitely generated 
$\mathbb{F} [\![X]\!] [F] $-submodules
$M \subseteq \pi^{N_1} $
which are stable by the action of  
$\mathbb{Z} _p^\times  $ 
and which are moreover \emph{admissible},
i.e.\ such that $M[X]= \{m\in M\mid Xm=0\} $
is finite dimensional over $\mathbb{F}  $.
If $M^\vee \defeq  \Hom_{\mathbb{F} } (M,\mathbb{F} ) $ is the 
algebraic $\mathbb{F}  $-linear dual of $M $,
we define
\[
	D^\vee_\xi(\pi) \defeq  \varprojlim_{M} M^\vee[X ^{-1}]
\]
to be the inverse limit over this set.
It naturally lies in $\PhGhat$ by
\cite[Proposition~2.7(i)]{Breuil15}.
\end{definition}

\subsection{The main four hypotheses}
	\label{sec:hypotheses} 
From now on, we assume that 
$\repr$ is semisimple.
Then, we will use the following
set of hypotheses on
an admissible smooth representation 
$\pi$ of $\GL_2(K)$ over $\mathbb{F}$
with a central character:
\begin{enumerate}[(i)]
	\item \label{hypothesis:i}
		there is an integer $r\ge 1$ such that
\[
\pi^{K_1}\cong D_0(\overline{\rho})^{\oplus r}
\]
		as representations of $\KK$;
	\item \label{hypothesis:ii} for any $\lambda\in \mathscr{P} $ , we have an equality of multiplicities
		\[
			[\pi[\mI^{3}]\colon \chi_{\lambda}] = [\pi[\mI]\colon \chi_\lambda];
		\]

	\item \label{hypothesis:iii}
if we let 
$\E[j](-) \defeq  \Ext^j_{\Lambda} (-, \Lambda)
= \Ext^j_{I/Z_1} (-, \Lambda)$,
then the $\Lambda$-module
$\pi^\vee$ is \textit{essentially self-dual}
of grade $2f$,
i.e.\ there exists a $\GL_2(K)$-equivariant isomorphism of $\Lambda$-modules
\[
\E(\pi^\vee)\cong \pi^\vee \otimes (\det(\overline{\rho})\omega^{-1})
\]	
where the action of $\GLfield $ on the
left-hand side is defined by Kohlhaase in 
\cite[Proposition~3.2]{Koh17};

	\item \label{hypothesis:iv}
for any smooth character $\chi \colon I \to \F ^\times $
and for $i \ge 0$
we have $\Ext^i_{I/Z_1} (\chi,\pi)\neq 0$
if and only if $[\pi[\mI]:\chi]\neq 0$,
in which case
\[
	\dim_{\F} \Ext^i_{I/Z_1} (\chi,\pi)=
	\binom{2f}{i} r,
\]
where $r \ge 1$ is the integer of \ref{hypothesis:i},
and where $\Ext^{i}_{I/Z_1}(\chi,-) $ 
is computed in the category of smooth $I/Z_1$-representations.
\end{enumerate}

Notice that, by the proof of \cite[Corollary~5.3.5]{BHHMS1},
hypothesis \ref{hypothesis:ii} implies that the 
$\grL $-module $\grm(\pi^\vee)
\defeq \bigoplus _{i \ge 0}
\mathfrak{m}^{i}\pi^\vee
/\mathfrak{m}^{i+1}\pi^\vee$
is annihilated by $J $, and so
its multiplicity $m_{\mathfrak{q} }(\grm(\pi^\vee)) $ is 
well-defined for each 
minimal prime
$\mathfrak{q}\subseteq \overline{R}$.

\section{On the structure of $\pi$}
	\label{sec:structure-pi}
The main goal of this section is to 
show \Cref{thm:principal-series} below,
stating that $\pi$ contains a direct sum of
$2r$ principal series 
when $\repr$ is reducible split.
In doing so, we set up the building
blocks for the rest of the article.

\subsection{On the structure of subrepresentations}
	\label{sec:structure-srep} 
Assume that $\pi$ satisfies hypothesis \ref{hypothesis:i}
of \Cref{sec:hypotheses},
and set $\V\defeq\F^{r}$.
Remember that,
in the paragraph after \eqref{eq:diagr1-def},
we chose the structure of
diagram on $\Diagr$ to be the same
of \cite[Theorem~1.3]{DL21},
which by the last paragraph of
\cite[§3.4.1]{BHHMS2}
coincides with the one of
\cite[Theorem~3.4.1.1]{BHHMS2}.
By \emph{loc.\ cit.}
we know that we can fix an isomorphism of diagrams
\begin{equation}
	\label{eq:1}
\iota_{\pi} :
( \pi^{I_1}\hookrightarrow\pi^{K_1})
\cong
\V \otimes_{\F} 
\Diagr.
\end{equation}
If $\pi'$ is a subrepresentation of $\pi$,
then $\iota_\pi(\pi^{\prime I_1} \hookrightarrow \pi^{\prime K_1})$
is a subdiagram of $\Diagr^{\oplus r}$,
and the goal of this section is to
determine which subdiagrams can occur.

We begin at the level of socles:
for each Serre weight $\sigma \in W(\repr)$,
let
\begin{equation}
	\label{eq:V-def} 
\Vwt \defeq \Hom_\Gamma(\sigma, \iota_\pi(\pi^{\prime K_1})) \hookrightarrow 
\Hom_\Gamma(\sigma, \iota_\pi(\pi^{K_1})) =
\Hom_\Gamma(\sigma, \V\otimes_{\F}\Diagr) \cong \V,
\end{equation}
where the isomorphism on the right
is given by
\begin{align*}
\Hom_\Gamma(\sigma,\V \otimes_{\F} \Diagr) \cong 
\V \otimes_{\F}\Hom_\Gamma(\sigma, \Diagr) \overset{ \eqref{eq:diagr-def}}{ = }&
\V \otimes_{\F}\Hom_\Gamma(\sigma, \Dx[\sigma])\\=&  
\V \otimes_{\F}\Hom_\Gamma(\sigma, \sigma)\cong 
\V.
\end{align*}
It is important to note that, 
despite our notation, $\Vwt$ also depends 
on the choice of\eqref{eq:1}.
\begin{remark}
	\label{rem:V-are-commensurable} 
With this definition, we can compare 
$\Vwt[\pi',\sigma_1] \subseteq \V \supseteq \Vwt[\pi',\sigma_2]$
for different Serre weights 
$\sigma_1,\sigma_2 \in W(\repr)$,
and we will freely do so throughout the rest.

Notice also that if $\pi_1 \subseteq \pi_2$
are two subrepresentations of $\pi$, then
$\Vwt[\pi_1,\sigma] \subseteq \Vwt[\pi_2,\sigma]$
for all $\sigma \in W(\repr)$.
Moreover,
$\Vwt[-,\sigma]$ only depends
on the $\GLring$-socle in the sense that
$\socR\pi_1 = \socR\pi_2$ implies
$\Vwt[\pi_1,\sigma] =  \Vwt[\pi_2,\sigma]$.
\end{remark}
\bigskip

If $\pi'$ is a subrepresentation of $\pi$
and if $V_{\pi'}^\sigma $ denotes the
$\sigma $-isotypic component of 
$ \socR \pi'$ (which is possibly zero),
then \eqref{eq:V-def} implies that
$\iota_\pi(V_{\pi'} ^{\sigma})=
\Vwt[\pi',\sigma] \otimes _{\F} \sigma$,
and so $\iota_ \pi$ restricts to an isomorphism
\begin{equation}
	\label{diag:V-def}
	\begin{tikzcd}[column sep=small]
	V^\sigma_{\pi'} \ar[r, "\subseteq ", phantom]  \ar[d, "\iota_ \pi"', "\wr"]  & V^\sigma_ \pi \ar[d, "\iota_ \pi", "\wr"'] \\   
	\Vwt \otimes_{\F} \sigma \ar[r, "\subseteq ", phantom] & \V \otimes_{\F} \sigma 
	\end{tikzcd}
\end{equation}
of $\GLring$-representations.
If
\begin{equation}
	\label{eq:def-multiplicity}
	\rwt \defeq [\socR(\pi'):\sigma]
\end{equation}
denotes the multiplicity of $\sigma\in W(\repr) $ in the $\GLring$-socle of $\pi' $,
then $\dim_{\F} \Vwt=\rwt.$

We now promote the inclusion of
$\iota_\pi(V_{\pi'} ^{\sigma})=
\Vwt[\pi',\sigma] \otimes _{\F} \sigma$
into
$\iota_\pi(\pi^{\prime K_1})$
to an inclusion of
$\Vwt[\pi',\sigma] \otimes _{\F} \Dx[\sigma]$.
\begin{lemma}
\label{lemma:fill-up}
Let $\pi $ be a admissible smooth representation of 
$\GLfield $ over $\mathbb{F}  $ satisfying 
hypothesis \ref{hypothesis:i},
and let $\pi' $ be a subrepresentation of $\pi $.
For each
Serre weight $\sigma \in W(\repr)$
the $\GLring $-representation 
$\iota_ \pi(\pi'^{K_1})$
contains $\Vwt \otimes_{\F} \Dx$.
\end{lemma}

\begin{proof}
We fix once and for all a Serre weight $\sigma \in W(\repr)$,
and we set $r' \defeq \rwt$, $V'\defeq \Vwt$.
Let $(e_i)_{i=1} ^r$ be the
standard basis of $\V = \F^r$,
and  notice that we can fix the identification $\iota_\pi $
so that $V'$
is the vector subspace 
$ \left\langle e_1,\dots e_{r'}  \right\rangle_{\F} $
spanned by $e_1, \dots, e_{r'} $.

Now, we follow the proof of 
\cite[Theorem~19.10]{BP12}.
Our strategy will be the following:
we show that $\iota_ \pi(\pi'^{K_1})$
contains 
$V'\otimes_{\F} \Dx[\delta(\sigma)]$.
Then, we iterate this procedure to 
show that
$\iota_ \pi(\pi'^{K_1})$
contains
$V'\otimes_{\F} \Dx[\delta^{n}(\sigma)]$
for all $n\ge 1$.
Since $\delta^n(\sigma)=\sigma$ for some $n\ge 1$,
we can conclude.

\paragraph{Step 1.}
We prove that $\iota_\pi({\pi'}^{K_1})$ contains 
$V' \otimes_{\F}  I(\repr, \sigma^{[s]})$.
We essentially follow
the proof of \cite[Theorem~19.10]{BP12},
but need to be more careful about multiplicities.

If $ \cInd \sigma$
denotes the compact induction of $\sigma $ from 
$\KK $ to $\GLfield $,
then from
\[
	V' \otimes_{\F}  \sigma 
\xrightarrow[\sim]{\iota_\pi ^{-1}}
	V_{\pi'} ^\sigma \subseteq \pi',
\]
coming from \eqref{eq:soc-line-formal-computation},
we obtain by Frobenius reciprocity a nonzero morphism
\begin{equation}
	\label{eq:pre-alpha}
 V' \otimes_{\F}  \cInd \sigma \to \pi'
\end{equation}
of $\GLfield$-representations.
If $R(\sigma)$
is the $\GLring$-subrepresentation
of $\cInd \sigma $
defined in 
\cite[Definition~17.9]{BP12},
then $K_1$ acts trivially on  $R(\sigma)$
by definition: 
in \emph{loc.\ cit.}\ it is defined as a representations of $\Gamma $, 
and we are 
considering it a representation of $\GLring $ by inflation of scalars.
So, \eqref{eq:pre-alpha}
restricts to
$V' \otimes_{\F}  R(\sigma)\to \pi'^{K_1}$,
which we then compose with $\iota_\pi$:
this defines the morphism
\begin{equation}
		\label{eq:alpha-def}
	\begin{tikzcd}
		\alpha \colon V' \otimes_{\F}  R(\sigma) 
		\ar[r] &  
		{\pi'}^{K_1}
		\ar[r, hook, "\iota_\pi"] &  
		\V \otimes_{\F}  \DZero
	\end{tikzcd}
\end{equation}
of $\GLring$-representations.

Note that, by definition of \eqref{eq:pre-alpha},
$\alpha$ restricts to the identity
on $V' \otimes_{\F}  \sigma$:
\begin{equation}
\begin{tikzcd}
		\label{eq:alpha-explicit}
	\id \colon
	V' \otimes_{\F}  \sigma \ar[r, "\sim"',"\iota_\pi ^{-1}"] \ar[d, phantom,  "\rotatebox{-90}{$\subseteq$} "] &
		V^\sigma_{\pi'} 	 \ar[r, "\sim"', "\iota_\pi"] \ar[d, phantom,  "\rotatebox{-90}{$\subseteq$} "] & 
		V' \otimes_{\F}  \sigma \ar[d, "\rotatebox{-90}{$\subseteq$} ", phantom] 	\\
 \alpha \colon V' \otimes_{\F}  R(\sigma) 
 \ar[r] &  
 {\pi'}^{K_1}
 \ar[r, hook, "\iota_\pi"] &  
 \V \otimes_{\F}  \DZero.
\end{tikzcd}
\end{equation}
Now, we use the action of $\Pi $ on $I_1$-invariants: 
for the sake of clarity,
for the rest of the proof
we denote the action
of $\Pi$ on $\DOne$
by $(-)^s$.

For $x \in \sigma^{I_1}\subseteq R(\sigma)^{I_1}$ 
and for $1\le i \le r'$
fixed, we have
$\alpha(e_i \otimes x)= e_i \otimes x$
by \eqref{eq:alpha-explicit}.
Then
$\Pi x$,
which a priori is only an element of 
$\cInd \sigma$, also lies in $R(\sigma)$
(take $J= \emptyset$
and $\tau= \sigma^{[s]}$
in \cite[Definition~17.9]{BP12}),
and we claim that
\begin{equation}
	\label{eq:Pi-and-alpha}
	\alpha(e_i \otimes \Pi x)=
	\alpha(e_i \otimes x)^s
	\overset{\eqref{eq:alpha-explicit}}{=}	e_i \otimes x^s .
\end{equation}
The claim follows from 
\eqref{eq:alpha-def},
using that
\eqref{eq:pre-alpha} is $\GLfield$-equivariant
and that $\iota_\pi$ is a morphism
of diagrams. 

Consider
the $\GLring$-subrepresentation
$\left\langle \GLring \cdot \Pi x \right\rangle $
of $R(\sigma) $
generated by $\Pi x$,
and consider its image
\begin{equation}
	\label{eq:temp-representation}
\alpha
\big(\F e_i \otimes_{\F}  \left\langle \GLring \cdot \Pi x \right\rangle
\big)
\overset{\eqref{eq:Pi-and-alpha}}{=}
\F e_i \otimes_{\F}  \left\langle \GLring \cdot x^s \right\rangle
\subseteq \V \otimes_{\F} \DZero. 
\end{equation}
We claim that \eqref{eq:temp-representation}
is the subrepresentation 
$\F e_i \otimes_{\F}  I(\repr, \sigma^{[s]}) $
of $\V \otimes_{\F}  \DZero$
(for the definition of $\sigma^{[s]}$
see \Cref{def:s-at-D0-level},
whereas 
$I(\overline{\rho}, \sigma^{[s]})$
is defined in \Cref{deflem:I}).
By construction,
$x \in \sigma^{I_1}$ 
implies $x^{s} \in (\sigma^{[s]})^{I_1}$,
and consequently $x^{s}$ belongs to
$ I(\repr, \sigma^{[s]})$
and is not mapped to zero in the cosocle
$\sigma^{[s]}=\cosocR I(\repr, \sigma^{[s]})$.
We deduce that we have an inclusion
$\F e_i \otimes_{\F}  \left\langle \GLring \cdot  x^{s} \right\rangle \subseteq 
 \F e_i \otimes_{\F}  I(\repr, \sigma^{[s]})$,
which is a surjection on the cosocles,
hence an equality.
This establishes an isomorphism
\begin{equation}
	\label{eq:convoluted-iso}
\alpha \colon \F e_i \otimes_{\F}  \left\langle \GLring \cdot \Pi x \right\rangle
\xrightarrow{\sim}
\F e_i \otimes_{\F}  I(\repr, \sigma^{[s]}).
\end{equation}

\paragraph{Step 2.}
We prove that $\iota_\pi({\pi'}^{K_1})$
contains $V' \otimes_{\F}  \Dx[\delta(\sigma)]$.
Let $Q_i \defeq \alpha(\F e_i \otimes_{\F}  R(\sigma)) $.
We want to apply
\cite[Lemma~19.5]{BP12}
to conclude that $Q_i $ is isomorphic,
as a $\Gamma$-representation,
to the quotient 
$Q(\repr, \sigma^{[s]} ) $
of $R(\sigma) $
defined in \emph{loc.\!\ cit.}
We have seen
in \eqref{eq:convoluted-iso} that
(ii) of \emph{loc.\ cit.} is satisfied.
As for (i) of \emph{loc.\ cit.}, we use that
$Q_i$ is also a subrepresentation of 
$\V \otimes_{\F}  \DZero$,
which has $\GLring$-socle
$\V \otimes_{\F}  \bigoplus_{\sigma\in W(\repr)} \sigma$,
together with the fact that
$R(\sigma)$ is multiplicity free
(cf.\ \cite[Lemma~17.11]{BP12}).
We conclude that
$Q_i \cong Q(\repr, \sigma^{[s]} ) $,
in particular $Q_i$
contains a subrepresentation
$Q'_i$ isomorphic to $\Dx[\delta(\sigma)]$,
by \cite[Lemma~19.7]{BP12}.
Note that 
\[
	\socR Q'_i = \F e_i \otimes_{\F} \delta(\sigma) \subseteq \F e_i \otimes_{\F}  \Dx[\delta(\sigma)],
\]
cf.\ \eqref{eq:convoluted-iso}.
But there is only one subrepresentation of 
$\V \otimes_{\F}  \Dx[\delta(\sigma)]$
isomorphic to $\Dx[\delta(\sigma)]$
and with socle $\F e_i \otimes_{\F}  \delta(\sigma)$,
namely 
$\F e_i \otimes_{\F}  \Dx[\delta(\sigma)]$.

Since $Q'_i \subseteq \im \alpha \subseteq \iota_\pi({\pi'}^{K_1})$,
we have shown that, for each $1\le i\le r'$,
$\iota_\pi({\pi'}^{K_1})$
contains
$\F e_i \otimes_{\F}  \Dx[\delta(\sigma)]$,
hence it contains
$V' \otimes_{\F}   \Dx[\delta(\sigma)]$,
which concludes the proof.
\end{proof}

Now, we prove a generalisation of
 \cite[Lemma~3.3.5.2]{BHHMS2},
to the case $r \ge 1 $.
\begin{lemma}
\label{lemma:main}
Let $\pi $ be a admissible smooth representation of 
$\GLfield $ over $\mathbb{F}  $ satisfying 
hypothesis \ref{hypothesis:i},
and let $\pi' $ be a subrepresentation of $\pi $.
\begin{enumerate}[(i)]
\item 
	\label{point:main-1}
	The multiplicity $\rwt=[\socR(\pi'):\sigma]$
only depends on the length $\ell(\sigma)$ of $\sigma$.
Moreover, if $\repr $ is irreducible,
$\rwt$ only depends on $\pi' $.

\item 
	\label{point:main-2}
If $\sigma,\sigma'\in W(\repr) $
have the same length $ \ell $, then 
$\Vwt[\pi', \sigma]= \Vwt[\pi', \sigma'] $
inside $\V$.
Moreover, if $\repr $ is irreducible,
then for any choice of $\sigma,\sigma'\in W(\repr) $
we have
$\Vwt[\pi', \sigma]=\Vwt[\pi', \sigma']$.
\end{enumerate}
\end{lemma}
\begin{remark}
\label{rem:length-uniformity}
In the case where $\repr $ is irreducible,
we should observe that 
the length of Serre weights,
while being defined, 
doesn't play a key role.
Despite this, speaking of length even in the irreducible case has the advantage of making 
later arguments uniform in the two cases.
\end{remark}
\begin{proof}
We begin with 
the case $\repr $ reducible split.
Fix $\ell \in \{0, \dots, f\}$,
and fix $\sigma_1\in W(\overline\rho)$ a Serre weight of length $ \ell $,
such that its multiplicity
\[
r' \defeq \rwt[\pi',\sigma_1]=[\socR(\pi'):\sigma_1]
\]
is maximal
among Serre weights of length $ \ell$.
If $r'=0$, then \ref{point:main-1}
and \ref{point:main-2} are clear,
so assume that $r' >0$.

Let $(e_i)_{i=1} ^r$ be the
standard basis of $\V = \F^r$.
As in the proof of \Cref{lemma:fill-up}, 
we may choose $\iota_\pi $ 
so that $V' \defeq \Vwt[\pi',\sigma_1]$
is the subspace 
$\left\langle e_1,\dots, e_{r'}  \right\rangle_{\F} $
spanned by $e_1, \dots, e_{r'} $.

For each $1 \le j \le r'$
we consider the pullback 
\begin{equation}
\begin{tikzcd}
	\label{eq:pullback-reducible}
    \dZero\ar[ddrr, phantom,  "\square"]\ar[rr, hook]\ar[dd, hook] &&
\F e_{j} \otimes_{\F}  \DZeroEll[\ell] \ar[d, hook]\\
{}&& \F e_{j} \otimes_{\F} \DZero\ar[d, hook] \\
\pi^{\prime K_1}\ar[r, phantom, "\subseteq "] & \pi^{K_1}\ar[r, "\sim"',"\iota_\pi "] &
V \otimes_{\F}  \DZero
\end{tikzcd}
\end{equation}
of $\GLring $-representations
(which is set-theoretically simply an intersection).
As $r' > 0$,
$\dZero $ contains one copy of $\sigma_1$ in its socle;
in particular, $\dZero$ is nonzero.
Moreover, if we let 
$\dOne \defeq (\dZero)^{I_1} $,
then it clearly coincides with the intersection $\pi^{\prime I_1}\cap  (\F e_j \otimes_{\F}  \DOneEll) $,
and thus
 $\diagr \defeq  (\dOne \hookrightarrow \dZero) $ defines
a diagram.

We now claim that
$\diagr$ is a direct summand of $\DEll$, and 
we begin by looking at
$\dZero $.
Remember that we have the decompositions
\begin{align*}
	\DZeroEll[\ell]&\overset{\eqref{eq:ell-diagr-def}}{=}\bigoplus_{\ell(\sigma)=\ell} \Dx.
\end{align*}

As a consequence of 
\Cref{lemma:fill-up}, we have
\[
	\dZero \cong \bigoplus_{\sigma \subseteq \soc \dZero} \Dx[\sigma],
\]
and this gives the direct sum decomposition 
of representations
\begin{equation}
	\label{eq:ugly-decomposition}
	\DZeroEll[\ell] \cong 
\dZero \oplus D''_0,
\end{equation}
where
\[
	D''_0 \defeq 
\bigoplus_{\substack{\sigma \not \subseteq  \soc \dZero \\ \ell(\sigma)= \ell}} \Dx[\sigma].
\]

Since $\diagr$ is a diagram we know that $\dOne$
is $\Pi$-stable, and we claim that 
\[
	D''_1 \defeq (D''_0)^{I_1}=\bigoplus_{\substack{\sigma\not \subseteq  \soc \dZero \\ \ell(\sigma)= \ell}} \DOnex[\sigma]
\]
is also $\Pi $-stable,
so that we have a decomposition of diagrams.

To see this, 
remember that weight cycling on $\Diagr$
is such that, for each $H$-eigenvector
$v^{\chi}\in D''_1$
of eigencharacter $\chi:H\to \F^\times $,
there is a unique $\sigma\in W(\repr)$
(of length $\ell$,
since $\repr$ is reducible)
such that $\Pi v^{\chi}\in \DOnex $.
But
$\sigma $ cannot be in $\soc D_0^{\prime(j)}$,
since $D_0^{\prime (j)}$
is itself $\Pi$-stable.

Then, according to \cite[Theorem~15.4(ii)]{BP12},  
we conclude that the diagram $\diagr$,
defined after \eqref{eq:pullback-reducible},
is isomorphic to $\DEll$.

In particular,
for a fixed $\sigma\in W(\repr)$
of length $\ell$, and
for every $1\le j\le r'$,
$\iota_\pi({\pi'}^{K_1})$
contains $\F e_j \otimes_{\F} \sigma$ in its socle, 
hence
\begin{equation}
	\label{eq:aligning-the-V}
V' \otimes_{\F}  \sigma \subseteq 
\Vwt \otimes_{\F}  \sigma 
\end{equation}
and so the multiplicity 
$[\socR (\pi'^{K_1}):\sigma] $
is at least $r' $.
By the maximality of $\sigma_1 $, 
\eqref{eq:aligning-the-V} is an equality,
which proves
\ref{point:main-1} and \ref{point:main-2}
of the statement at the same time.

In the case $\repr $ irreducible
the argument is very similar:
for each $1\le j\le r'$, 
the pullback 
\begin{equation}
	\label{eq:pullback-irreducible}
\begin{tikzcd}
\dZero\ar[drr, phantom,  "\square"]\ar[rr, hook]\ar[d, hook]& &
\F e_j \otimes_{\F} \DZero\ar[d, hook], \\
\pi'^{K_1}\ar[r, phantom, "\subseteq "] & \pi^{K_1}\ar[r, "\sim"',"\iota_\pi "] &V \otimes_{\F}  \DZero 
\end{tikzcd}
\end{equation}
admits a direct summand
$\bigoplus_{\sigma\not \subseteq  \soc \dZero} D_{0,\sigma}$,
and the direct sum decomposition 
\[
\DZero =
\dZero \oplus \bigoplus_{\sigma\not \subseteq  \soc \dZero}\Dx[\sigma] 
\]
gives a decomposition of diagrams as well.
According to \cite[Theorem~15.4(i)]{BP12},  
we conclude that 
$\diagr\cong \Diagr$.

Hence, for $\repr$ irreducible
\eqref{eq:aligning-the-V}
holds without restrictions on $\sigma\in W(\repr)$,
from which 
\ref{point:main-1} and \ref{point:main-2}
follow.
\end{proof}
Using \Cref{lemma:main},
we set $\rell \defeq [\socR(\pi'):\sigma]$ and
$\Vell \defeq  \Vwt$,
where  $\sigma\in W(\overline \rho)$ is any Serre weight of length $\ell$.
Moreover, if $\repr $ is irreducible, then we set 
$\rnull \defeq [\socR(\pi'):\sigma]$
and $\Vnull \defeq \Vwt$.
In the spirit of \Cref{rem:length-uniformity},
we will sometimes use the first notation even in the irreducible case.

\bigskip
The following lemma will allow us 
to make ``linewise'' 
(in the sense of 
\Cref{lemma:line-intersection}(i))
arguments rigorous.
\begin{lemma}
	\label{lemma:line-intersection}
Let $S$ be a (not necessarily commutative)
$\F$-algebra,
$D$ an $S$-module that admits a
(finite) composition series,
$W$ a finite-dimensional $\F$-vector space.
Suppose that the composition factors of $D$ 
are pairwise non-isomorphic
(i.e.\ $D$ is multiplicity free),
and denote by $\JH(D)$ the set of composition
factors of $D$.
Suppose moreover that 
for every $\tau \in \JH(D)$
we have $\End_S(\tau) = \F$.
\begin{enumerate}[(i)]
\item 
Let $M_1,M_2 \subseteq W \otimes_{\F}D$
be two $S$-submodules.
If $M_1 \cap (L \otimes_{\F}D)=
M_2 \cap (L \otimes_{\F}D)$
inside $W \otimes_{\F}D$
for every line $L \subseteq W$,
then $M_1=M_2$.
\item 
Assume that $D$ decomposes into a direct sum 
$D = \bigoplus_{\sigma \in \soc_S(D)}D_\sigma$,
for some $S$-submodules $D_\sigma \subseteq D$
with socle $\soc_S(D_\sigma)= \sigma$.

Let $M \subseteq W \otimes_{\F}D$
be an $S$-submodule.
Then, there exists a unique direct sum decomposition 
$M = \bigoplus_{\sigma \in \soc_S(D)}M_\sigma$
such that $M_\sigma \subseteq W \otimes_{\F} D_\sigma$.
Moreover, there exist unique
$\F$-vector subspaces
$W_\sigma \subseteq W$ for
$\sigma \in \soc_S(D)$
such that 
$\soc_S(M_\sigma)=W_\sigma \otimes_{\F}\sigma$,
and there is an inclusion
$M_\sigma \subseteq
W_\sigma \otimes_{\F}D_\sigma$.
\end{enumerate}
\end{lemma}
\begin{proof}
(i)
Up to replacing $M_1$ with $M_1 \cap M_2$
we can assume without loss of generality 
that $M_1 \subseteq M_2$.
Suppose for the sake of contradiction
that this inclusion is strict.
If $(\soc_i(D))_{i \ge 0}$ denotes
the $S$-socle filtration on $D$,
we let $d \ge 0$ be the largest integer
such that $M_0 \defeq M_1 \cap 
(W \otimes_{\F}\soc_d(D))=
M_2 \cap (W \otimes_{\F}\soc_d(D))$
(which exists by the assumption
$M_1 \subsetneq M_2$, 
and since $D$ admits a composition series).

When we quotient out $W \otimes_{\F} D$
by $W \otimes_{\F} \soc_d(D)$,
we see that the inclusion
$M_1 \subseteq M_2$
inside $W \otimes_{\F} D$
is sent to an injection
$M_1/M_0 \hookrightarrow  M_2/M_0$
inside $W \otimes_{\F} (D/\soc_d(D))$.
Hence, we can assume that $d=0$,
i.e.\ that $\soc_S(M_1) \subsetneq \soc_S(M_2)
\subseteq W \otimes_{\F}\soc_S(D)$.
In particular, there exists a simple $S$-module
$\tau$ and an $S$-module injection
$\tau \hookrightarrow M_2$
that does not factor through $M_1$.

Since $D$ is multiplicity free,
the image of the composition
$\tau \hookrightarrow M_2 \subseteq W \otimes_{\F} D$
is $L \otimes_{\F} \tau$ for a unique
line $L \subseteq W$
(since $\End_S(\tau)=\F$,
by an abuse of notation we have identified $\tau$
with the image of the unique nonzero 
$S$-module homomorphism 
$\tau \hookrightarrow D$,
up to nonzero scalars).
In particular,
$\tau \hookrightarrow M_2 \subseteq W \otimes_{\F} D$
factors through $M_2 \cap (L \otimes_{\F}D)$.
However, $M_1 \cap (L \otimes_{\F}D)
=M_2 \cap (L \otimes_{\F}D)$
by assumption, so $\tau \hookrightarrow M_2$
factors through $M_1$, contradiction.

(ii)
For each $\sigma \in \soc_S(D)$,
the condition
$M_\sigma \subseteq W \otimes_{\F} D_\sigma$
implies
$M_\sigma \subseteq  M \cap (W \otimes_{\F}D_\sigma)$,
while the condition
$M = \sum_{\sigma \in \soc_S(D)} M_\sigma$
forces the equality
$M_\sigma = M \cap (W \otimes_{\F}D_\sigma) $ 
to hold.
Therefore, we set $M_\sigma \defeq 
M \cap (W \otimes_{\F}D_\sigma)$,
and we prove the direct sum decomposition
of the statement. 

We see that the sum $\sum_{\sigma \in \soc_S(D)} M_\sigma$
is direct, because the natural map 
$\bigoplus_{\sigma \in \soc_S(D)} M_\sigma
\twoheadrightarrow
\sum_{\sigma \in \soc_S(D)} M_\sigma
\subseteq W \otimes_{\F}D$
is injective on socles.
It only remains to show that the inclusion
$\bigoplus_{\sigma \in \soc_S(D)} M_\sigma \subseteq M$
is an equality.
By applying (i) to 
$M_1 = \bigoplus_{\sigma \in \soc_S(D)} M_\sigma$,
$M_2=M$, 
and noticing that 
$M_1 \cap (L \otimes_{\F} D)
= \bigoplus_{\sigma \in \soc_S(D)} (M \cap (L \otimes_{\F}D_\sigma))$
for every line $L \subseteq W$,
we reduce to the case $W = \F$.
In this case, since is $D$ multiplicity free
it is enough to show that every
composition factor 
$\tau \in \JH(M)$ of $M$
also occurs in 
$M_1 = \bigoplus_{\sigma \in \soc_S(D)}
M_\sigma $.
There exists a unique $\sigma \in \soc_S(D)$
such that $\tau \in \JH(D_\sigma)$,
and so $\tau \in \JH(M \cap D_\sigma)$.
However, $M \cap D_\sigma = M_\sigma
\subseteq \bigoplus_{\sigma \in \soc_S(D)}
M_\sigma =M_1$,
so $\tau \in \JH(M_1)$.
This shows the first part. 

For the second part,
fix $\sigma \in \soc_S(D)$,
and observe that
$\soc_S(M_\sigma)$ is an $S$-subrepresentation
of $W \otimes_{\F}\soc_S(D_\sigma)=
W \otimes_{\F} \sigma$,
and so it is of the form $W_\sigma \otimes_{\F}\sigma$ for a unique
$\F$-vector subspace $W_\sigma \subseteq W$.
We want to apply (i) to
$D = D_\sigma$, $M_1= M_\sigma \cap (W_\sigma 
\otimes_{\F}D_\sigma) ,M_2=M_\sigma.$
For a line $L \subseteq W$,
we distinguish two cases:
if $L \subseteq W_\sigma$ then
$M_1 \cap (L \otimes_{\F}D_\sigma)$ and
$M_2 \cap (L \otimes_{\F}D_\sigma)$ are both equal
to $M_\sigma \cap (L \otimes_{\F}D_\sigma)$;
if $L \cap W_\sigma=0$ 
then we need to show that
$M_\sigma \cap (L \otimes_{\F}D_\sigma)=0$,
and we do this by showing that its socle
is zero.
Indeed, one computes
\begin{align}
	\label{eq:soc-line-formal-computation}
\soc_S(M_\sigma \cap (L \otimes_{\F} 
D_\sigma)) \subseteq&
\soc_S(M_\sigma)\cap
\soc_S(L \otimes_{\F} D_\sigma) 
\\ \nonumber =&
(W_\sigma \otimes_{\F} \sigma) \cap
(L\otimes_{\F}\sigma)=
(W_\sigma \cap L) \otimes_{\F} \sigma= 0.
\qedhere
\end{align}
\end{proof}

We conclude the section with the following
corollary.
\begin{corollary}
\label{cor:direction}
Keep the assumptions of \Cref{lemma:main}.

If $\repr$ is reducible split, 
the isomorphism \eqref{eq:1} 
restricts to an isomorphism 
\begin{equation}
	\label{eq:best-restriction-red}
	\iota_ \pi \colon
(\pi^{\prime I_1} \hookrightarrow
\pi^{\prime K_1})
\xrightarrow{\sim} \displaystyle
\bigoplus_{\ell=0}^f 
\left( \Vell \otimes_{\F}\DEll \right)
\end{equation}
of diagrams.

If $\repr $ is irreducible,
the isomorphism \eqref{eq:1} 
restricts to an isomorphism 
\begin{equation}
	\label{eq:best-restriction-irr}
\iota_ \pi \colon
(\pi^{\prime I_1} \hookrightarrow \pi^{\prime K_1})
\xrightarrow{\sim} \Vnull\otimes_{\F} \Diagr.
\end{equation}
\end{corollary}
\begin{proof}
We apply \Cref{lemma:line-intersection}(ii)
with $S= \F[\Gamma], D=\DZero, W=\V,
M= \iota_\pi(\pi^{\prime K_1})$.
For $\sigma \in W(\repr)$,
notice that the 
$\F$-vector subspace $W_\sigma$ of
\emph{loc.\ cit.}\ coincides with
$\Vwt[\pi',\sigma]$
by construction of
$\Vwt[\pi',\sigma]$, cf.\ \eqref{eq:V-def},
hence the $\Gamma$-subrepresentation
$M_\sigma \subseteq 
\iota_\pi(\pi^{\prime K_1})$ of 
\emph{loc.\ cit.}\ is contained inside
$\bigoplus _{\sigma \in W(\repr)}
\Vwt[\pi',\sigma] \otimes_{\F}\Dx[\sigma]$.
By the proof of
\emph{loc.\ cit.}\ we have
$M_\sigma= \iota_\pi(\pi^{\prime K_1})
\cap \V \otimes_{\F} \Dx[\sigma]$,
and so $M_\sigma$ contains
$\Vwt[\pi',\sigma] \otimes_{\F} \Dx[\sigma]$
by \Cref{lemma:fill-up}.
Hence, the equality
$M_\sigma =\Vwt[\pi',\sigma] \otimes_{\F} \Dx[\sigma]$
must hold, and we have the direct sum
decomposition
\begin{equation}
	\label{eq:almost-best}
\iota_\pi(\pi^{\prime K_1})=
\bigoplus_{\sigma \in W(\repr)} M_\sigma 
=\bigoplus_{\sigma \in W(\repr)}
\left( \Vwt[\pi',\sigma] \otimes_{\F}
\Dx[\sigma] \right).
\end{equation}

Suppose that $\repr$ is reducible split.
By \Cref{lemma:main}(ii), 
the right-hand side of \eqref{eq:almost-best}
is equal to
$\bigoplus_{\ell=0}^f
\left( \Vell[\pi',\ell] \otimes_{\F}
\DZeroEll[\ell]\right)$,
and taking the $I_1$-invariants of 
\eqref{eq:almost-best} we find 
$ \iota_\pi(\pi^{\prime I_1})=
\iota_\pi(\pi^{\prime K_1})^{I_1}=
\bigoplus_{\ell=0}^f
\left( \Vell[\pi',\ell] \otimes_{\F}
\DOneEll[\ell]\right).$
Hence, $\iota_\pi$ maps
the diagram
$(\pi^{\prime I_1} \hookrightarrow
\pi^{\prime K_1})$
to the subdiagram
$\bigoplus_{\ell=0}^f
\left( \Vell[\pi',\ell] \otimes_{\F}
\DEll[\ell]\right)$
of $V \otimes_{\F} \Diagr$,
showing \eqref{eq:best-restriction-red}.

The case $\repr$ is irreducible is
analogous:
by \Cref{lemma:main}(ii), 
the right-hand side of \eqref{eq:almost-best}
becomes
$\Vnull[\pi'] \otimes_{\F}\DZero$,
and taking $I_1$-invariants we find 
$\iota_\pi(\pi^{\prime I_1})=
\Vnull[\pi']\otimes_{\F} \DOne$,
hence $\iota_\pi$ maps
the diagram
$(\pi^{\prime I_1} \hookrightarrow
\pi^{\prime K_1})$
to the subdiagram
$ \Vnull[\pi'] \otimes_{\F}\Diagr,$
of $ \V \otimes_{\F}\Diagr$,
showing \eqref{eq:best-restriction-irr}.
\end{proof}

\subsection{Generation by the socle}
	\label{sec:generation} 
The goal of this section is to 
prove that a $\GLfield$-representation $\pi$ 
satisfying hypotheses \ref{hypothesis:i}
to \ref{hypothesis:iii} of
\Cref{sec:hypotheses}
is generated by its socle
(this is \Cref{prop:generation})
and that, when $\repr$ is reducible split,
we have the decomposition
\eqref{eq:pi-red-decomposition}
(this is \Cref{thm:principal-series}).

We start by proving a lemma which generalises
\cite[Corollary~3.3.2.2]{BHHMS2}
to the case $r \ge 1 $. 
For $0\le \ell\le f$,
remember the subset $\mathscr{P}_\ell$
of $\mathscr{P}$
defined in \Cref{def:J-and-ell}.
Then,
we define the graded $\grL$-module $\Nell$ with $H$-action
as 
\begin{equation}
	\label{eq:N-ell}	
	\Nell \defeq \bigoplus_{\lambda\in \mathscr{P}_{\ell}} \chi_ \lambda^{-1} \otimes R/\mathfrak{a}(\lambda),
\end{equation}
where $\mathfrak{a} (\lambda) $ are certain ideals of $R $ indexed by
$\lambda\in \mathscr{P}  $
which are defined in 
\cite[Definition~3.3.1.1]{BHHMS2}.
Finally, we define
\begin{equation}
	\label{eq:N-null}	
	\Nnull \defeq \bigoplus_{\ell=0} ^f \Nell.
\end{equation}

\begin{lemma}
\label{lemma:gr}
Let $\pi $ be a admissible smooth representation of $\GLfield $ over $\mathbb{F}  $ such that
hypotheses \ref{hypothesis:i} and \ref{hypothesis:ii} hold.
Let $\pi_1 \subseteq \pi_2$ be two subrepresentations of $\pi$,
and let
\begin{align}
	\label{eq:Ni-def}
N_i \defeq \bigoplus_{\ell=0}^f
\left( \Vell [\pi_i,\ell]^\vee \otimes_{\F} \Nell \right),
\end{align}
for $i=1,2$.

\begin{enumerate}[(i)]
\item
There exist surjections
\begin{align}
\label{eq:surjection_Ni}
 \theta_i \colon
N_i &\twoheadrightarrow \grm (\pi_i^{\vee}) 
\end{align}
of $R$-modules with compatible $H$-action,
for $i=1,2$,
such that the square
\begin{equation}
	\label{diag:thetas-square}
\begin{tikzcd}
\gr_{\mI} ( \pi_2 ^\vee ) \ar[r, two heads, "\psi"] &
\gr_{\mI} ( \pi_1 ^{\vee} )  \\
N_2\ar[u, two heads, " \theta _2"] \ar[r, two heads]  & 
N_1
 \ar[u, two heads, " \theta_1"'] 
\end{tikzcd}
\end{equation}
commutes.
Here, the morphism on the bottom
is the direct sum of
the restriction maps
$ \Vell[\pi_2,\ell]^\vee \otimes_{\F} \Nell 
\twoheadrightarrow 
\Vell[\pi_1,\ell]^\vee \otimes_{\F}  \Nell $,
and $\psi$ is given by the graded pieces of
the restriction map
$ \pi_2^{\vee} \twoheadrightarrow \pi^{\vee}_1$.

\item
Let $\pi' \defeq \pi_2/\pi_1$
and endow $\pi^{\prime \vee }$
with the submodule filtration $F$
induced from the $\mI$-adic filtration 
on $\pi_2 ^\vee$.

If $\repr$ is reducible split,
for each $0 \le \ell \le f$
pick a complementary subspace
$V'(\ell)$ of $\Vell[\pi_1,\ell]$ in $\Vell[\pi_2,\ell]$,
and set 
\begin{equation}
	\label{eq:N-prime-def-red}
N' \defeq 
\bigoplus_{\ell=0}^f
\left( V'(\ell) ^\vee \otimes_{\F} \Nell \right).
\end{equation}
If $\repr$ is irreducible,
pick a complementary subspace
$V'$ of $\Vnull[\pi_1]$ in $\Vnull[\pi_2]$,
and set 
\begin{equation}
	\label{eq:N-prime-def-irred}
N' \defeq 
 V^{\prime\vee}\otimes_{\F} \Nnull .
\end{equation}
Then, \eqref{diag:thetas-square} fits into
a commutative diagram 
\begin{equation}
	\label{diagr:thetas-ses}
\begin{tikzcd}
	0 \ar[r] & \gr_F(\pi^{\prime\vee}) \ar[r] &  
\gr_{\mI} ( \pi_2 ^\vee ) \ar[r, "\psi"] &
\gr_{\mI} ( \pi_1 ^{\vee} ) \ar[r] & 0 \\
0 \ar[r] & 
N' \ar[u, "\theta'"] 
\ar[r] & 
N_2 \ar[u, two heads, " \theta_2 "] \ar[r]   & 
N_1
 \ar[u, two heads, " \theta_1"] 
\ar[r] & 0
\end{tikzcd}
\end{equation}
\end{enumerate}
\end{lemma}
Note that hypothesis \ref{hypothesis:ii}
is needed in order to say that $\gr_{\mI} (\pi^\vee )$
is an $R$-module.
\begin{proof}
(i)
We adapt the  proof of \cite[Corollary~3.3.2.2]{BHHMS2}.
To make the arguments uniform, 
even when $\repr$ is irreducible
we let $V'(\ell) \defeq V'$,
for $0 \le \ell \le f$.
For brevity, 
let
$V_i(\ell)\defeq \Vell[\pi_i,\ell] $,
$r_i(\ell)\defeq \rell[\pi_i,\ell] $,
for $i=1,2$, $0\le{\ell}\le f$.

\paragraph{Step 1.}
We construct a basis $\mathcal{B}_i$
of $(\pi_i^{I_1})^\vee $
(as an $\F$-vector space),
for $i=1,2$.

For each $0\le{\ell}\le f$,
we pick a basis $(v_i(\ell))_{0\le i\le r_2(\ell)}$
of $V_2(\ell)$ adapted to the direct sum decomposition
$V_2(\ell)= V_1(\ell) \oplus V'(\ell)$, 
i.e.\ such that 
$(v_i(\ell))_{0\le i \le r_1(\ell)}$
is a basis of $V_1(\ell)$,
and such that $(v_i(\ell))_{r_1(\ell)+1 \le i \le r_2(\ell)}$
is a basis of $V'(\ell)$.
Let $(e_i(\ell))_{0 \le i \le r_2(\ell)} $
be the dual basis of 
$(v_i(\ell))_{0 \le i \le r_2(\ell)} $
in $V_2(\ell) ^\vee $.

Moreover, for each $\lambda \in \mathscr{P}$,
we pick a nonzero eigenvector
$v^\lambda \in \DOne^{\chi_\lambda}$
of eigencharacter $\chi_\lambda$,
so that $(v^\lambda)_{\lambda\in \mathscr{P}} $
is a basis of $\DOne$ (since it is multiplicity free),
and we let $(e^\lambda)_{\lambda\in \mathscr{P}} $
be its dual basis in $\DOne ^\vee $.
If $\lambda\in \mathscr{P}$ has length
$\ell(\lambda)$,
we let
\begin{equation}
	\label{eq:boring-basis}
e_i^\lambda \defeq \iota_\pi^\vee (e_i( \ell(\lambda))\otimes e^\lambda )\in (\pi_2^{I_1})^\vee,
\end{equation}
where
$\iota_\pi^\vee \colon
\bigoplus_{\ell=0}^f \left( V_2(\ell) ^\vee \otimes_{\F}  
\DEll
\right) \to (\pi_2^{I_1}) ^\vee$
is the $\F$-linear dual of 
\eqref{eq:best-restriction-red}.
By construction,
\[
	 \mathcal{B}_2 \defeq \left( e_i^\lambda\mid {\lambda\in \mathscr{P}},{1\le i\le r_2(\ell(\lambda))} \right)  
\]
is a basis of 
$(\pi_2^{I_1})^\vee $,
and since $(v_i(\ell))_{0\le i\le r_2(\ell)}$
was adapted to $V(\ell)= V_1(\ell)\oplus V'(\ell)$
in the sense of the previous paragraph,
we also know that
\[
	\mathcal{B}_1 \defeq 	\bigl({e_i^\lambda}|_{\pi_1^{I_1}}\mid {\lambda\in \mathscr{P}},{1\le i \le r_1(\ell(\lambda))} \bigr)
\]
is a basis of
$(\pi_1^{I_1})^\vee $.

\paragraph{Step 2.}
We define the surjections
\eqref{eq:surjection_Ni}.
We follow \cite[Theorem~3.3.2.1]{BP12}.
Start with a homomorphism 
\[
\Theta_2	\colon \bigoplus_{\ell=0} ^f\left(	V_2(\ell )^\vee \otimes_{\F}  \bigoplus_{\lambda\in \mathscr{P}_\ell} R \right) \to \grm( \pi_2^\vee)
\]
of $R $-modules, defined as follows:
for each $0\le \ell\le f$,
$ \bigoplus_{\lambda\in \mathscr{P}_\ell}R $
comes with the standard $R $-module basis,
which we call
$(\varepsilon_\lambda)_{\lambda\in \mathscr{P}_\ell } $,
and we define $\Theta_2$ on a basis by
\[
\Theta_2\colon e_i(\ell ) \otimes\varepsilon_\lambda\mapsto
e^\lambda_i 
\overset{\eqref{eq:boring-basis}}{\in}
(\pi_2^{I_1})^\vee =\grm(\pi_2)_0 \subseteq 
\grm(\pi_2),
\]
for $0\le \ell\le f$, $\lambda\in \mathscr{P}_\ell$.
Note that we are using that $\grm(\pi_2^\vee )$
is an $R$-module, which is a consequence
of hypothesis \ref{hypothesis:ii},
cf.\ the discussion in the paragraph
preceding \Cref{sec:structure-pi}.

Next, we make $\Theta_2$ into an $H $-equivariant homomorphism by tensoring with $\chi_\lambda ^{-1} $ as follows:
\[
 \widetilde{\Theta}_2\colon \bigoplus_{\ell =0} ^f\left(	V_2(\ell)^\vee \otimes_{\F}  \bigoplus_{\lambda\in \mathscr{P}_\ell} 
	(\chi_\lambda ^{-1} \otimes R)
 \right) \to \grm( \pi_2^\vee).
\]
As $e^\lambda $ is annihilated by 
$\mathfrak{a} (\lambda) $
for each $1 \le i \le r$, $ \lambda \in \mathscr{P}$ (cf.\ the proof of \cite[Theorem~3.3.2.1]{BP12}),
we have a factorisation
\begin{equation}
	\label{diag:theta-actual-def}
	\begin{tikzcd}
\displaystyle\bigoplus_{\ell =0} ^f\left(	V_2(\ell )^\vee \otimes_{\F}  \displaystyle\bigoplus_{\lambda\in \mathscr{P}_\ell } 
	(\chi_\lambda ^{-1} \otimes R)
 \right) 
\ar[d, two heads]
\ar[r, "\tilde \Theta_2"]
& \grm(\pi_2^\vee)
\\
N_2=\displaystyle\bigoplus_{\ell =0} ^f\left(	V_2(\ell)^\vee \otimes_{\F}  \displaystyle\bigoplus_{\lambda\in \mathscr{P}_\ell} 
	(\chi_\lambda ^{-1} \otimes R/\mathfrak{a}(\lambda))
 \right) 
\ar[ur, dotted, bend right = 15, "\theta_2"']
	\end{tikzcd}
\end{equation}
which defines $\theta_2$.
By an abuse of notation, we will also be referring to the image of  
$ \varepsilon_\lambda$
inside 
$\bigoplus_{\lambda\in\mathscr P_\ell}\left(\chi_\lambda^{-1}\otimes R/\mathfrak a(\lambda)\right)$ 
as 
$ \varepsilon_\lambda$,
in particular we write
\begin{equation}
\label{eq:theta-def}
\theta_2\colon e_i(\ell) \otimes\varepsilon_\lambda\mapsto
e^\lambda_i.
\end{equation}
\paragraph{Step 3.}
\label{Step:3}
We prove that both $\psi $ and $\theta_2 $ are surjective.

Since $\grm(\pi_2^\vee ) $ is generated by $\grm(\pi_2^\vee )_0 $ as a $\grL $-module,
it is enough to show 
that the compositions 
\begin{align*}
	N_2
&
\xrightarrow{\theta_2} 
\grm(\pi_2^\vee) \twoheadrightarrow
 \grm(\pi_2^\vee)_0, \\
	\grm(\pi_2^\vee )& \xrightarrow{\psi} 
	\grm(\pi_1^{\vee}) \twoheadrightarrow \grm(\pi_1^{\vee})_0
\end{align*}
are surjective.

As for $\theta_2$, surjectivity at the level of $\gr_0 $ comes by construction, 
while for $\psi $ the argument is as follows:
as $\pi_1 \subseteq  \pi_2 $,
we obtain
$\pi_1^{I_1} \subseteq  \pi_2^{I_1}$,
and, taking duals, we obtain a surjection
\begin{align*}
\grm(\pi_2)_0=(\pi_2^{I_1})^\vee \twoheadrightarrow (\pi_1^{I_1})^\vee= \grm(\pi_1)_0.
\end{align*}
As this surjection is the grade $0$ part 
of $\psi $ by construction, 
this concludes the second step.
\paragraph{Step 4.}
We define the morphism 
$\theta_1: N_1 \to \grm(\pi_1^{\vee})$
of $R$-modules,
and show that it is surjective.

Consider the natural $\F$-linear maps
\[
\begin{tikzcd}
	V_1(\ell)
	\ar[r, hook, shift right=1.4, "i_1(\ell)"']&
	\ar[l, two heads, shift right=1.4, "p_1(\ell)"']  V_2(\ell) \ar[r, two heads, shift left=1.4, "p'(\ell)"] 
   &\ar[l, hook', shift left=1.4, "i'(\ell)"]   V'(\ell)
\end{tikzcd}
\]
given by the direct sum decomposition
$V_2(\ell)=V_1(\ell)\oplus V'(\ell)$.
They define,
by summing over $0\le{\ell}\le f$ and taking a dual,
the $R$-linear maps 
\[
\begin{tikzcd}
	N_1 \ar[r, hook, shift left=1.4, "p_1^{\vee} "]&
	\ar[l, two heads, shift left=1.4, "i_1^{\vee }"]
	N_2\ar[r, two heads, shift right=1.4, "i^{\prime\vee}"']
							   &
	N' \ar[l, hook', shift right=1.4, "{p} ^{\prime\vee}"'].
\end{tikzcd}
\]
They satisfy the following 
orthogonality relations
(we only write down the ones that will be relevant):
\begin{subequations}
\begin{align}
	\label{eq:direct-sum-relations(vanish)}
p'(\ell) \circ i_1(\ell) &= 0,
	\\
	\label{eq:direct-sum-relations(section1)}
	i_1^{\vee} \circ p_1^{\vee} &=\id_{N_1} , \\
	\label{eq:direct-sum-relations(section')}
	i^{\prime\vee} \circ p^{\prime\vee} &=\id_{N'} , \\
	\label{eq:direct-sum-relations(sum)}
p_1^{\vee}\circ i_1^{\vee}+
p^{\prime\vee}\circ i^{\prime\vee}&= \id_{N_2} .
\end{align}
\end{subequations}
Now, we define $\theta_1$ as the composition
\begin{equation}
	\begin{tikzcd}
 \theta_1 \colon\ 	N_1	
 \ar[r, hook, "p_1^{\vee}" ] 
	& N_2
	\ar[r, two heads, "\theta_2" ]
	& {\grm(\pi_2^\vee)} \ar[r, two heads, "\psi"]
	&{ \grm( \pi_1^{\vee}),}
\end{tikzcd}
\label{eq:big-composition}
\end{equation}
which we claim to be surjective.
As in step 2, it is enough to show surjectivity at the level of $\gr_0$.

Notice that, 
by \eqref{eq:direct-sum-relations(section1)},
$p_1^{\vee}$ sends
\begin{equation}
	\label{eq:retraction-equation}
p_1^{\vee} \colon e_i(\ell)|_{V_1(\ell)}  \otimes n
\mapsto e_i(\ell) \otimes n, 
\end{equation}
for $n \in \Nell$,
$1\le i\le r_1(\ell)$.

The situation is summarised as follows:
for $0\le {\ell}\le f$,
 $\lambda\in \mathscr{P}_\ell$,
and $1\le i\le r_1(\ell)$,
we have
\[
	\begin{tikzcd}[row sep = tiny, column sep = small]
	\theta_1\colon \ 	N_1
	\ar[r, hook, "p_1^{\vee}"] 
	& N_2
	\ar[r, two heads, "\theta_2" ]
	& {\grm(\pi_2^\vee)} \ar[r, two heads, "\psi"]
	&{ \grm( \pi_1^{\vee})}
	\\
e_i(\ell)|_{V_1(\ell)}\otimes \varepsilon_\lambda
\ar[r,maps to, "\eqref{eq:retraction-equation}"] 
	&
e_i(\ell)\otimes \varepsilon_\lambda
\ar[r, maps to, "\eqref{eq:theta-def}"]
	& e_i^\lambda 
\ar[r, maps to]
	& e_i^\lambda |_{\pi_1^{I_1}}
\in \mathcal{B}_1 .
\end{tikzcd}
\]
Since $\mathcal{B}_1$
is a basis of 
$(\pi_1^{I_1})^\vee= \grm(\pi_1)_0
\subseteq \grm(\pi_1)$,
this proves our claim.
\paragraph{Step 5.}
Finally, we show that \eqref{diag:thetas-square} commutes.

Using \eqref{eq:direct-sum-relations(sum)},
it is enough to 
show that $\psi \circ \theta_2 = \theta_1 \circ i_1^{\vee} $
holds
when precomposed with either
$p_1^{\vee} $ or ${p}^{\prime\vee }$.
In the first case, 
it follows from \eqref{eq:big-composition}
and \eqref{eq:direct-sum-relations(section1)},
while in the second case
our goal is to show that the composition
\begin{equation}
\label{eq:dumb-composition}
\begin{tikzcd}
&	\gr_{\mI} ( \pi_2 ^\vee ) \ar[r, two heads, "\psi"] &
	\gr_{\mI} ( \pi_1 ^\vee )\\
	N' \ar[r, hook, "p^{\prime\vee}"']  &
N_2 \ar[u, " \theta_2 ", two heads] 
\end{tikzcd}
\end{equation}
vanishes.
We can see this by observing that $N' $
is generated by its degree $0 $ part,
which is sent to zero in $\gr_{\mI} ( \pi _1 ^\vee ) $:
indeed, when we take the degree $0 $ part of \eqref{eq:dumb-composition}, we can identify the maps with
\begin{equation}
	\label{eq:crooked-composition}
	\begin{tikzcd}[column sep = 5ex, row sep= tiny]
&V_2(\ell)^\vee \otimes_{\F}  	\DOne 
		\ar[r, two heads,  "\bigoplus i_1(\ell)^\vee "{yshift=+3pt}, "\eqref{eq:best-restriction-red}"'] &
\displaystyle\bigoplus_{ \ell=0}^f V_1(\ell)^\vee \otimes_{\F}  
 \DOneEll
\\
\displaystyle\bigoplus_{ \ell=0}^f V'(\ell)^\vee \otimes_{\F}  \Nell_0 
\ar[r, hook, "\bigoplus p'(\ell)^\vee "'{yshift =-3pt}]  
&(N_2)_0 \ar[u,"\iota_\pi\circ (\theta_2)_0" ], 
\end{tikzcd}
\end{equation}
using $\iota_\pi$,
and the conclusion follows from 
\eqref{eq:direct-sum-relations(vanish)}.
In \eqref{eq:crooked-composition},
we have written $p'(\ell)^\vee $ instead
of  $p'(\ell)^\vee \otimes \id _{\Nell_0} $, and
$i_1(\ell)^\vee $ instead
of  $i_1(\ell)^\vee \otimes \id_{\DOneEll}$,
for simplicity.
This finishes the proof of (i).

(ii)
The vanishing of the composition in
\eqref{eq:dumb-composition}
implies that $\theta_2$ restricts to 
$\theta_2:N' \to \gr_F(\pi^{\prime\vee})$,
which is the desired map $\theta'$ 
fitting in the
commutative diagram
\begin{equation*}
\begin{tikzcd}
	0 \ar[r] & \gr_F( \pi^{\prime\vee}	) \ar[r] &	\gr_{\mI} ( \pi_2 ^\vee ) \ar[r] &
	\gr_{\mI} ( \pi_1^{\vee}) \ar[r] & 0 \\
	0 \ar[r] & N' \ar[u, dashed, "\theta'"]  \ar[r, "p^{\prime\vee}"'] &
	N_2 \ar[u, " \theta_2 "] \ar[r, "i_1^{\vee}"']  & 
	N_1  \ar[u, " \theta_1"] \ar[r] & 0
\end{tikzcd}
\end{equation*}
with exact rows.
For the splitting of the bottom diagram, 
we can use
\eqref{eq:direct-sum-relations(section1)}
to say that
 $p_1^\vee $ is a section of $i_1^\vee $,
 or
\eqref{eq:direct-sum-relations(section')}
to say that $i^{\prime \vee}$ is a retraction of 
$p^{\prime \vee}$.
\end{proof}

\begin{remark}
From the proof of \cite[Theorem~3.3.2.3]{BHHMS2},
together with \cite[Lemma~3.3.1.3(i)]{BHHMS2},
it follows that we have,
for $\lambda \in \mathscr{P}$,
\begin{align}
	m_{\mathfrak p_0}(R/\mathfrak a(\lambda))&=
\mathbbm{1}_{\mathscr{D}}(\lambda), 
    \label{eq:mult_id}
\end{align}
where $\mathbbm{1}_{\mathscr{D}}  $ denotes the indicator function of $\mathscr{D}  $, 
i.e.\ $ 	m_{\mathfrak p_0}(R/\mathfrak a(\lambda))=1 $ if $\lambda\in \mathscr{D}  $, and $0 $ otherwise.
\end{remark}

The following proposition generalises \cite[Proposition~3.3.5.3]{BHHMS2} to the case $r \ge 1 $.
\begin{prop}
\label{prop:main-memoire}
Assume that $\repr$ is 
$2f$-generic.
Let $\pi $ be an admissible smooth representation of $\GLfield $ over $\mathbb{F}  $ such that
hypotheses \ref{hypothesis:i} to \ref{hypothesis:iii} hold.
Let $\pi'$ be a subquotient of $\pi$.
    \begin{enumerate}[(i)]
	    \item \label{point:i}
	    We have $\dim_{\mathbb{F}(\!(X)\!)}D^\vee_\xi(\pi')=m_{\mathfrak p_0}(\grm(\pi'^\vee))$.
    \item \label{point:ii}
	    Assume that $\pi'$ is a subrepresentation of $\pi$.
	    Then,
    \[
	    \dim_{\mathbb{F}(\!(X)\!)}D^\vee_\xi(\pi')=m_{\mathfrak p_0}(\grm (\pi^{\prime\vee}))=\lg_{\GLring} \socR(\pi').
    \] 
    In particular, if $\pi'\neq 0$ then $D^\vee_\xi(\pi')\neq 0$.

	\item \label{point:iii}
Assume that $\pi'\subsetneq \pi$
 is a subrepresentation, and set
$0 \neq \pi'' \defeq \pi/\pi'$.
 Then, $D^\vee_\xi(\pi'')\neq 0$.
    \end{enumerate}
\end{prop}
We will use the following notation
in the proof:
if $\pi $ is an admissible smooth representation of $\GLfield $ over $\mathbb{F}  $ such that
hypothesis \ref{hypothesis:iii} holds,
we fix once and for all
a $\GL_2(K)$-equivariant isomorphism 
\begin{equation}
	\label{eq:kappa-pi-def}
\kappa_\pi \colon \pi \xrightarrow{\sim}
\E(\pi^\vee)^\vee \otimes (\det(\overline{\rho})\omega^{-1})
\end{equation}	
of $\Lambda$-modules using
\cite[Corollary~1.8]{Koh17}.
Let $\pi'$ be a subrepresentation of $\pi$,
and set $\pi'' \defeq \pi/\pi'$.
Since $\Lambda$ is Auslander regular 
(cf.\ the discussion before
\Cref{def:multiplicity})
and $\pi ^\vee $ is
essentially self-dual of grade $2f$
by hypothesis \ref{hypothesis:iii},
it follows from
\cite[Cor.~III.2.1.6]{LvO}
that $\pi^{\prime \prime \vee}$ 
is of grade $\ge 2f$,
hence $E^{2f-1}_{\Lambda} (\pi^{\prime\prime\vee})=0$.
In particular, the natural inclusion
$\pi' \subseteq \pi$ induces a surjection
$ \E(\pi^\vee)^\vee 
\twoheadrightarrow 
\E(\pi^{ \prime\vee})^\vee$.

\begin{definition}[]
	\label{def:conjugate-srep}
We define the subrepresentation
$ \widetilde{\pi}' \subseteq \pi$
as 
\begin{equation}
	\label{eq:conjugate-def}
\widetilde{\pi}' \defeq 
\ker( \pi 
\xrightarrow[\sim]{\kappa_\pi}
\E(\pi^\vee)^\vee \otimes (\det(\overline{\rho})\omega^{-1})
\twoheadrightarrow 
\E(\pi^{ \prime\vee})^\vee \otimes (\det(\overline{\rho})\omega^{-1})
).
\end{equation}
\end{definition}
\begin{remark}
	\label{rem:conjugate-when-CM}
Keep the hypotheses of \Cref{def:conjugate-srep},
and consider the long exact sequence
induced by the short exact sequence
$
0 \to \pi^{\prime \prime\vee} \to \pi^\vee  \to \pi^{\prime\vee }\to 0
$
after taking $ \Hom_{\Lambda}(-, \Lambda)$.
From that exact sequence we obtain 
the following commutative diagram
with exact rows:
\begin{equation}
	\label{eq:conjugate-when-CM}
	\begin{tikzcd}
	0 \ar[r] &  	\E(\pi^{\prime\vee})\ar[r]
&  \E(\pi^\vee)  \ar[r]\ar[d,"{ \kappa_\pi^\vee \otimes (\mytwist) }"', "{\wr}"] 
& \im \left( 
\E(\pi^\vee) \to \E(\pi^{\prime \prime\vee}) 
 \right)
\ar[d,"{ \kappa_\pi^\vee \otimes  (\mytwist)}", "{\wr}"'] \ar[r]
& 0 \\
       & 
& \pi^\vee \otimes (\mytwist)  \ar[r]
&\widetilde{\pi}^{\prime\vee}\otimes (\mytwist) .
\end{tikzcd}\end{equation}
Indeed, 
we have seen that $\E[2f-1](\pi^{\prime \prime \vee })=0$
always holds, so the top row of 
\eqref{eq:conjugate-when-CM} is exact,
and we are left to show that the square is commutative.
Now, \eqref{eq:conjugate-def}
implies that $\kappa_\pi$ restricts 
to an isomorphism
\[
\widetilde{\pi}'
\xrightarrow[\sim]{\kappa_\pi}
\ker\left( \E(\pi^\vee)^\vee \otimes (\mytwist)
\twoheadrightarrow 
\E(\pi^{ \prime\vee})^\vee \otimes (\mytwist) \right),
\]
and we notice that the right-hand side 
is isomorphic to 
$ \im \left( 
\E(\pi^\vee) \to \E(\pi^{\prime \prime\vee}) 
 \right)\!^\vee$.
We conclude by taking 
$\F$-linear duals and twisting by $\mytwist$.

Notice that the
diagram \eqref{eq:conjugate-when-CM}
simplifies when
$\E[2f+1](\pi^{\prime \vee} )=0$,
as one can replace the term
$ \im \left( 
\E(\pi^\vee) \to \E(\pi^{\prime \prime\vee}) 
 \right)$
by $\E[2f](\pi^{\prime \prime \vee})$.
\end{remark}

We recall the definition of
the characteristic cycle functor given in
\cite{BHHMS2}.
\begin{definition}[{\cite[Definition~3.3.4.1]{BHHMS2}}]
	\label{def:characteristic-cycle}
Let $N$ be a finitely generated module
over $\grL$ which is annihilated by some
power of $J$.
We define the \emph{characteristic cycle}
of $N$, denoted by $\mathcal{Z}(N)$,
as
\[
\mathcal{Z}(N) \defeq \sum _{\mathfrak{q}} 
m_{\mathfrak{q}} (N)\mathfrak{q}
\in \bigoplus _{\mathfrak{q}} \mathbb{Z}_{ \ge 0} \mathfrak{q},
\]
where $\mathfrak{q}$ runs over all
minimal prime ideals of $\overline{R}$
and where the integer $m_{\mathfrak{q}} (N)$
was defined in \Cref{def:multiplicity}.
\end{definition}
The characteristic cycle is additive in short exact sequences
by \cite[Lemma~3.3.4.2]{BHHMS2}.

If $F'$, $F''$
are two good filtrations
(in the sense of \cite[§I.5]{LvO})
on a finitely generated $\Lambda$-module
$M$, such that $\gr_{F'}(M)$
(or equivalently $\gr_{F''} (M)$,
by the discussion before
\cite[Proposition~3.1.2.11]{BHHMS2})
is annihilated by some power of $J$,
then 
$\gr_{F'}(M)$ and $\gr_{F''} (M)$
are finitely generated as $\grL$-modules
by \cite[Lemma~I.5.4]{LvO},
and moreover
$\mathcal{Z}(\gr_{F'}(M)) =\mathcal{Z}(\gr_{F''} (M))$
 by \cite[Lemma~3.3.4.3]{BHHMS2}.
Therefore, we can define 
$m_{\mathfrak{q}} (M)$
to be  $m_{\mathfrak{q}} (\gr_F(M))$
and $\mathcal{Z}(M)$
to be  $\mathcal{Z}(\gr_F(M))$,
where $F$ is any good filtration on $M$.

\begin{lemma}
	\label{lemma:Z-equal-Z}
Assume that $\pi$ satisfies
hypotheses \ref{hypothesis:ii} and
\ref{hypothesis:iii}.
Let $\pi'$ be a subrepresentation of $\pi$,
and set $\pi'' \defeq \pi/\pi'$.
If $\widetilde{\pi}'$ 
is the subrepresentation 
of \Cref{def:conjugate-srep},
then 
\[
\mathcal{Z}(\pi^{\prime \prime \vee})
=
\mathcal{Z}(\widetilde{\pi}^{\prime \vee})
\]
\end{lemma}
\begin{proof}
It follows from \eqref{eq:conjugate-when-CM},
that we have the short exact sequence
\begin{equation*} 
0 \to   \E(\pi^{\prime\vee}) \to \E(\pi^\vee) 
\to 
\widetilde{\pi}^{\prime \vee} \otimes (\mytwist) 
\to 0.
\end{equation*}
By the additivity of $\mathcal{Z}$ we have 
\begin{align*}
\mathcal{Z}( \widetilde{\pi}^{\prime \vee} )
& = 
\mathcal{Z}(
\widetilde{\pi}^{\prime \vee} \otimes (\mytwist) )
= 
\mathcal{Z}(\E(\pi^\vee)) -
\mathcal{Z}(\E(\pi^{\prime\vee})), \\
\mathcal{Z}( \pi^{\prime \prime \vee} )
& = 
\mathcal{Z}(\pi^\vee) -
\mathcal{Z}(\pi^{\prime\vee})),
\end{align*}
and the two quantities are equal
by \cite[Lemma~3.3.4.5]{BHHMS2}
(hypothesis \ref{hypothesis:ii}
ensures that the assumptions of 
\emph{loc.\ cit.}\ are satisfied).
\end{proof}

\begin{proof}[Proof of \Cref{prop:main-memoire}]
(i)
The proof of  \cite[Proposition~3.3.5.3]{BHHMS2}(i)
goes through \emph{verbatim} when $r > 1$.
Note that it is only here that
we are using the genericity
assumption on $\repr$.

(ii)
Following the strategy of 
\cite[Proposition~3.3.5.3]{BHHMS2},
we prove the inequalities
\[
	\dim_{\mathbb{F}(\!(X)\!)}D^\vee_\xi(\pi')\le m_{\mathfrak p_0}(\grm\pi^{\prime\vee})\le \lg_{\GLring} \socR(\pi')\le \dim_{\mathbb{F}(\!(X)\!)}D^\vee_\xi(\pi').
\]

We know from \cite[Corollary~3.1.4.5]{BHHMS2} that 
\[
\dim_{\mathbb{F}(\!(X)\!)}D^\vee_\xi(\pi')\le m_{\mathfrak p_0}(\grm\pi'^\vee).
\]
In order to show that 
\[
m_{\mathfrak p_0}(\grm\pi'^\vee)\le\lg_{\GLring} \socR(\pi'),
\]
we consider the surjection \eqref{eq:surjection_Ni}
of \Cref{lemma:gr}(i)
(this is where we need hypothesis
\ref{hypothesis:ii}).
If we let 
$r'(\ell)\defeq \rell$,
then by the additivity of $m_{\mathfrak p_0}$ we have
\begin{align*}
    m_{\mathfrak p_0}(\grm\pi'^\vee)
    &\le
    \sum_\ell r'(\ell)\sum_{\lambda\in\mathscr P_\ell}
    m_{\mathfrak p_0}\left(\chi_\lambda^{-1}\otimes_{\F}  R/\mathfrak a(\lambda)\right)
    =\\
    &\overset{\eqref{eq:mult_id}}{=}\sum_\ell r'(\ell)\sum_{\lambda\in\mathscr P_\ell}
    \mathbbm{1}_{\mathscr{D}}(\lambda)
    =\sum_\ell r'(\ell)\cdot|\mathscr P_\ell\cap \mathscr{D}|
    =\\
    &=\lg_{\GLring} \socR(\pi').
\end{align*}
Finally, we conclude by showing that 
\[
\lg_{\GLring} \socR(\pi')\le \dim_{\mathbb{F}(\!(X)\!)}D^\vee_\xi(\pi').
\]
Denote by $\chi_{\pi}  $
the central character of $\pi $.
Recall from the paragraph after 
\eqref{eq:diagr1-def}
that $K^{\times}$ acts on $\DZero$
via the central character $\mytwist$,
in particular we have
$\chi_{\pi} =\det(\repr)\omega ^{-1}$
by hypothesis \ref{hypothesis:i} of 
\Cref{sec:hypotheses}.

Recall that
$N_1 \defeq \ker \left( 
N_0 \overset{\eqref{eq:N0-def}}{\cong} \mathcal{O} _K 
\xtwoheadrightarrow{\Tr}
\mathbb{Z} _p \right)$,
 and that in the discussion before
\Cref{sec:hypotheses}, we have given $\pi^{N_1}$ an 
$\F [\![X]\!] [F]$-module structure
with a compatible action of 
$\mathbb{Z}_p ^{\times}$.
If $M$ is any $\F [\![X]\!] [F]$-submodule
of $\pi^{N_1}$ which is stable under
$\mathbb{Z}_p^{\times}$,
then following the paragraph before
\cite[Lemma~3.2.1.2]{BHHMS2}
we denote by $M \otimes \chi_\pi^{-1} $
the $\F [\![X]\!] $-module $M$
where the action of $F$ is multiplied
by $\chi_\pi(p)^{-1} $
and the action of $ x \in \mathbb{Z}_p^{\times} $
is multiplied by $\chi_\pi(x)^{-1} $.

Set
$n(\sigma)\defeq |\{\delta^i(\sigma) \mid i\in \mathbb{N}\}|$.
Following the paragraph before
\cite[Proposition~3.2.4.2]{BHHMS2},
for $\sigma \in W(\repr)$
we consider the
$\mathbb{F}[\![X]\!][F]$-submodule 
$M_\sigma\otimes\chi_{\pi}^{-1}$
of $\pi^{N_1}$
generated by 
$Y^{1-m}\delta^{i}(\sigma)^{N_0}$,
for $0 \le i \le n(\sigma)-1$,
where $m$ is the integer defined in 
the paragraph before 
\cite[Proposition~3.2.3.1]{BHHMS2}.
For this definition,
we remind the reader that an embedding 
$\sigma \hookrightarrow \socR (\pi)$ 
was fixed,
cf.\ the paragraph before
\cite[Proposition~3.2.3.1]{BHHMS2}.

To be more precise,
if $L \subseteq V$ is a line,
we denote by $L \otimes_{\F} M_{\sigma}$
the $\mathbb{F}[\![X]\!][F]$-submodule 
of $\pi^{N_1}$ obtained using the embedding 
$\sigma \cong 
\iota_\pi^{-1} (L \otimes _{\F} \sigma)
\subseteq \pi^{N_1}$
(which does not depend on the choice of
isomorphism
$\sigma \cong 
\iota_\pi^{-1} (L \otimes _{\F} \sigma)$).
Here, of course, we are regarding
$L \otimes_{\F} \sigma$ as a subset
of $V \otimes_{\F} \DZero$.
Notice also that $L \otimes_{\F} M_{\sigma} $
only depends on the orbit of $\sigma$
under the action of $\delta$,
cf.\ \Cref{def:weight-cycling}.

If $W \subseteq \F^{r}$ is any 
vector subspace, then we set
$W \otimes_{\F} M_{\sigma} 
\defeq \sum _{L \subseteq W \text{ line}}
L \otimes_{\F}M_{\sigma}$,
and we notice that 
$W \otimes_{\F} M_{\sigma} \subseteq 
\pi^{\prime N_1}$
whenever 
$W \subseteq \Vell[\pi',\ell(\sigma)]$,
using \Cref{cor:direction}.
Then we define
the $\mathbb{F}[\![X]\!][F]$-module $M_{\pi'}$
following the paragraph before
 \cite[Proposition~3.2.4.6]{BHHMS2}:
\[
M_{\pi'}
\defeq 
\bigoplus_{\sigma\in \mathcal{O}(\overline \rho)}
\bigl( \Vell[\pi',\ell(\sigma)]
\otimes_{\F} 
M_\sigma \bigr)
\subseteq \pi^{\prime N_1},
\]
where $\mathcal{O}(\repr)$ is a set of representatives for the $\delta$-orbits in $W(\overline \rho)$
(by the previous paragraph, we see that
$M_{\pi'}$
does not depend on the choice of 
$\mathcal{O}(\repr)$).

It follows from \cite[Proposition~3.2.4.2]{BHHMS2} that 
$(L \otimes_{\F}  M_\sigma\otimes\chi_{\pi}^{-1})^\vee[1/X]$ is
an $\mathbb{F}(\!(X)\!)$-vector space of dimension $n(\sigma)$.
In fact, a basis $e_1(L), \dots, e_{n(\sigma)}(L) $
is constructed explicitly
in \emph{loc.\ cit.}
(and denoted there by 
$e_1, \dots, e_{n(\sigma)}$).

For $\sigma \in  \mathcal{O}(\repr)$,
set $r'(\sigma) \defeq \dim_{\F}
\Vell[\pi',\ell(\sigma)]$,
and choose $r'(\sigma)$
lines $L_1, \dots, L_{r'(\sigma)}$
in direct sum inside $\Vell[\pi',\ell(\sigma)]$.
Then, one can check that
$\left(e_j(L_m)\mid
1 \le j \le n(\sigma),
1 \le m \le r'(\sigma) \right)$
is a basis of 
$(\Vell[\pi',\ell(\sigma)] \otimes_{\F}
M_\sigma\otimes\chi_{\pi}^{-1})^\vee[1/X]$.

From this discussion, it follows that
the $\F (\!(X)\!) $-vector space
 $(M_{\pi'}\otimes\chi_{\pi}^{-1})^\vee[1/X]$
is of dimension
\begin{equation}
	\label{eq:dim=r-n=lg-soc}
\dim_{\mathbb{F}(\!(X)\!)}(M_{\pi'}\otimes\chi_{\pi}^{-1})^\vee[1/X]=
	\sum_{\sigma\in\mathcal{O}(\overline \rho)}r'(\sigma)\cdot n(\sigma)
=
\lg_{\GLring} \socR(\pi'),
\end{equation}
for the last equality remember that
$n(\sigma)$ is the size of a $\delta$-orbit,
and that we are summing over all $\delta$-orbits.

We then have the surjection
\[
D_\xi^\vee(\pi')\twoheadrightarrow (M_{\pi'}\otimes\chi_{\pi}^{-1})^\vee[1/X],
\]
which follows directly from 
\Cref{def:Dvee}.
The fact that $M_{\pi'} \otimes \chi_\pi ^{-1} $
is a submodule of $\pi'^{N_1} $ has been established in the previous paragraph,
while
it is shown in
 \cite[Proposition~3.2.4.2]{BHHMS2}
 that $M_\sigma \otimes \chi_\pi^{-1} $
is admissible and stable under the action of $\mathbb{Z} _p ^\times  $,
hence the same is true of  
$M_{\pi'} \otimes \chi_\pi^{-1} $. 

Counting dimensions, we find
\begin{align*}
\dim_{\mathbb{F}(\!(X)\!)}D^\vee_\xi(\pi')\ge \dim_{\mathbb{F}(\!(X)\!)}(M_{\pi'}\otimes \chi_\pi ^{-1})^\vee[1/X]& 
\overset{\eqref{eq:dim=r-n=lg-soc}}{=}
\lg_{\GLring} \socR(\pi'),
\end{align*}
which concludes the proof.

(iii)
Let 
$\widetilde{\pi}'$ be the subrepresentation
of \Cref{def:conjugate-srep}
associated to $\pi'$
(this is where we use hypothesis
\ref{hypothesis:iii}).
Then,
\begin{equation}
	\label{eq:Z-Z-application}
\mathcal{Z}(\pi^{\prime \prime \vee})
= 
\mathcal{Z}(\widetilde{\pi}^{\prime \vee})
\end{equation}
by \Cref{lemma:Z-equal-Z}.
It follows from
\cite[Prop.~III.4.2.8(1)]{LvO}
and from
\cite[Prop.~III.4.2.9]{LvO}
that $\pi^{\prime\prime\vee}$ has grade $2f$,
hence $\mathcal{Z}(\pi^{\prime\prime \vee})$
is nonzero, for example by the discussion
before \cite[Theorem~3.3.4.5]{BHHMS2}.
In particular,
$\widetilde{\pi}^{\prime} $ is also
nonzero, and so
$ 0 \neq \lg_{\GLring}  \socR(\widetilde{\pi}^{\prime})
    \overset{\text{(ii)}}{= }
    \dim_{\mathbb{F}(\!(X)\!)}D^\vee_\xi(\widetilde{\pi}^{\prime}).$

But \eqref{eq:Z-Z-application}
implies
\[
    \dim_{\mathbb{F}(\!(X)\!)}D^\vee_\xi(\pi^{\prime\prime })=
    \dim_{\mathbb{F}(\!(X)\!)}D^\vee_\xi(\widetilde{\pi}^{\prime})\neq 0
\] 
by (i).
\end{proof}

The following is a generalisation of \cite[Theorem~3.3.5.5]{BHHMS2}
to the case $r\ge 1 $:
\begin{prop}
	\label{prop:generation}	
Assume that $\repr$ is 
$2f$-generic,
and that $\pi$ satisfies
hypotheses \ref{hypothesis:i} to
\ref{hypothesis:iii}.
Then, $\pi $ is generated,
as a $\GLfield $-representation,
by its $\GLring $-socle.
\end{prop}
In particular, $\pi $ is of finite type
as a $\F[\GLfield] $-module.
\begin{proof}
The proof of 
\cite[Theorem~1.3.8]{BHHMS2}
passes through in the $r\ge 1 $ case,
we report it here for the reader's convenience.
Let $\tau \defeq  \socR \pi$,
and let $\pi' = \left\langle \GLfield \cdot \tau \right\rangle$
be the $\GLfield $-subrepresentation of $\pi $
generated by $\tau$.
If we define $\pi'' $ to be the quotient $\pi'' \defeq \pi/\pi'  $,
then by the exactness of the functor $\Dvee[-] $ 
(cf.\ \cite[Theorem~3.1.3.7]{BHHMS2})
we have
\begin{equation}
\label{eq:D_ex}
	\dim_{\F (\!(X)\!) } \Dvee[\pi] =
	\dim_{\F (\!(X)\!) } \Dvee[\pi'] +
	\dim_{\F (\!(X)\!) } \Dvee[\pi''].
\end{equation}
According to
\Cref{prop:main-memoire} \ref{point:ii},
and by construction of $\pi' $,
we have 
\begin{align*}
	\dim_{\F (\!(X)\!) } \Dvee[\pi] =
	\lg_{\GLring} \socR(\pi)&=
       \lg_{\GLring} \tau= \\
       =\lg_{\GLring} \socR(\pi')&=
	\dim_{\F (\!(X)\!) } \Dvee[\pi'].
\end{align*}
Substituting into \eqref{eq:D_ex}
we find $\dim_{\F (\!(X)\!) } \Dvee[\pi'']= 0$
and, using 
\Cref{prop:main-memoire} \ref{point:iii},
we deduce that $\pi''=0$, which concludes the proof.
 \end{proof}

The following generalises 
\cite[Corollary~3.3.5.8]{BHHMS2}.
\begin{thm}[]
	\label{thm:principal-series}
Assume that $\repr$ is 
$2f$-generic,
and that $\repr$ is reducible split,
say of the form $\repr \cong 
\begin{pmatrix} 
\chi_1& 0 \\
0 & \chi_2 \\
\end{pmatrix}$.
Then, $\pi$ has the form
\[
	\pi = \pi_0^{\oplus r} \oplus 
	\pi_f^{\oplus r} \oplus \pi',
\]
where
\begin{enumerate}[(i)]
\item 
	\label{item:PS-1}
the representations $\pi_0$ and
$\pi_f$ are the irreducible principal 
series 
\[
\begin{array}{ll}
\pi_0 \cong    \Ind_{B(K)} ^{\GLfield}(\chi_1
\otimes \chi_2 \omega^{-1} ),
 &  
\pi_f \cong   \Ind_{B(K)} ^{\GLfield}(\chi_2
\otimes \chi_1 \omega^{-1} ),
\end{array}
\]
\item 
	\label{item:PS-2}
the representation $\pi'$ is generated by its $\GLring$-socle
	and $\pi^{\prime\vee}$ is essentially
self-dual 
(as in hypothesis \ref{hypothesis:iii} of
\Cref{sec:hypotheses}).
\end{enumerate}
\end{thm}
\begin{proof}
We follow the proof of 
\cite[Corollary~3.3.5.8]{BHHMS2}.
Let $\sigma_0\in W(\repr)$
be the unique Serre weight of length $0$
in $W(\repr)$,
and let $\chi_{\sigma_0} $ be the character
of $I$ acting on $\sigma_0^{I_1}$.
For each line $L \subseteq \V$,
we let
\begin{equation}
	\label{eq:pi0-line}
	\pi_0(L) \defeq \left\langle \GLfield \cdot \iota_\pi^{-1}(L \otimes \sigma_0) \right\rangle
\hookrightarrow \pi,
\end{equation}
and we show that it is an irreducible
principal series.

First, we claim that
\begin{equation}
	\label{eq:wt-intersection-singleton}
	\JH(\Ind_{I} ^{\GLring} \chi_{\sigma_0} ) \cap W(\repr) = \{ \sigma_0 \}.
\end{equation}
Notice that $\ell(\sigma_0)=0$ 
implies 
$\sigma_0 =(r_0, \dots, r_{f-1})$,
by \Cref{def:D-red-irr}
and \Cref{def:J-and-ell}.
In addition to the set of $f$-tuples 
$\mathscr{D}$ of \Cref{def:D-red-irr},
we recall also the set 
$\mathcal{P}(x_0, \dots, x_{f-1})$
of $f$-tuples defined in
\cite[§2]{BP12},
whose definition we report here 
for the reader's convenience:
$\lambda \in \mathcal{P}(x_0, \dots, x_{f-1})$
is an $f$-tuple with 
$\lambda_i(x_i) \in \mathbb{Z} \pm x_i$.
If $f=1$, $\mathcal{P}(x_0)=\{x_0, p-1-x_0\}$,
and if $f>1$ then 
$(\lambda_0(x_0), \dots,\lambda_{f-1}(x_{f-1})) \in \mathcal{P}(x_0, \dots, x_{f-1})$
if and only if the following three
conditions are satisfied, for 
$i \in \{0,\dots, f-1\}$:
\begin{subequations}
\label{eq:P-conditions}
\begin{align}
	\label{eq:P-condition-a}
\lambda_i(x_i) &\in \{x_i,x_i-1,p-2-x_i,p-1-x_i\},\\
	\label{eq:P-condition-b}
\lambda_i(x_i) &\in \{x_i, x_i-1\} \implies
\lambda_{i+1}(x_{i+1}) \in \{x_{i+1} ,p-2-x_{i+1}\},\\
	\label{eq:P-condition-c}
\lambda_i(x_i) &\in \{p-2-x_i, p-1-x_i-1\} \implies
\lambda_{i+1}(x_{i+1})\in \{p-1-x_{i+1},x_{i+1}-1\},
\end{align}
\end{subequations}
with the conventions 
$x_f = x_0$, $\lambda_f(x_f)=\lambda_0(x_0)$.
By \cite[Lemma~2.2]{BP12}
(and since $\sigma_0$ is at least $1$-generic),
\begin{equation}
	\label{eq:JH-Ind-explicit}
\JH(\IndI \chi_{\sigma_0}) =
\left\{
(\lambda(r_0), \dots,\lambda(r_{f-1}))
\otimes \textstyle\det^{e(\lambda)(r_0,\dots,r_{f-1})}
\mid \lambda \in 
\mathcal{P}(x_0, \dots, x_{f-1})
\right\},
\end{equation}
where $e(\lambda)$ was defined in
\eqref{eq:e-lambda-def}.

Hence, the claim is equivalent to
the combinatorial statement that
$\mathcal{P}(x_0, \dots, x_{f-1}) \cap \mathscr{D} = 
\{(x_0, \dots, x_{f-1})\}$.
We give a sketch of the argument:
take $\lambda$ in this intersection, 
then \eqref{eq:P-condition-a}
and \Cref{def:D-red-irr}(i)
imply that
$\lambda_i(x_i) \in \{x_i, p-2-x_i\}$
for $i \in \{0, \dots, f-1\}$.
Using \eqref{eq:P-condition-c},
we see that we must have
$\lambda_{i}(x_i) = x_i$ for all $i \in \{0, \dots, f-1\}$.

From \eqref{eq:wt-intersection-singleton}
it follows that,
if
 $\tau \in  \JH(\Ind_{I} ^{\GLring} \chi_{\sigma_0} ) $
is a Serre weight
distinct from $\sigma_0$, then
we have $\Hom_{\GLring} (\tau, \pi_0(L))=0$.
So, the hypothesis of \cite[Lemma~5.14]{HW22}
is satisfied. 
As in the proof of 
\emph{loc.\ cit.},
if $ \cInd$
denotes the compact induction from 
$\KK $ to $\GLfield $,
we see that the morphism
\[
	\cInd \sigma_0 \to \pi_0(L)
\]
induced by Frobenius reciprocity from
$\sigma_0 \cong L \otimes \sigma_0
\xhookrightarrow{\iota_\pi^{-1}}
  \pi_0(L)$
factors through 
$\cInd \sigma_0/(T- \mu_0(L))$,
for some scalar $\mu_0(L) \in \F^\times.$
By \cite[Theorem~30]{BL94},
$\cInd \sigma_0/(T- \mu_0(L))$
is irreducible and isomorphic to some
principal series 
$\cInd \sigma_0/(T- \mu_0(L)) \cong 
\Ind_{B(K)} ^{\GLfield}\chi_0,$
where $\chi_0: B(K) \to \F^{\times}$
is a smooth character.
We will later determine the form of $\chi_0$.
Then,
\begin{equation}
	\label{eq:pi-L}
\Ind_{B(K)} ^{\GLfield}\chi_0 \cong 
	\cInd \sigma_0/(T- \mu_0(L)) \xrightarrow{\sim}
\pi_0(L)
\end{equation}
is an isomorphism,
since $\pi_0(L)$ is generated by $ \sigma_0$.
In particular, the $\GLring$-socle of 
$\pi_0(L)$ is 
exactly $\sigma_0$, and so $\Vell[\pi_0(L),\ell]$
is $L$ for $\ell=0$, and $0$ otherwise.

\paragraph{Step 1.}
We prove that all the $\pi_0(L)$
are isomorphic as the line $L \subseteq V$
varies.

Suppose by contradiction
 that $\pi_0(L_1) \not\cong \pi_0(L_2)$,
for some lines $L_1,L_2 \in V$.
Let $W \defeq L_1 + L_2 \subseteq V$,
and pick any line $L' \subseteq W$ different
from $L_1$ and $L_2.$
Since $\pi_0(L_1) + \pi_0(L_2)$
is a direct sum of two irreducible representations,
in particular it is
semisimple.
It also contains $\pi_0(L')$,
so
we know that $\pi_0(L')$
is isomorphic to either 
 $\pi_0(L_1)$
or 
$ \pi_0(L_2)$,
say the former without loss of generality.
On the other hand,
$\pi_0(L_2)$ is contained in
$\pi_0(L_1) + \pi_0(L')$,
which implies that
$\pi_0(L_1) \cong \pi_0(L_2) $,
contradiction.

An alternative proof is that,
keeping notation as above,
there are $|\F|+1$ lines $L'$ in $W$
but only $|\F|-1$ values that $\mu_0(L')$
can take, so there exists lines
$L'_1,L'_2 \subseteq W$ 
with $\mu_0(L_1')= \mu_0(L_2')$,
hence $\pi_0(L_1')\cong \pi_0(L_2')$,
which together with
$\pi_0(L_1), \pi_0(L_2) \subseteq \pi_0(L_1')+ \pi_0(L_2')$
implies that
$\pi_0(L_1)\cong \pi_0(L_2)$, contradiction.

Analogously, if 
$\sigma_f$ is
 the unique Serre weight of length $f$
in $W(\repr)$,
and if
we let
\[
	\pi_f(L) \defeq \left\langle \GLfield \cdot \iota_\pi^{-1}(L \otimes \sigma_f) \right\rangle,
\]
where $L \subseteq V $ is a line,
then there is an isomorphism
\begin{equation}
	\label{eq:pi-L-2}
\Ind_{B(K)} ^{\GLfield}\chi_f  \xrightarrow{\sim}
 \pi_f(L)
\end{equation}
of $\GLring$-representations,
for some smooth character
$\chi_f: B(K) \to \F^{\times}$,
and the $\GLring$-socle 
of $\pi_f(L)$ is
exactly equal to $\sigma_f$.

From now on, all tensor products are
assumed to be over $\F$.
If $\pi_0$ 
is (a fixed representative of)
the isomorphism class of the $\pi_0(L)$,
and if $\pi_f$ 
is (a fixed representative of)
the isomorphism class of the $\pi_f(L)$,
then we can fix a $\GLring$-equivariant injection
\begin{equation}
	\label{eq:pi0-pif-inside-pi}
V \otimes \left(  \pi_0 \oplus \pi_f  \right) \hookrightarrow \pi.
\end{equation}

We let 
$\pi' \defeq \pi/
\left(  V \otimes \left(  \pi_0 \oplus \pi_f  \right) \right)$
and consider the short exact sequence 
\begin{equation}
	\label{eq:PS-ses}
0 \to V   \otimes \left(  \pi_0  \oplus \pi_f \right) \to \pi  
\to {\pi'}\to 0.
\end{equation}

\paragraph{Step 2.}
We prove that \eqref{eq:PS-ses} splits,
and prove \ref{item:PS-1} of the
statement.

By dualising
\eqref{eq:PS-ses}
we obtain
a short exact sequence 
\[
0 \to \pi^{\prime\vee}\to \pi ^\vee 
\to V ^\vee  \otimes \left(  \pi_0 ^\vee \oplus \pi_f ^\vee  \right)\to 0
\]
of $\Lambda$-modules.
As $\Lambda$ is Auslander regular 
(cf.\ the discussion before
\Cref{def:multiplicity})
and since $\pi ^\vee $ is 
essentially self-dual of grade $2f$
by hypothesis \ref{hypothesis:iii},
it follows from
\cite[Cor.~III.2.1.6]{LvO}
that $\pi^{ \prime \vee}$ 
is of grade $\ge 2f$,
hence $E^{2f-1}_{\Lambda} (\pi^{\prime\vee})=0$
and we have the following exact sequence of
(finitely-generated) $\Lambda$-modules:
\begin{equation}
	\label{eq:E-long-ex-seq}
	0\to V \otimes \left(  \E (\pi_0 ^\vee ) \oplus \E (\pi_f ^\vee ) \right)
	\to  \E(\pi ^\vee ) 
	\to  \E(\pi^{\prime\vee}) .
\end{equation}
By \cite[Proposition~5.4]{Koh17},
there exist irreducible principal series
$\pi'_i$, for $i\in\{0,f\}$,
that satisfy
\begin{equation}
	\label{eq:PS-self-duality}	
	\pi_i^{\prime\vee} \otimes(\det(\repr) \omega ^{-1}) \cong \E({\pi_{f-i} ^\vee }).
\end{equation}
Moreover, following the proof of \cite[Corollary~3.3.5.8]{BHHMS2}, 
one computes the $\GLring$-socle of $\pi_i'$
to be $\sigma_i$, for $i\in \{0,f\}$.
Hypothesis \ref{hypothesis:iii}
of \Cref{sec:hypotheses} implies that,
after dualising \eqref{eq:E-long-ex-seq}
and twisting by $\det(\repr)\omega^{-1}$,
we obtain a surjection
$\pi \twoheadrightarrow V^\vee \otimes \left( {\pi_0'}\oplus \pi_f' \right) $.

For $i\in \{0,f\}$,
we claim that the composition
\begin{equation}
	\label{eq:morally-an-endomorphism}
V\otimes  {\pi_i} \hookrightarrow \pi \twoheadrightarrow 
V^\vee \otimes  {\pi_i'} 
\end{equation}
is an isomorphism of $\GLring$-representations,
which implies that the short exact sequence
\eqref{eq:PS-ses}
admits a retraction.

By reasons of length, it suffices to 
show that is is surjective;
moreover, it suffices
to show that, for each
line $L \subseteq V$, the composition
\begin{equation}
	\label{eq:argument-for-surjectivity}
V\otimes  {\pi_i} \hookrightarrow \pi \twoheadrightarrow 
V^\vee \otimes  {\pi_i'} 
\twoheadrightarrow
L^\vee \otimes  {\pi_i'} 
\end{equation}
is surjective 
(because any subrepresentation of $V^\vee \otimes \pi_i'$
is of the form
$(V/W)^\vee \otimes \pi_i'$
for some subspace $W \subseteq V$).
If we write $\im(\eqref{eq:morally-an-endomorphism})=(V/W)^\vee \otimes \pi_i'$,
and if $W\neq V$,
then any line $L$ in direct sum with $W$
will be such that \eqref{eq:argument-for-surjectivity}
is $0$.

But $\pi$ is generated by its $\GLring$-socle,
which implies that
\[
	\iota_0 \colon\socR \pi \to L^\vee \otimes \pi_i'
\]
is nonzero, and factors through
the socle $L^\vee \otimes\sigma_i$ of 
$L^\vee \otimes \pi_i'$.
So, $\iota_0$ has to be nonzero
when restricted to $V \otimes \sigma_i$,
so \eqref{eq:argument-for-surjectivity} 
is nonzero, hence surjective,
using that $\pi_i'$ is irreducible.

From \eqref{eq:morally-an-endomorphism}
being an isomorphism, we deduce
$\pi_i \cong \pi_i'$, for $i \in \{0,f\} $,
and in particular, $\socR \pi_i = \sigma_i$.

This forces $\chi_i = \sigma_i^{I_1}$,
for example using Frobenius reciprocity
on \eqref{eq:pi-L} and \eqref{eq:pi-L-2}.
From this we deduce that
$\chi_i$ is as in (i)
of the statement,
i.e.\ $\chi_0= \chi_1 \otimes \chi_2 \omega^{-1} $,
$\chi_f= \chi_2 \otimes \chi_1 \omega^{-1} $.

\paragraph{Step 3.}
We prove \ref{item:PS-2} of the
statement.
Generation by the $\GLring$-socle follows from \Cref{prop:generation}:
$\pi$ is generated by
$V \otimes (\sigma_0 \oplus \sigma_f)$ and by 
$V \otimes \bigoplus_{\substack{0<\ell(\sigma)<f\\ \sigma\in W(\repr)} } \sigma$,
but the former is contained in 
$V \otimes(\pi_0 \oplus \pi_f)$
and the latter is contained in $\pi'$.

The essential self-duality of $\pi_0$
and $\pi_f$ has been shown in the previous step,
we now prove the essential self-duality
of $\pi'$.

First, we combine the essential self-duality
of $\pi_0$ and $\pi_f$ with 
the splitting of \eqref{eq:PS-ses}:
we obtain an isomorphism
\begin{equation}
	\label{eq:towards-pi'-self-dual} 
\E(\pi^\vee )^\vee \otimes \det(\repr )\omega^{-1}\cong
\left( V \otimes (\pi_0 \oplus \pi_f) \right) \oplus 	
\left[ \E(\pi^{\prime\vee})^\vee  \otimes \det(\repr )\omega^{-1}\right],
\end{equation}
which also shows that \eqref{eq:E-long-ex-seq}
is a short exact sequence.
Then, the essential self-duality of $\pi$
tells us that the left hand side of
\eqref{eq:towards-pi'-self-dual} is
isomorphic to $\pi \cong \left( V \otimes (\pi_0 \oplus \pi_f) \right) \oplus \pi'$.
Putting this isomorphism together with
\eqref{eq:towards-pi'-self-dual}, we find
\[
i \colon\left( V \otimes (\pi_0 \oplus \pi_f) \right) \oplus \pi'
\cong 
\left( V \otimes (\pi_0 \oplus \pi_f) \right) \oplus 	
\left[ \E(\pi^{\prime\vee})^\vee \otimes \det(\repr )\omega^{-1} \right],
\]
which restricts to an isomorphism
$i \colon  V \otimes (\pi_0 \oplus \pi_f)
\xrightarrow{\sim}
V \otimes (\pi_0 \oplus \pi_f)$.
Moreover, if we take the $\GLring$-socle
of both sides, it restricts to
\begin{equation}
	\label{eq:soc-is-right}
\socR i \colon \socR \pi' \cong 
\socR \left[ \E(\pi^{\prime\vee})^\vee \otimes \det(\repr )\omega^{-1} \right].
\end{equation}
This is because 
\begin{align*}
	\socR \pi' &\cong \socR\bigoplus_{ 0 < \ell<f} 
\DEll[\ell],\\
\socR \pi_0 &\cong \socR\DEll[0],\\
\socR \pi_f &\cong \socR\DEll[f],
\end{align*}
which share no common Jordan-Holder factors.

Let $p_2$ be the natural projection
\[
	p_2 \colon
	\left( V \otimes (\pi_0 \oplus \pi_f) \right) \oplus 	
	\left[ \E(\pi^{\prime\vee})^\vee \otimes \det(\repr )\omega^{-1} \right]
	\twoheadrightarrow 
	\E(\pi^{\prime\vee})^\vee \otimes \det(\repr )\omega^{-1}, 
\]
and consider the composition
\begin{equation}
	\label{eq:to-be-iso}
	\pi' \xhookrightarrow{i|_{\pi'} }
\left( V \otimes (\pi_0 \oplus \pi_f) \right) \oplus 	
\left[ \E(\pi^{\prime\vee})^\vee \otimes \det(\repr )\omega^{-1} \right]
\xtwoheadrightarrow{p_2}
	\E(\pi^{\prime\vee})^\vee \otimes \det(\repr )\omega^{-1}. 
\end{equation}
We claim that it is an isomorphism.
Injectivity follows from
\eqref{eq:soc-is-right},
while surjectivity follows from
the fact that $p_2 \circ i$ is surjective,
together with
$i( V \otimes (\pi_0 \oplus \pi_f))
= V \otimes (\pi_0 \oplus \pi_f)$.
\end{proof}

\section{Cohen-Macaulayness and finite length}
	\label{sec:CM-fl} 
The goal of this section is to show,
in \Cref{prop:main}(ii) below,
that every subquotient $\pi'$ 
of the representation $\pi$ of
\eqref{eq:piShi-dev} is Cohen-Macaulay
of grade $2f$ and,
in \Cref{thm:finite-length-yay}(ii) below,
that $\pi$ has finite length,
with an explicit upper bound.
We conclude with some results concerning
the exactness of $\mI^{n}$-torsion
and $K_1$-invariants on subquotients of $\pi$.

In addition to hypotheses
\ref{hypothesis:i} and \ref{hypothesis:ii}
of \Cref{sec:hypotheses},
throughout the Section
we also assume hypotheses \ref{hypothesis:iii}
and \ref{hypothesis:iv}
of \emph{loc.\ cit}.
\subsection{Cohen-Macaulayness}
	\label{sec:CM}
We prove that $\pi$, and all its subquotients,
are Cohen-Macaulay of grade $2f$;
we start by recalling a theorem from \cite{BHHMS4}.
\begin{thm}[{\cite[Theorem~2.1.2]{BHHMS4}}]
	\label{thm:N-to-gr-iso}
Assume that $\repr$ is $9$-generic.
If $\Nnull$ is the graded $\grL$-module
with compatible $H$-action defined in 
\eqref{eq:N-null},
then the morphism 
\begin{equation}
 \theta  \colon \V^\vee  \otimes_{\F} \Nnull {\to} \gr_{\mI} (\pi^\vee )
\end{equation}
of \eqref{eq:surjection_Ni}
for $\pi_i=\pi$
is an isomorphism 
of graded $\grL$-modules with 
compatible $H$-action.
In particular, $\gr_{\mI}(\pi^\vee ) $
is a Cohen-Macaulay $\grL$-module
of grade $2f$. 
Moreover, $\gr_{\mI}(\pi^\vee ) $
is essentially self-dual in the sense that
\begin{equation}
	E^{2f}_{\grL} (\gr_{\mI} (\pi^\vee )) \cong 
	\gr_{\mI} (\pi^\vee ) \otimes(\det(\repr)\omega^{-1})
\end{equation}
as $\grL$-modules (without grading)
with compatible $H$-action.
\end{thm}

The following proposition generalises
both
\cite[Proposition~3.2.2]{BHHMS4}
and
\cite[Corollary~3.2.4]{BHHMS4}
to the case
$r\ge 1$.
\begin{prop}
	\label{prop:main}
Suppose that $\repr$ is
$\max \{9, 2f+1\} $-generic, 
let $0 \subsetneq \pi_1 \subsetneq \pi $ 
be a subrepresentation of $\pi $,
and let $\pi_2 \defeq \pi/\pi_1 $.
Endow $\pi_2^{\vee }$
with the submodule filtration $F$
induced from the 
$\mI$-adic filtration on $\pi ^\vee$.
Set
\begin{align}
	\label{eq:N-and-N1-def}
N \defeq 
 \V^\vee \otimes_{\F} \Nnull ,
\quad
N_1 \defeq \bigoplus_{\ell=0}^f
\left( \Vell [\pi_i,\ell]^\vee \otimes_{\F} \Nell \right),
\end{align}
for $i=1,2$.
If $\repr$ is reducible split,
for each $0 \le \ell \le f$
pick a complementary subspace
$V_2(\ell)$ of $\Vell[\pi_1,\ell]$ in $\V$
and set 
\begin{equation}
	\label{eq:N2-def-red}
N_2 \defeq 
\bigoplus_{\ell=0}^f
\left( V_2(\ell) ^\vee \otimes_{\F} \Nell \right),
\end{equation}
where $\Nell$ is the graded $\grL$-module
with compatible $H$-action defined in 
\eqref{eq:N-ell}.
If $\repr$ is irreducible,
pick a complementary subspace
$V_2$ of $\Vnull[\pi_1]$ in $\V$
and set 
\begin{equation}
	\label{eq:N2-def-irred}
N_2 \defeq 
 V_2 ^\vee \otimes_{\F} \Nnull .
\end{equation}
Then, the commutative diagram 
\eqref{diagr:thetas-ses}
of \Cref{lemma:gr}(ii) for
$\pi_1=\pi_1, \pi_2=\pi, \pi'=\pi_2$
takes the form
\begin{equation}
	\label{diag:thetas}
\begin{tikzcd}
	0 \ar[r] & \gr_F( \pi _2 ^\vee 	) \ar[r] &	\gr_{\mI} ( \pi ^\vee ) \ar[r] &
	\gr_{\mI} ( \pi_1 ^\vee ) \ar[r] & 0 \\
	0 \ar[r] & N_2 \ar[u, "\theta_2"]  \ar[r] & 	N \ar[u, " \theta "] \ar[r]  & 
	N_1  \ar[u, " \theta _1"'] \ar[r] & 0.
\end{tikzcd}
\end{equation}
\begin{enumerate}[(i)]
\item 
The $R$-module homomorphisms 
$\theta$, $\theta_1$ and $\theta_2$ of
\eqref{diag:thetas} are isomorphisms.
\item 
Both $\gr_{\mI} (\pi_1 ^\vee ) $and
$\gr_{F}  (\pi_2 ^\vee ) $
are Cohen-Macaulay
$\grL	 $-modules
of grade $2f $.
In particular, 
$\pi_1 ^\vee  $
and
$\pi_2 ^\vee  $
are Cohen-Macaulay $\Lambda $-modules of grade $2f $.
\item The $\mI$-adic filtration
on $ \pi ^\vee  $ induces
the $\mI$-adic filtration
on $ \pi _2 ^\vee  $.
\item The sequence 
$0 \to 
\gr_{\mI} ( \pi _2 ^\vee ) \to 
\gr_{\mI} ( \pi  ^\vee ) \to 
\gr_{\mI} ( \pi _1 ^\vee ) \to 
0$
of graded $\grL $-modules
with compatible $H $-action is split exact.
\end{enumerate}
\end{prop}

\begin{remark}
The reader should keep in mind that
the splitting in (iv) is in no way canonical:
in the case $\repr$ reducible, for example,
it depends on the choice of complementary subspaces
made in the statement.
\end{remark}

\begin{proof}
(i)
We follow the proof of \cite[Proposition~3.2.2]{BHHMS4}.
Set
$V_1(\ell) \defeq \Vell[\pi_1, \ell]$.
To make the arguments uniform, 
even when $\repr$ is irreducible
we let $V_2(\ell) \defeq V_2$,
for $0 \le \ell \le f$.
Note that we allow for either
$V_1(\ell)$ or $V_2(\ell)$ to be zero.

We know from \Cref{thm:N-to-gr-iso}
that $\theta$ is an isomorphism,
which implies that
$\theta_1$ is surjective and
$\theta_2$ is injective.
Remember that the 
characteristic cycle of 
\Cref{def:characteristic-cycle}
is additive by \cite[Lemma~3.3.4.2]{BHHMS2},
and so 
\begin{equation}
\mathcal{Z} (N_1) \ge \mathcal{Z} (\gr_{\mI} ( \pi _1 ^\vee )), \quad	
\mathcal{Z} (N_2) \le \mathcal{Z} (\gr_{F} ( \pi _2 ^\vee ))
\end{equation}
as divisors.

Let $\widetilde{\pi}_2 \subseteq \pi$
be the subrepresentation of 
\Cref{def:conjugate-srep} for
$\pi' = \pi_1$.
We have
\begin{equation}
	\label{eq:Z-Z-application-2}
\mathcal{Z}(\gr_{F} (\pi_2 ^\vee ))=
\mathcal{Z}(\gr_{\mI} (\widetilde{\pi}_2^\vee))
\end{equation}
by \Cref{lemma:Z-equal-Z}.
As in Step 2 of the proof of
\cite[Proposition~3.2.2]{BHHMS4},
we can construct the following commutative diagram with exact rows:
\begin{equation*}
\begin{tikzcd}
	0 \ar[r] & \gr(\E( \pi _{1} ^\vee ) ) \ar[r] \ar[d, hook]  & 
	\gr(\E( \pi ^\vee )) \ar[d, "\wr"] \ar[r]   & \gr(\widetilde{ \pi }_2 ^\vee \otimes_{\F}  (\mytwist)  ) \ar[dd]  \ar[r] & 0 \\
	0 \ar[r] & \Eg( \gr_{\mI} (\pi _{1} ^\vee )) ) \ar[r] \ar[d, hook]  & 
	 \Eg( \gr_{\mI} (\pi ^\vee )) ) \ar[d, equals]  & \\
	0  \ar[r] & \Eg (N_1) \ar[r] & 
	\ar[r]\Eg (N) \ar[r] & 
	\ar[r] \Eg (N_2) \ar[r] & 0 .
\end{tikzcd}
\end{equation*}
Again as in Step 2 of the proof of \cite[Proposition~3.2.2]{BHHMS4},
we find that 
$\Eg(N_2)\otimes_{\grL} \F  $
is a subquotient of 
$\bigoplus_{i=0} ^f \gr_{\mI} 
(\widetilde{\pi}_2 ^\vee )_{-i} \otimes_{\F} 
(\mytwist) $,
as $H $-modules.

Now, define the graded $\grL$-module
\begin{equation}
	\label{eq:N'-def}
{N}'_2 \defeq \bigoplus_{\ell=0} ^f 
\left( V_2(f- \ell ) \otimes_{\F}  \Nell \right), 
\end{equation}
so that we have 
an isomorphism
\begin{equation}
	\label{eq:N-hat-raison-d'être}
	\Eg (N_2)\cong {N}'_2 \otimes_{\F}  \mytwist
\end{equation}
of graded $\grL $-modules with compatible 
$H $-action
by \cite[Prop. 3.3.1.10]{BHHMS2}
and by \cite[Corollary 3.1.2]{BHHMS4}.

\Cref{cor:direction} applied to 
$\pi_i = \pi$ and $\pi_1$ 
yields isomorphisms
\begin{align}
		\label{eq:pi-and-pi-1-iso}
& (\pi^{I_1} \hookrightarrow \pi ^{K_1})
\xrightarrow[\sim]{ \iota _ \pi }
\bigoplus_{\ell=0} ^f
\left( V\otimes_{\F}  
\DEll \right),
& (\pi_1^{I_1} \hookrightarrow \pi_1 ^{K_1})
\xrightarrow[\sim]{ \iota _ \pi }
\bigoplus_{\ell=0} ^f
\left( V_1(\ell) \otimes_{\F}  
\DEll \right).
\end{align}
Moreover, if we set 
$\widetilde{V}_2(\ell) \defeq \Vell[\widetilde{\pi}_2,\ell]$
for $0\le{\ell}\le f$,
then \Cref{cor:direction}
applied to $\pi' = \widetilde{\pi}_2$ 
yields the isomorphism
\begin{equation}
		\label{eq:pi-tilde-2}
		(\widetilde{ \pi}_2 ^{I_1} \hookrightarrow\widetilde{ \pi}_2 ^{K_1})
\xrightarrow[\sim]{ \iota _ \pi }
		  \bigoplus_{\ell=0} ^f
\left( \widetilde{V}_2(\ell) \otimes_{\F}  
	\DEll \right).
\end{equation}
Define the graded
$\grL $-module with compatible $H $-action
\begin{equation}
	\label{eq:N-widetilde-def}
	\widetilde{N}_2 \defeq \bigoplus_{\ell=0} ^f 
\left( \widetilde{V}_2(\ell)^\vee \otimes_{\F} \Nell	 \right),
\end{equation}
and let
\begin{equation}
	\label{eq:N-tilde-raison-d'être}
\widetilde{\theta}_2  \colon
\widetilde{N}_2 \twoheadrightarrow  \gr_{\mI}  (\widetilde{\pi}_2^\vee )
\end{equation}
be the surjection of graded $\grL$-modules 
\eqref{eq:surjection_Ni} for $\pi_i=\widetilde{\pi}_2$.

We claim that there exists an injective homomorphism of $H $-modules
$({N}'_2)_0 
\hookrightarrow (\widetilde{N}_2)_0 $.
Indeed, the fact that 
\[
	\F\otimes_{\grL} 
	E^{2f}_{\grL}(N_2)
	\overset{\eqref{eq:N-hat-raison-d'être}}{\cong} 
	({N}'_2)_0 \otimes_{\F} \mytwist
	\cong 
	(\F\otimes_{\grL} 
	N'_2) \otimes_{\F} \mytwist
\]
is a subquotient (as $H$-modules)
of $\bigoplus_{i=0} ^f \gr_{\mI} (\widetilde{\pi}_2 ^\vee )_{-i} \otimes_{\F} \mytwist $,
hence a fortiori of
\[
 \bigoplus_{i=0} ^f	(\widetilde{N}_2)_{-i} \otimes_{\F} \mytwist
\]
using 
\eqref{eq:N-tilde-raison-d'être},
implies
the existence of an injection
\begin{equation}
	\label{eq:temp-injection}
	({N}'_2)_0 \hookrightarrow 
	\bigoplus_{i = 0 } ^f (\widetilde{N}_2)_{-i}
\end{equation}
of $H $-modules
($\F[H]$ is a semisimple algebra,
so being a subquotient is equivalent 
to being a submodule). 

It is enough to show that 
all the compositions
\[
	({N}'_2)_0 
\xhookrightarrow{\eqref{eq:temp-injection}}
\bigoplus_{i = 0 } ^f (\widetilde{N}_2)_{-i} \twoheadrightarrow (\widetilde{N}_2)_{-i} \subseteq 
N_{-i}
\]
vanish,
for $i > 0 $.
If this were not the case, 
in particular we would find a common
Jordan-H\"older factor in
$ ({N}'_2)_0 $
and $N_{-i} $. 
We claim that $N_{-i} $ has no common 
Jordan-H\"older factor with
$N_{-j} $ for $0\le i \neq j \le f$,
which leads to a contradiction.
This follows from the fact that
$\bigoplus_{i=0} ^f (\Nnull)_{-i}$
is multiplicity free,
which is in turn a consequence of
\cite[Lemma~2.3.7]{BHHMS4}
with $n=f+1$
(here we are using that $\repr$ is
$(f+1)$-generic).

For $0\le{\ell}\le f$,
we let $r_1(\ell) \defeq \rell[\pi_1,\ell]$,
$\widetilde{r}_2(\ell) \defeq \rell[\widetilde{\pi}_2,\ell]$,
and $r_2(\ell) \defeq r- r_1(\ell).$
If we use that
\begin{align}
	\label{eq:N0-D1-iso}
&({N}'_2)_0 \cong \bigoplus_{\ell=0} ^f
\bigl( V_2(f-\ell) \otimes_{\F} \DOneEll ^\vee \bigr), 
&(\widetilde{N}_2)_0 \cong \bigoplus_{\ell=0} ^f
\bigl( \widetilde{V}_2(\ell)^\vee  \otimes_{\F} \DOneEll ^\vee \bigr),
\end{align}
and the fact that $\DOne $ is multiplicity-free,
then it follows
from the injection 
$ ({N}'_2)_0 \hookrightarrow 
(\widetilde{N}_2)_0 $
of $H $-modules
that 
\begin{equation}
	\label{eq:r2-tilde-r2-inequality}
\widetilde{r}_2({\ell}) \ge r_2(f-{\ell})
= r-r_1(f-\ell),
\end{equation}
	
for $0 \le \ell \le f$.

Now, we look at the length of
$\socR (\widetilde{ \pi }_2)$:
\begin{align}
	\lgR \left(\socR (\widetilde{ \pi }_2)	\right) 
& \overset{\eqref{eq:pi-tilde-2}}{=} \sum_{\ell=0} ^f \widetilde{r}_2( \ell) \lgR\left(\socR (\DZeroEll)\right)
\nonumber \\
	\label{eq:ineq-will-be-eq}
&\overset{ \eqref{eq:r2-tilde-r2-inequality}}{\ge }
 \sum_{\ell= 0}  ^f {r}_2(f- \ell) \lgR\left(\socR (\DZeroEll)\right) 
\\
&\overset{(*)}{=} 
 \sum_{\ell'= 0}  ^f {r}_2( \ell') \lgR\left(\socR (\DZeroEll[\ell'])\right)
\nonumber\\
& \overset{ \eqref{eq:r2-tilde-r2-inequality}}{=} \sum_{\ell'= 0}  ^f(r- {r}_1( \ell')) \lgR(\socR (\DZeroEll[\ell']))
\nonumber\\
& \overset{ \eqref{eq:pi-and-pi-1-iso}}{=} \lgR\left(\socR ( \pi )\right)
- \lgR\left(\socR ( \pi_1 )\right)\nonumber,
\end{align}
where in $(*)$ we have used the substitution $\ell' \defeq f-\ell $,
and used the fact that 
\[
\lgR(\socR (\DZeroEll))=
\lgR(\socR (\DZeroEll[f- \ell])).
\]
Now, remember that
\eqref{eq:Z-Z-application-2}
together with \Cref{prop:main-memoire}(i)
imply
\begin{equation}
	\label{eq:D-pi2-pi2-tilde}
\dim_{\F (\!(X)\!)  } \Dvee[ \pi _2]  
=
\dim_{\F (\!(X)\!)  } \Dvee[\widetilde{ \pi }_2].  
\end{equation}
From \Cref{prop:main-memoire}(ii) 
we obtain the equalities
\begin{align*}
\lgR(\socR (\widetilde{ \pi }_2)	) & = 
\dim_{\F (\!(X)\!)  } \Dvee[\widetilde{ \pi }_2]  
 \overset{ \eqref{eq:D-pi2-pi2-tilde}}{=}
\dim_{\F (\!(X)\!)  } \Dvee[ \pi _2], \\
\lgR\!(\socR ( \pi )) - \lgR\!(\socR ( \pi_1 ))
&=
\dim_{\F (\!(X)\!)  } \Dvee[ \pi]-
\dim_{\F (\!(X)\!)  } \Dvee[ \pi_1].
\end{align*}
By the exactness of the functor $\Dvee[-] $, 
the two quantities on the right-hand side
coincide.
So, we have equality in \eqref{eq:ineq-will-be-eq},
which forces
$ \widetilde{r}_2(\ell) = r_2(f-{\ell}) $
for all $0\le{\ell}\le f$.
Hence, we have the following (abstract) 
isomorphisms of $R $-modules:
\begin{equation}
	\label{eq:conjugate-gateway}
N'_2
\overset{ \eqref{eq:N'-def}}{ \cong }
\bigoplus_{\ell=0} ^f 
\Nell^{\oplus r_2(f-\ell)} 
\overset{ \eqref{eq:N-widetilde-def}}{ \cong }
\widetilde{N}_2 . 
\end{equation}

Then, we consider
\begin{equation}
	\label{eq:Z-chain-of-inequalities}
\mathcal{Z}(\gr_F( \pi _2 ^\vee  ))
\overset{\eqref{diag:thetas}}{\ge} \mathcal{Z} (N_2) = \mathcal{Z} ({N}'_2) =
\mathcal{Z} (\widetilde{N}_2) \overset{\widetilde{ \theta} _2}{\ge} \mathcal{Z} (\gr_{\mI} (\widetilde{ \pi }_2 ^\vee )),
\end{equation}
where the first equality comes from 
\cite[Theorem~3.3.4.5]{BHHMS2},
together with \eqref{eq:N-hat-raison-d'être},
and the second one from ${N}'_2 \cong\widetilde{N}_2  $.
From \eqref{eq:Z-Z-application-2}
we deduce that equality holds everywhere in \eqref{eq:Z-chain-of-inequalities}.

Then, from \eqref{diag:thetas} and from 
the additivity of $\mathcal{Z}$
we also deduce 
\begin{equation}
	\label{eq:Z-1-equality}
\mathcal{Z}(N_1)=\mathcal{Z}(\gr_{\mI} (\pi_1 ^\vee )).
\end{equation}
Now, we claim that $N$
is \emph{pure} in the sense of
\cite[Definition~III.4.2.7]{LvO}.
This is because it is essentially self-dual
of grade $2f$
(cf.\ the proof of \cite[Theorem~2.1.2]{BHHMS4}),
and so we can apply
\cite[Definition~III.4.2.7]{LvO}(i)
with $R = \grL$, $M = N$.

Set $K \defeq \ker (\theta_1)$.
Then, $K$ is a submodule of $N_1$,
which in turn injects inside $N$
(the bottom row of \eqref{diag:thetas}
is split).
So, \cite[Definition~III.4.2.9]{LvO}
implies that $K$ is of grade $2f$ over
$\grL$,
hence of grade $0$ over $\overline{R}$
by the second statement in 
\cite[Lemma~3.3.1.9]{BHHMS2}.
In particular, if $K \neq 0$,
then 
we would have 
$\Ext^0_{\overline{R}}(K, \overline{R})=
\Hom_{\overline{R}}(K, \overline{R})\neq 0$,
so fix a nonzero homomorphism 
$h \colon K \to \overline{R}.$
On the other hand,
using \eqref{eq:Z-1-equality}
and the additivity of $\mathcal{Z}$
we see that $\mathcal{Z}(K) = 0$.
In particular, for all minimal primes
$\mathfrak{q}$ of $\overline{R}$,
we have
$\Hom_{\overline{R}}(K, \overline{R})_{\mathfrak{q}} 
\cong 
\Hom_{\overline{R}}(K_{\mathfrak{q}} , \overline{R}_{\mathfrak{q}})
=0$,
hence $h_{\mathfrak{q}} = 0$.
In particular, 
$\im(h)_{\mathfrak{q}} = \im(h_{\mathfrak{q}} ) = 0$
by the exactness of localisation,
and so $I_h \defeq \im(h) \subseteq \overline{R}$
is an ideal of $\overline{R}$
whose localisation at all minimal primes is zero.
But $\overline{R}$ is reduced, 
and so the localisation $\overline{R}_{\mathfrak{q}} $
is a field (being artinian local and reduced),
which implies that the localisation map
$ \overline{R} \to \overline{R}_{\mathfrak{q}} $
has kernel $\mathfrak{q}$.
This implies that $I_h \subseteq \mathfrak{q}$
for all $\mathfrak{q} \subseteq \overline{R}$
minimal, hence $I_h \subseteq 
\bigcap \mathfrak{q} = \sqrt{0} =0$,
again because $\overline{R}$ is reduced.

In conclusion,
$I_h = \im(h) = 0$, so $h = 0$, contradition.
We then must have $K=0$,
i.e.\ $\theta_1 \colon N_1 \twoheadrightarrow \gr_{\mI} ( \pi _1 ^\vee ) $
is an isomorphism.
This implies
$\theta_2 \colon N_2 \cong \gr_F( \pi _2 ^\vee ) $
by \eqref{diag:thetas}
and finishes the proof of (i).

(ii)
By the proof of 
\cite[Theorem~2.1.2]{BHHMS4},
$N$ is Cohen-Macaulay of grade $2f$,
which implies that its direct summands 
$N_1 \cong \grm(\pi_1^\vee ) $ and 
$N_2 \cong \gr_F(\pi_2^\vee )$ 
both are Cohen-Macaulay.
As a consequence of 
\cite[Proposition~III.2.2.4]{LvO},
$\pi_1^\vee $ and
$\pi_2^\vee $ are Cohen-Macaulay.

(iii)
The proof of \cite[Corollary~3.2.4]{BHHMS4} goes through,
we report it here for the reader's convenience.
The inclusion
$ \mathfrak{m}^n \pi _2 ^\vee \subseteq 
\pi _2 ^\vee \cap \mathfrak{m}^n \pi ^\vee = F_{-n} \pi _2 ^\vee $
induces a natural morphism
\[
	\kappa:
	\gr_{ \mathfrak{m}} ( \pi _2 ^\vee )
	\to
	\gr_{F} ( \pi _2 ^\vee )
	\cong N_2,
\]
where the last isomorphism was established in the proof of \Cref{prop:main} using 
\eqref{diag:thetas}.
We easily see that $ \kappa  $ 
is surjective in degree $0 $.
Since $N_2 $ is generated by its degree $0 $
part, $ \kappa 	 $ is surjective,
and it follows from \cite[Theorem~I.4.2.4(5)]{LvO}
(applied with $L=M= \pi_2 ^\vee  $ and $N=0 $)
that $ \mathfrak{m}^n \pi _2 ^\vee  = F_{-n} \pi _2 ^\vee  $ 
for all $n \ge 0 $.

(iv)
This follows from the split exactness of 
\[
	0 \to N_2 \to N \to N_1 \to 0.\qedhere
\]
\end{proof}

\subsection{Finite length}
	\label{sec:fl}
Keep the hypotheses
\ref{hypothesis:i} to
\ref{hypothesis:iv}
of \Cref{sec:hypotheses}.
With the following theorem,
which generalises 
\cite[Theorem~3.2.3]{BHHMS4} to the case
$r \ge 1$,
we prove that $\pi$ has finite length.
Unlike \emph{loc.\ cit.}, we will not 
assume that $\repr$ is reducible,
because the result is new even in the
irreducible case.
\begin{thm}
	\label{thm:finite-length-yay}
Assume that $\repr$ is
$\max \{9, 2f+1\} $-generic,
and let $\pi'$ be a subrepresentation of $\pi$.
	\begin{enumerate}[(i)]
	\item The representation $\pi'$ is generated by its $\GLring $-socle.	
	\item 
Let $\pi'_0 \subseteq \pi'$ be a subrepresentation of $\pi'$.

If $\repr$ is reducible, then
$\lg_{\GLfield}(\pi'/\pi'_0) \le  \sum_{\ell=0} ^f (\rell[\pi',\ell]-\rell[\pi'_0,\ell])
\le r \cdot (f+1). $

If $\repr$ is irreducible, then 
$\lg_{\GLfield}(\pi'/\pi'_0) \le \rnull[\pi']-\rnull[\pi'_0]\le r. $
\end{enumerate}
In particular, $\pi$
has finite length,
bounded above by $r\cdot (f+1)$ in the 
reducible case, and by $r$ in the
irreducible case.
\end{thm}
\begin{proof}
(i)
Let $ \pi_1' $ be the subrepresentation 
of $ \pi' $ generated by $\socR ( \pi') $.
In particular, 
$\socR ( \pi') =\socR ( \pi _1') $,
so they give the same
$\Vell[\pi',\ell]= \Vell[\pi_1',\ell]$
by the last part of 
\Cref{rem:V-are-commensurable},
which implies by \Cref{cor:direction}
that the subdiagrams
$(\pi^{\prime I_1} \hookrightarrow \pi^{\prime K_1}) $
and
$(\pi_1^{\prime I_1} \hookrightarrow \pi_1^{\prime K_1}) $
coincide.

Consider the commutative diagram 
\eqref{diag:thetas-square},
with $\pi_2=\pi'$ and $\pi_1=\pi_1'$:
 \[
 \begin{tikzcd}
\bigoplus_{\ell=0}^f
\left( \Vell[\pi',\ell] ^\vee \otimes_{\F} \Nell \right)
\ar[d,   " \theta'"'] 
\ar[r, equals] 
&
\bigoplus_{\ell=0}^f
\left( \Vell[\pi_1',\ell] ^\vee \otimes_{\F} \Nell \right)
\ar[d, " \theta_1'" ]
\\ 
 \gr_{\mI} (\pi^{\prime\vee}) \ar[r, two heads, "\psi"'] 
 &
	\gr_{\mI} (\pi_1^{\prime\vee}).
\end{tikzcd}
 \]
It follows from 
\Cref{prop:main}(i)
with $\pi_1 = \pi'$
(resp.\ $\pi_1 = \pi_1'$)
that $\theta'$ (resp.\ $\theta'_1$)
is an isomorphism, and so 
$\psi \colon \gr_{\mI} (\pi^{\prime\vee})
\cong\gr_{\mI} (\pi_1^{\prime\vee})$
is an isomorphism, from which we deduce
(for dimension reasons) that
$\pi^{\prime\vee}/\mI ^n \xrightarrow{\sim} \pi_1^{\prime\vee}/\mI ^n $ for all $n \ge 1 $,
and hence $\pi' = \pi_1' $.
This proves (i).

(ii) We only consider the case $\repr$
reducible for simplicity.
Let $N \defeq \sum_{\ell=0} ^f(\rell[\pi']- \rell[\pi'_0])$,
and suppose for the sake of a contradiction
that we have a strictly increasing chain
\[
 \pi'_0 \subsetneq \pi_1' \subsetneq \dots \subsetneq \pi'_{N} \subsetneq  \pi' 
\]
of subrepresentations.
Let $V'(\ell) \defeq \Vell[\pi',\ell] $,
and let 
$ V'_i(\ell) \defeq \Vell[\pi'_i,\ell]\subseteq V'(\ell)$
for $0\le i \le N$.
By (i), the chain 
\begin{equation}
	\label{eq:soc-chain-red}
\begin{tikzcd}[column sep= 1.5ex]
	\socR (\pi'_0 )\ar[r, "\subsetneq", phantom] &  
	\cdots\ar[r, "\subsetneq", phantom] &  
	\socR (\pi'_{N}) \ar[r, "\subsetneq", phantom] &  
	\socR (\pi')
\end{tikzcd}
\end{equation}
must also be strictly increasing. 
But we know from \Cref{cor:direction} that 
\begin{equation}
	\label{eq:soc-iso-red}
\begin{tikzcd}
\socR( \pi'_i )
\ar[r, "\iota_\pi", "\sim"'] &
\left( \displaystyle\bigoplus_{ \ell=0} ^f V'_i(\ell) \otimes_{\F}  
       \socR 
       \left( 
       \DZeroEll
        \right) \right)
\end{tikzcd}
\end{equation}
for all $0\le i\le N$,
so we have the inclusions
\begin{equation}
	\label{eq:inclusion-chain}
V'_i(\ell)\subseteq V'_{i+1} (\ell) \subseteq V'(\ell),
\end{equation}
for all $0\le \ell\le f$ and
$0\le i< N$.
Fix one of the inclusions in \eqref{eq:inclusion-chain}.
Because of \eqref{eq:soc-chain-red} and
\eqref{eq:soc-iso-red},
there exists at least one
$\ell \in \{0, \dots f\} $ such that
that inclusion is strict.
In particular, when we sum over $\ell$, 
we obtain a strictly increasing sequence
\[
\bigoplus_{ \ell=0} ^f V'_0( \ell)	
 \subsetneq
\bigoplus_{ \ell=0} ^f V'_1( \ell)	
\subsetneq
\cdots
\subsetneq
\bigoplus_{ \ell=0} ^f V'_{N}( \ell)	
\subsetneq
\bigoplus_{ \ell=0} ^f V'(\ell)
\]
of $N+2$  vector subspaces of $\bigoplus_{\ell=0} ^f V'(\ell)$,
contradicting
\[
	\dim_{\F}\left( \frac{\bigoplus_{\ell=0} ^f V'(\ell)}{\bigoplus_{\ell=0} ^f V_0'(\ell)} \right)
	=  \sum_{\ell=0} ^f (\rell-\rell[\pi'_0,\ell])=N.\qedhere
\]
\end{proof}

We conclude the section with a corollary
of \Cref{thm:finite-length-yay}.
\begin{corollary}
	\label{cor:srep-datum-uniquely}
Assume that $\repr$ is
$\max \{9, 2f+1\} $-generic.
\begin{enumerate}[(i)]
\item 
If $\repr$ is reducible,
a subrepresentation $ \pi' $
of $\pi  $
is uniquely determined
by the $(f+1)$-tuple 
$(\Vell[\pi',\ell])_{0\le \ell\le f} $
of vector subspaces of $V$.

If $\repr$ is irreducible,
a subrepresentation $ \pi' $
of $\pi  $
is uniquely determined
by the vector subspace $\Vnull \subseteq V$.

\item 
Conversely, 
if $\repr$ is reducible, suppose that
for any $(f+1)$-tuple
$(V'(\ell))_{0\le \ell\le f} $ of vector subspaces
of $V$
there exists a (necessarily unique) subrepresentation 
$\pi' \subseteq \pi$ such that
$V'(\ell)= \Vell[\pi', \ell]$.
Then, $\pi$ is
of the form 
$\bigoplus_{\ell=0} ^f 
\pi_\ell^{\oplus r}$,
for some $\GLfield$-representations 
$\pi_\ell$ such that 
$\pi_\ell^{K_1}\cong \DZeroEll$.

Similarly, if
$\repr$ is irreducible, suppose that
for any vector subspace
$V'$ of $V$
there exists a (necessarily unique) subrepresentation 
$\pi' \subseteq \pi$ such that
$V'= \Vnull[\pi']$.
Then, $\pi$ is of the form 
$ \pi^{\prime \oplus r}$,
for some $\GLfield$-representation
$\pi'$ such that 
$\pi^{\prime K_1}\cong \DZero$.
\end{enumerate}
\end{corollary}
\begin{remark}
	\label{rem:srep-datum-uniquely}
In (i) of
\Cref{cor:srep-datum-uniquely},
note that $\pi'$ is uniquely determined only
as a subrepresentation of $\pi$,
which remains a mysterious object.
In (ii) of \emph{loc.\ cit.},
note that the assumption is very strong,
and we are not in the position
to prove it for cases which are 
not already known in the literature
(such as $r=1$, $\repr$ irreducible
or $r=1$, $f=2$, $\repr$ reducible).
\end{remark}
	\label{rem:not-V-to-pi1}
\begin{proof}[Proof of \Cref{cor:srep-datum-uniquely}]
(i)
A direct consequence of
\Cref{thm:finite-length-yay}(i) and 
\Cref{cor:direction}.

(ii)
We keep to the reducible case for simplicity:
for a given linear subspace 
$W \subseteq V$
and a given integer $\ell \in \{0, \dots,f\}$,
let $\pi_{\ell} (W)$ be the unique
subrepresentation of $\pi$ with 
\[
\Vwt[\pi_{\ell} (W),\ell']=
\begin{cases}
W & \text{if }\ell' = \ell, \\
0 & \text{if }\ell' \neq \ell,
\end{cases}
\]
for $\ell' \in \{0, \dots, f\}$.
It is a consequence of \Cref{cor:direction} 
that 
$
\iota_\pi \colon \socR\pi_\ell(W) \xrightarrow{\sim} 
W \otimes _{\F} \socR\DZeroEll[\ell].$
From this
(and from the uniqueness in (i)),
we deduce that 
\begin{equation}
	\label{eq:pi-W-socle}
\pi_{\ell_1}(W_1) \cap \pi_{\ell_2} (W_2)
= \begin{cases}
\pi_{\ell_1} (W_1 \cap W_2) & \text{if }\ell_1 = \ell_2, \\
0 & \text{otherwise.}
\end{cases}
\end{equation}
Now, take $W$ to be a line $L \subseteq V$.
By arguing as in step 1 of the proof of
\Cref{thm:principal-series},
one proves that the $\pi_{\ell} (L)$
are isomorphic as $L \subseteq V$ varies.
Call $\pi_\ell$
(a fixed representative of)
 this isomorphism class.

Let $e_1, \dots, e_r$ be the standard basis
of $V = \F^r$.
\Cref{thm:finite-length-yay}(i)
implies that 
$\pi_\ell(V) = 
\sum_{i=1} ^{r}\pi_\ell(\F e_i)$,
in particular it only has one isotypic component,
of type $\pi_\ell$.
By looking at $\lgR(\socR \pi_\ell(V))$,
we conclude that 
the subrepresentation 
$\pi_\ell(V) \subseteq \pi$
is isomorphic to $\pi_\ell^{\oplus r}$.
\Cref{thm:finite-length-yay}(i)
implies that 
$\sum_{0 \le \ell \le f} \pi_\ell(V)= \pi$,
and we claim that this is a direct sum.
For this, it is enough to show that $\pi_\ell$
is not isomorphic to $\pi_{\ell'} $
for $\ell \neq \ell'$, and this is 
already true at the level of $\GLring$-socles.
\end{proof}

We conclude the section with 
the following result, 
which generalises 
\cite[Corollary~3.2.5]{BHHMS4}.
Note that the proof is substantially
more involved,
due to the presence of multiplicities
to take care of.
\begin{prop}[]
	\label{cor:pi1-pi2-exact}
Assume that $\repr$ is
$\max \{9, 2f+1\} $-generic.
Let $ \pi _1 \subseteq  \pi _2 $
be subrepresentations of $ \pi  $.
\begin{enumerate}[(i)]
\item 
For any $n \ge 1 $,
the sequence
\begin{equation}
	\label{eq:pi1-pi2-exact}
	0 \to
	\pi _1[ \mathfrak{m}^n] \to
	\pi _2[ \mathfrak{m}^n] \to
	(\pi _2/ \pi _1)[ \mathfrak{m}^n] \to
	0
\end{equation}
of $ \Lambda  $-modules is exact.
\item
This short exact sequence splits as $I $-representations 
if $\repr $ is $2n $-generic,
in particular if $n \le 3 $.
\end{enumerate}
\end{prop}
\begin{proof}
(i)
We always have an exact sequence
\[
	0 \to
	\pi _1[ \mathfrak{m}^n] \to
	\pi _2[ \mathfrak{m}^n] \to
	(\pi _2/ \pi _1)[ \mathfrak{m}^n],
\]
we only need to show the surjectivity of the last map.
We first treat the special case $ \pi _2 = \pi  $.
For dimension reasons,
it is enough to show that
\begin{equation}
		\label{eq:pi1-pi2-dim}
\dim_{\F}
(	(\pi / \pi _1)[ \mathfrak{m}^n])
=
\dim_{\F} (\pi [ \mathfrak{m}^n])-
\dim_{\F} (\pi _1[ \mathfrak{m}^n]).
\end{equation}
but this follows from
\[
	\gr_{ \mathfrak{m}} ( \pi  ^\vee ) \cong 
	\gr_{ \mathfrak{m}} ( \pi _1 ^\vee )  \oplus
	\gr_{ \mathfrak{m}} ( (\pi /\pi _1) ^\vee ) ,
\]
which is \Cref{prop:main}(iv).

Onto the general case.
We know that $ \pi [ \mathfrak{m}^n] \to ( \pi / \pi _2)[ \mathfrak{m}^n] $
is surjective by the previous case
and factors through $(\pi/\pi_1)[\mI^{n}]$,
hence we have a surjective morphism $(\pi/\pi _1) [\mI^{n}] \twoheadrightarrow
(\pi/\pi _2)[\mI^{n}]$,
and a short exact sequence
\[
0 \to
( \pi_2/ \pi _1) [\mI^{n}] \to 
( \pi/ \pi _1) [\mI^{n}] \to ( \pi / \pi _2)[\mI^{n}] \to 0.
\]
Counting dimensions, we see that
\begin{align*}
\dim_{\F}
\big(	(\pi_2 / \pi _1)[ \mathfrak{m}^n]\big)
=&
\dim_{\F} \big((\pi/\pi_1 )[ \mathfrak{m}^n]\big)-
\dim_{\F} \big((\pi/\pi _2)[ \mathfrak{m}^n]\big)
\\
\overset{\eqref{eq:pi1-pi2-dim}}{=}&
\big( 
\dim_{\F} (\pi[ \mathfrak{m}^n])-
\dim_{\F} (\pi_1 [ \mathfrak{m}^n])
 \big) 
-
\big(
\dim_{\F} (\pi[ \mathfrak{m}^n])-
\dim_{\F} (\pi _2[ \mathfrak{m}^n])
  \big) 
\\
=& \dim_{\F} \big(\pi_2 [ \mathfrak{m}^n]\big)-
\dim_{\F} \big(\pi _1[ \mathfrak{m}^n]\big),
\end{align*}
from which (i) follows.

(ii)
It suffices to show
that $ \pi _1 [ \mathfrak{m}^n] $
is a direct summand of $ \pi [ \mathfrak{m}^n] $:
if we find a direct sum decomposition 
of $I $-representations
of the form
$ \pi  [ \mathfrak{m}^n]= \pi _1 [ \mathfrak{m}^n] \oplus W   $,
then clearly we also have
$ \pi_2  [ \mathfrak{m}^n]= \pi _1 [ \mathfrak{m}^n] \oplus (W \cap    \pi_2  [ \mathfrak{m}^n])$:
so, we may assume $\pi_2= \pi$.

\paragraph{Step 1.}
We start with some setup.
For $n \ge 1$, denote by
$\mathcal{I}^{(n)}$
the $H$-stable graded ideal 
$(y_j^n, z_j^n, h_j\mid 0\le j < f)$
of $\grL$
(cf.\ the commutator relations 
\eqref{eq:grL-commutator-relations}).
We consider
the finite-dimensional $I$-representation
$ \tau ^{(n)} \defeq \V\otimes_{\F}  
\left( \bigoplus_{ \lambda \in \mathscr{P}} \tau _ \lambda ^{(n)}  \right)$
of \cite[Lemma~2.4.1]{BHHMS4}
and the isomorphism of graded $H$-modules
\begin{equation}
	\label{eq:kappa-lambda}
\kappa_{\lambda}  \colon \grm((\tau_{\lambda} ^{(n)})^\vee )
\xrightarrow{\sim} \chi_\lambda ^{-1} \otimes
R/(\mathcal{I}^{(n)}+\mathfrak{a}(\lambda))
\end{equation}
of \emph{loc.\ cit.} for 
$\lambda \in \mathscr{P}$.
If we set 
If $\Nnull$ is the graded $\grL$-module
of \eqref{eq:N-null},
then we set
$N \defeq \V^\vee \otimes_{\F} \Nnull$
and we define from \eqref{eq:kappa-lambda}
 the isomorphism
\[
\kappa  \colon \grm((\tau^{(n)})^\vee )
\xrightarrow{\sim}  N/\mathcal{I}^{(n)}N
\]
of graded $H$-modules by
$\kappa \defeq 
\bigoplus_{\lambda\in \mathscr{P}}\id_{V^\vee } \otimes \kappa_\lambda$.

We let $i: \tau^{(n)} \hookrightarrow \pi$
be the $I$-equivariant injection of
\cite[Lemma~2.4.2]{BHHMS4},
so that in particular the composition
\begin{equation}
	\label{eq:kappa-and-i}
	\begin{tikzcd}
		N \ar[r, "\eqref{eq:surjection_Ni}","\sim"'] &  
\grm(\pi^\vee ) \ar[r, "\grm(i^\vee )", two heads] &  
\grm((\tau^{(n)})^\vee ) \ar[r,"\kappa", "\sim"'] &  
N/\mathcal{I}^{(n)}N
\end{tikzcd}
\end{equation}
is the projection modulo $\mathcal{I}^{(n)}$.

Now, notice that the $\grL$-module $N$
factors as an $\overline{R}$-module by construction, and that
$\mathcal{I}^{(n)}\overline{R}
\subseteq \overline{\mI}^n\overline{R}$,
where $\overline{\mI}$ is the unique
graded maximal ideal of $\grL$.
In particular, 
$\mathcal{I}^{(n)}N \subseteq \overline{\mI}^nN$,
hence
\eqref{eq:kappa-and-i} becomes an isomorphism
in degrees $\ge -(n-1)$,
and so $i$ restricts to
an isomorphism
$i[\mI^n] \colon   \tau ^{(n)} [ \mathfrak{m}^n]\cong \pi [ \mathfrak{m}^n]$
for dimension reasons.

If we set
\begin{align}
	\label{eq:tau1-N1-def}
	\tau _1^{(n)} & \defeq \bigoplus_{ \ell=0} ^f
\left( 	\Vell[\pi_1,\ell] \otimes_{\F}  
	\left( \bigoplus_{ \lambda \in \mathscr{P}_{\ell} }  \tau _ \lambda ^{(n)} \right) \right),\\
	N_1 & \defeq \bigoplus_{ \ell=0} ^f \Bigl(  \Vell[\pi_1,\ell]^\vee  \otimes_{\F}  
	\Nell \Bigr),
\end{align}
then we can define
\[
	\kappa_{1}  \colon \grm((\tau_1 ^{(n)})^\vee )
	\xrightarrow{\sim}   N_1/\mathcal{I}^{(n)}N_1
\]
as 
$\bigoplus_{\lambda\in \mathscr{P}}
\left( \id_{\Vwt[\pi_1,\ell(\lambda)]} \otimes \kappa_\lambda \right)$,
and the commutativity of the diagram  
\begin{equation}
	\label{diag:kappa-kappa1}
\begin{tikzcd}
	\grm((\tau^{(n)})^\vee ) \ar[d, two heads]  \ar[r, "\kappa", "\sim"'] &  N/\mathcal{I}^{(n)}N\ar[d, two heads]\\
	\grm((\tau_1^{(n)})^\vee ) \ar[r, "\kappa_1"', "\sim"] &  N_1/\mathcal{I}^{(n)}N_1
\end{tikzcd}
\end{equation}
is a routine check.
Here, the vertical arrows are constructed as follows:
if $i_\lambda \colon \Vwt[\pi_1,\ell(\lambda)]\subseteq \V$
is the natural inclusion,
then they are given by
$\bigoplus_{\lambda\in \mathscr{P}} i_\lambda^\vee \otimes \id$,
where the identity is on $\grm(\tau_\lambda^{(n)})$ 
for the  arrow on the left, and on 
$R/(\mathcal{I}^{(n)}+\mathfrak{a}(\lambda))$
for the arrow on the right.

\paragraph{Step 2.}
We claim that
$i \colon \tau^{(n)} \hookrightarrow \pi$
restricts to $i_1 \colon \tau_1^{(n)}\hookrightarrow \pi_1$,
giving in particular the commutative diagram
\begin{equation}
	\label{diag:i-i1}
\begin{tikzcd}
	\grm(\pi^\vee ) \ar["\grm(i^\vee )", r, two heads] \ar[d, two heads]  &  
	\grm((\tau^{(n)})^\vee )\ar[d, two heads] \\
	{\grm(\pi_1^\vee )} \ar["\grm(i_1^\vee )", r, two heads] &  
	{\grm((\tau_1^{(n)})^\vee ).}
\end{tikzcd}
\end{equation}
Here we depart from the proof 
of \cite[Corollary~3.2.5]{BHHMS4},
where there is no need to introduce $i_1$.

For this, we imitate the proof of
\cite[Lemma~2.4.2]{BHHMS4}.
We at least know that the image of
$i|_{\tau_1 ^{(n)}[\mI]}\colon \tau ^{(n)}[\mI]\to \pi$ 
is contained inside $\pi_1$,
by \Cref{cor:direction} and by
$i[\mI^n] \colon   \tau ^{(n)} [ \mathfrak{m}^n]\cong \pi [ \mathfrak{m}^n]$.
Then, we show that the restriction map
\begin{equation}
	\label{eq:tau-restriction}
	\Hom_{I/Z_{1}} (\tau_1^{(n)}, \pi_1) 
\xrightarrow{\sim}
\Hom_{I/Z_{1}} (\tau_1^{(n)}[\mI], \pi_1)
\end{equation}
is an isomorphism, and we do this by
proving that 
$\Ext^\varepsilon_{I/Z_{1}}(\tau_1^{(n)}/\tau_1^{(n)}[\mI], \pi_1)=0$,
for $\varepsilon \in \{0,1\} $.
With a dévissage, we reduce to showing
\begin{equation}
	\label{eq:step2-claim}
	\Ext^\varepsilon_{I/Z_1}(\chi_0, \pi_1)=0,
\end{equation}
for $\chi_0 \in \JH(\tau_1^{(n)}/\tau_1^{(n)}[\mI]) $
and for $\varepsilon \in \{0,1\} $.
But 
$\JH(\tau_1^{(n)}/\tau_1^{(n)}[\mI])\subseteq \JH(\tau^{(n)}/\tau^{(n)}[\mI])$
(the inclusion follows for example by introducing
\eqref{eq:tau2-def} 
and using the direct sum decomposion considered
in the line below),
and it is shown in the proof of
\cite[Lemma~2.4.2]{BHHMS4},
after (26), that
\[
	\JH(\tau^{(n)}/\tau^{(n)}[\mI]) \cap \JH(\pi[\mI])= \emptyset.
\]
In particular,
$\Ext^1_{I/Z_1}(\chi_0,\pi)=0$
by hypothesis \ref{hypothesis:iv},
and we can consider the long exact sequence
\[
0\to \pi_1^{I_1} \xhookrightarrow{j^{0}}
\pi^{I_1}\to
(\pi/\pi_1)^{I_1} \xrightarrow{0}  \Ext^1_{I_1/Z_1}(\F, \pi_1)
\xhookrightarrow{j^{1}} \Ext^1_{I_1/Z_1} (\F, \pi),
\]
where the transition map is $0$ by (i)
of the statement (with $n=1$).
For $\varepsilon \in \{0,1\} $, consider the decompositions into
$\chi$-isotypic components
\begin{equation*}
\Ext^{\varepsilon}_{I_1/Z_1} (\F,\pi_1)=\bigoplus_{\chi} 	\Ext^\varepsilon_{I/Z_1} (\chi,\pi_1),\quad
\Ext^{\varepsilon}_{I_1/Z_1} (\F,\pi)=\bigoplus_{\chi} 	\Ext^\varepsilon_{I/Z_1} (\chi,\pi),
\end{equation*}
where $\chi$ runs over all smooth characters
of $I/Z_1$ over $\F$.
In particular, taking $\chi=\chi_0$,
we know that $j^{\varepsilon}$ restricts to
\[
	(j^{\varepsilon})^{\chi_0} \colon \Ext^{\varepsilon}_{I/Z_1} (\chi_0,\pi_1) \hookrightarrow \Ext^{\varepsilon}_{I/Z_1} (\chi_0,\pi)=0,
\]
which implies that \eqref{eq:step2-claim} holds,
as was to be shown.
Hence, \eqref{eq:tau-restriction} is an isomorphism,
and we define $i_1$ as the preimage
of $i|_{\tau^{(n)}[\mI]} $ via
\eqref{eq:tau-restriction}.

Finally, we have the commutative diagram
\begin{equation}
	\label{eq:kappa1-and-i}
	\begin{tikzcd}[column sep = large]
 N \ar[r,"\sim"']\ar[d, two heads]\ar[dr, phantom, "\eqref{diag:thetas-square}"] &  
	\grm(\pi^\vee ) \ar[r, "\grm(i^\vee )", two heads] \ar[d, two heads] \ar[dr, phantom, "\eqref{diag:i-i1}"]&  
	\grm((\tau^{(n)})^\vee ) \ar[r,"\kappa", "\sim"'] \ar[d, two heads]\ar[dr, phantom,"\eqref{diag:kappa-kappa1}"]&  
N/\mathcal{I}^{(n)}N
\ar[d, two heads]\\
 N_1 \ar[r,"\sim"] &  
\grm(\pi_1^\vee ) \ar[r, "\grm(i_1^\vee )"', two heads] &  
\grm((\tau_1^{(n)})^\vee ) \ar[r,"\kappa_1"', "\sim"] &  
N_1/\mathcal{I}^{(n)}N_1,
\end{tikzcd}
\end{equation}
and we claim that the bottom row 
composes to the projection modulo $\mathcal{I}^{(n)}$.
But this follows from the commutativity of 
\eqref{eq:kappa1-and-i},
and the fact that
\eqref{eq:kappa-and-i}
is given by the projection modulo 
$\mathcal{I}^{(n)}$.

\paragraph{Step 3.}
We claim that 
$i_1 \colon \tau_1^{(n)} \hookrightarrow\pi_1$
restricts to an isomorphism
$i_1[\mI^n] \colon  \tau _1^{(n)}[ \mathfrak{m}^n] \xrightarrow{\sim}  \pi _1[ \mathfrak{m}^n]$
of $\Lambda$-modules,
which is enough to conclude:
we then have the following commutative diagram
with exact rows
\[
\begin{tikzcd}
	0 \ar[r] & 
	\tau _1^{(n)}[ \mathfrak{m}^n] \ar[r]  \ar[d, "{i_1[\mI^n]}"',"\wr"]& 
	\tau ^{(n)}[ \mathfrak{m}^n] \ar[r] \ar[d, "{i[\mI^n]}"',"\wr"]& 
	(\tau ^{(n)}/\tau^{(n)}_1)[ \mathfrak{m}^n] \ar[r] \ar[d, dashed, "\wr"]  & 
	0\\
	0 \ar[r] & 
	\pi _1[ \mathfrak{m}^n] \ar[r] & 
	\pi [ \mathfrak{m}^n] \ar[r] & 
	(\pi /\pi_1)[ \mathfrak{m}^n] \ar[r] & 
	0, 
\end{tikzcd}
\]
so it is enough to find a splitting 
of the top row.
But
notice that if $V'(\ell)$ is a complementary subspace
of $\Vell[\pi_1,\ell]$ in $\V$, for all
$0\le \ell\le f$, 
and if 
\begin{equation}
	\label{eq:tau2-def}
	\tau_2^{(n)} \defeq \bigoplus_{\ell=0} ^f
\left( 	V'(\ell)\otimes_{\F} \left( \bigoplus_{ \lambda \in \mathscr{P}_{\ell}}  \tau _ \lambda ^{(n)} \right)
\right),
\end{equation}
then
we have a direct sum decomposition
$\tau^{(n)}[\mI^n]= \tau_1^{(n)}[\mI^n]\oplus \tau_2^{(n)}[\mI^n]$
of $I$-representations,
with $\tau_2^{(n)}[\mI^n]\cong (\tau^{(n)}/\tau_1^{(n)})[\mI^n]$
(also as $I$-representations),
which defines a splitting.

Now, to show that $i_1[\mI^n]$
is an isomorphism, it is enough to prove
that
\[
 {\grm(i_1[\mI^n]^\vee )}\colon \grm(\pi_1[\mI^n]^\vee ) \twoheadrightarrow 
	\grm(\tau_1^{(n)}[\mI^n]^\vee )
\]
is an isomorphism of $\grL$-modules.
Indeed, $i_1[\mI^n]^\vee $ is strict 
(being surjective $\Lambda$-linear between two $\Lambda$-modules with the
$\mI$-adic topology), 
hence we can use 
\cite[Theorem I.4.2.5(1)]{LvO}
and \cite[Theorem I.4.2.5(3)]{LvO}
(in our case $\pi_1[\mI^n]^\vee $
is $\mI$-adically
complete,
because its $\mI$-adic filtration is discrete)
to conclude
that $i_1[\mI^n]^\vee $ is an isomorphism.
For ease of notation, let 
$\psi \defeq \grm(i_1[\mI^n]^\vee )$.
\paragraph{Step 4.}
We show that $\psi$ is an isomorphism of $\grL$-modules.
Consider the following commutative diagram:
\begin{equation}
	\begin{tikzcd}
	N_1 \ar[r, "\theta_1", "\sim"']\ar[d, two heads] &  
	\grm(\pi_1 ^\vee ) \ar[r, two heads, "\grm(i_1^\vee )"]\ar[d, two heads] &  
	\grm((\tau_1^{(n)})^\vee ) \ar[r, "\sim"', "\kappa_1"] \ar[d, two heads]&  
	N_1/\mathcal{I}^{(n)}N_1 \ar[d, two heads]\\
	N_1/\overline{\mI}^n \ar[r, "\sim"'] &  
	\grm(\pi_1 ^\vee )/\overline{\mI}^n \ar[r, two heads]\ar[d,"\wr"
] &  
	\grm((\tau_1^{(n)})^\vee )/\overline{\mI}^n  \ar[r, "\sim"']\ar[d,"\wr"'
] &  
	N_1/(\mathcal{I}^{(n)}+ \overline{\mI}^n )N_1 \\
														  & \grm(\pi_1[\mI^n]^\vee )\ar[r, two heads, "\psi"']
	& \grm(\tau_1^{(n)}[\mI^n]^\vee ).
\end{tikzcd}
	\label{diag:T-piece}
\end{equation}
Here, the first row is the bottom row of \eqref{eq:kappa1-and-i}
and the second row is obtained from 
the first by modding out by $\overline{\mI}^n$.
The isomorphisms in the bottom square come from
the following natural isomorphism:
if $M$ is a smooth $I/Z_1$-module, then
$ M^\vee/\mI^{n} \cong M[\mI^{n}]^\vee  $
and 
$ \grm(M^\vee /\mI ^{n} ) \cong \grm(M^\vee )/\barmI^{n}$.

Finally, remember that
$\mathcal{I}^{(n)}N_1
\subseteq  \overline{\mI}^n N_1$,
which implies that the middle row of 
\eqref{diag:T-piece} is an isomorphism.
By
the commutativity of the bottom square of
\eqref{diag:T-piece},
$\psi$ is also an isomorphism, 
which concludes the proof.
\end{proof}

\subsection{On the structure of 
$\mK^{2}$-torsion}
	\label{sec:mK-torsion}

Let $\mK$ be the  maximal ideal of 
$\F [\![K_1/Z_1]\!]$,
and set 
$\Gt \defeq \F [\![\GLring/Z_1]\!]/\mK^{2}$.
The goal of this section is to determine,
in \Cref{prop:iota-tld},
the $\mK^{2}$-torsion of $\pi$ and its
subrepresentations.
In \Cref{prop:K1-exact} we extend this
result to all subquotients of $\pi$,
by showing 
that the natural morphism 
$\pi_2[\mK^{2}] \to 
(\pi_2/\pi_1)[\mK^{2}]$
is surjective for
all subrepresentations
$\pi_1 \subseteq \pi_2$ of $\pi$.
We assume hypotheses
\ref{hypothesis:i} to
\ref{hypothesis:iv}
of \Cref{sec:hypotheses} throughout.
\bigskip

We denote by $\tilDZero$ the
finite-dimensional $\Gt$-module
of \cite[Definition~4.3]{HW22}.
It is the unique (up to isomorphism)
$\Gt$-module which is maximal with respect to
the following two properties:
\begin{enumerate}[(i)]
\item $\socGt \tilDZero \cong \bigoplus_{\sigma \in W(\repr)} \sigma$;
\item any Serre weight of $W(\repr)$
occurs in $\tilDZero$ with multiplicity one.
\end{enumerate}
It follows from \cite[Proposition~4.1]{HW22}
that $\tilDZero$ decomposes as the direct sum
\begin{equation}
	\label{eq:til-direct-sum}
\tilDZero = \bigoplus_{\sigma \in W(\repr)} 
\tilDx,
\end{equation}
for some $\Gt$-modules $\tilDx$ with
$\socGt \tilDx = \sigma$.
For $\ell \in \{0, \dots, f\}$,
we set
$ \tilDZeroEll \defeq  \bigoplus_{\sigma \in W(\repr)_{\ell}} \tilDx,$
so that
$ \tilDZero = \bigoplus_{\ell =0}^f \tilDZeroEll.$
Remember that $\tilDZero^{K_1} \cong \DZero$
by \cite[Theorem~4.6]{HW22}.
The follow proposition generalises 
\cite[Proposition~3.1.8]{BHHMS5} 
to the case $r \ge 1$.
\begin{prop}
	\label{prop:iota-tld}
Suppose that $\repr$ is $\max \{9, 2f+1\} $-generic.
\begin{enumerate}[(i)]
\item 
The $\Gamma$-module isomorphism 
$\iota_\pi \colon \pi^{K_1}
\xrightarrow{\sim} V \otimes_{\F}
\DZero$
can be extended to a $\Gt$-module isomorphism
\begin{equation}
	\label{eq:iota-tld}
\widetilde{\iota}_\pi \colon
\pi[\mK^{2}] \xrightarrow{\sim} \V \otimes_{\F} \tilDZero.
\end{equation}
\item 
Let $ \pi'$ be
a subrepresentation of $\pi$.
If $\repr$ is irreducible,
then any $\widetilde{\iota}_\pi$
as in (i) restricts to
\[
\widetilde{\iota}_\pi \colon
\pi'[\mK^{2}] \xrightarrow{\sim} 
 \Vnull[\pi'] \otimes_{\F} \tilDZero,
\]
and if $\repr$ is reducible split,
then any $\widetilde{\iota}_\pi$
as in (i) restricts to
\[
\widetilde{\iota}_\pi \colon
\pi'[\mK^{2}] \xrightarrow{\sim} 
\bigoplus_{\ell =0} ^{f}
 \Vell[\pi',\ell] \otimes_{\F} \tilDZeroEll[\ell].
\]
\end{enumerate}
For the definition of 
$\Vnull[\pi']$ and $\Vell[\pi',\ell]$,
we refer to the paragraph after
the proof of \Cref{lemma:main}.
\end{prop}
We begin with some preliminary results.

The following lemma generalises 
\cite[Lemma~3.2.6]{BHHMS4} to the case $r\ge 1$.
\begin{lemma}
	\label{lemma:soc-ex}
Suppose that $\repr$ is $\max \{9, 2f+1\} $-generic.
Let $ \pi _1 \subseteq  \pi _2 $
be subrepresentations of $ \pi  $.
Then, the natural sequence
\[
	0 \to \socR ( \pi _1) 
	 \to \socR ( \pi_2 ) 
	 \to \socR (\pi_2/ \pi _1) 
	 \to 0
\]
is exact.
\end{lemma}
\begin{proof}
For $0\le \ell\le f$,
and for $i=1,2$,
we let $V_i(\ell)\defeq \Vell[\pi_i,\ell]$,
and we
pick a complementary subspace $V'(\ell)$
of $V_1(\ell)$ in $V_2(\ell)$.
If we let, for $i=1,2$,
\begin{equation}
	\label{eq:Di-D'-def}
\begin{array}{ccc}
	D_i \defeq \bigoplus_{\ell=0} ^f 
\left( 	V_i(\ell)\otimes_{\F}  \DZeroEll \right), 
& \text{and}& 
D' \defeq \bigoplus_{\ell=0} ^f \left( V'(\ell)\otimes_{\F}  \DZeroEll \right),
\end{array}
\end{equation}
then we have a direct sum
decomposition
$D_2= D_1 \oplus D'$, which identifies
$D_2/D_1 \cong D'$.
Consider
the following commutative
diagram of $\GLring$-representations
with exact rows
\begin{equation}
	\label{diag:D-pi-split}
	\begin{tikzcd}
		0 \ar[r] &  D_1 \ar[r]&  
		D_2 \ar[r] &  
		D'\ar[r] &  0\\
		0 \ar[r] &  \pi_1^{K_1} \ar[u, "\iota_\pi", "\wr"']\ar[r] &  
		\pi_2^{K_1} \ar[r]\ar[u, "\iota_\pi", "\wr"'] &  {\pi_2}^{K_1}/{\pi_1 ^{K_1}} \ar[r] \ar[u, dashed, "\wr"']&  0,
	\end{tikzcd}
\end{equation}
coming from \Cref{cor:direction}.
We know that the top row splits
($D'$ is naturally a subrepresentation of $D_2$),
which implies that the bottom row also splits.
In particular, the rows of
\eqref{diag:D-pi-split}
remain exact when we take
$I_1$-invariants.
By \Cref{cor:pi1-pi2-exact}(i)
with $n=1$,
we see that $(\pi_2/\pi_1) ^{I_1} \cong 
\pi_2^{I_1}/\pi_1^{I_1}$,
and so the bottom row becomes
\[
\begin{tikzcd}
0 \ar[r] &  \pi_1^{I_1} \ar[r] &  
\pi_2^{I_1} \ar[r] &  (\pi_2/\pi_1) ^{I_1} \ar[r] &  0.
\end{tikzcd}
\]
In particular, we have obtained 
an isomorphism
\begin{equation}
	\label{eq:quotient-and-I_1}
	( \pi _2/ \pi _1) ^{I_1}
	\cong
 D^{\prime I_1}
\end{equation}
of $I$-representations.
On the other hand, there exists an inclusion
\begin{equation}
	\label{eq:D'-and-quotient}
D' \cong \pi_2^{K_1}/\pi_1^{K_1} \hookrightarrow
	( \pi _2/ \pi _1) ^{K_1}
\end{equation}
of $\GLring$-representations,
hence an inclusion
\[
S \defeq 
 \socR(D') \hookrightarrow
	 \socR ( \pi _2/ \pi _1) ^{K_1},
\]
which we claim to be an isomorphism:
if not, there there exists some Serre weight $\sigma $
such that
$\sigma \oplus S \hookrightarrow (\pi_2/\pi_1)^{K_1} $
as $\GLring $-representations.
This inclusion, taken together with \eqref{eq:D'-and-quotient},
promotes to an inclusion
$\sigma \oplus D' \hookrightarrow (\pi_2/\pi_1)^{K_1}$,
which contradicts \eqref{eq:quotient-and-I_1}
upon taking $I_1 $-invariants.
\end{proof}

The following is a corollary of
\cite[Corollary~5.6]{BP12}
and \cite[Lemma~4.3.4]{BHHMS4}.
\begin{lemma}
	\label{prop:E-tau}
Let $\tau$ be a Serre weight.
\begin{enumerate}[(i)]
\item 
We have 
$\dim_{\F}\Ext^{1}_{\Gamma}(\tau, \tau')\le 1$
for every Serre weight $\tau'$.
If we set
$\mathscr{E}(\tau)$
be the set of Serre weights $\tau'$ 
such that 
$\Ext^{1}_{\Gamma} (\tau,\tau')\neq 0$,
then this set has cardinality 
$|\mathscr{E}(\tau)| \le 2f$,
and we have $\tau' \in \mathscr{E}(\tau)$
if and only if 
$\tau \in \mathscr{E}(\tau')$.
If $\tau$ is $1$-generic, 
then we have an equality
$|\mathscr{E}(\tau)| = 2f$.

\item 
Assume that $\tau$ is $2$-generic.
Then, for every Serre weight $\tau$
the natural $\F$-linear map
$\Ext^{1}_{\Gamma}(\tau, \tau')
\xrightarrow{\sim}
\Ext^{1}_{\GLring}(\tau, \tau')$
is an isomorphism.
\end{enumerate}
\end{lemma}
For $\tau'\in \mathscr{E}(\tau)$, we let
$E_{\tau',\tau} $ be the unique 
(up to isomorphism) $\Gamma$-extension of
$\tau$ by $\tau'$.
\begin{proof}
Part (i) follows from
\cite[Corollary~5.6(i)]{BP12}
with $\sigma = \tau'$.

For part (ii), 
suppose that there exists a Serre weight $\tau'$
such that the natural $\F$-linear map
$\Ext^{1}_{\Gamma}(\tau, \tau') \hookrightarrow 
\Ext^{1}_{\GLring}(\tau, \tau')$
is not an isomorphism.
Since this map is always injective,
this is equivalent to
$\Ext^{1}_{\Gamma}(\tau, \tau') \not \cong
\Ext^{1}_{\GLring}(\tau, \tau')$
abstractly.
By (i) and (ii) of
\cite[Corollary~5.6]{BP12}
with $\sigma= \tau'$
(the assumption $p>0$ follows from 
$\tau$ $2$-generic),
this implies that
$\Ext^{1}_{\Gamma}(\tau, \tau')=0$.
Indeed, a direct computation shows that
the conditions (ii.a) to (ii.d)
of \emph{loc.\ cit.}\ are incompatible
with 
the conditions (i.a), (i.b)
of \emph{loc.\ cit.}\ (we 
are crucially using $\tau$ $2$-generic
in this computation).
Moreover, it follows from
the conditions (ii.a) to (ii.d)
of \emph{loc.\ cit.}\ that
$\tau'$ is $0$-generic.

We can then apply
\cite[Lemma~4.3.4]{BHHMS4}
with $\sigma=\tau', \sigma'=\tau$
(the case $\sigma \cong \sigma' \cong 
\Sym^{p-2}\F^{2} \otimes \eta$ 
does not occur because $\tau$ is $2$-generic)
to deduce that 
$ \Ext^{1}_{\GLring}(\tau, \tau')=0$,
contradiction.
\end{proof}

From now on we fix a Serre weight
$\tau$, and set $\chi \defeq \tau^{I_1}$
(this will be the Serre weight 
of \eqref{eq:1st-nonsplit} below
in the proof of \Cref{prop:K1-exact}).

Recall that $\mI$ is the augmentation ideal of
$\F [\![I_1/Z_1]\!]$.
We choose a projective envelope $\ProjI \chi$ of $\chi$ 
in the category of $\F [\![I/Z_1]\!]  $-modules,
and define 
$\W{n} \defeq (\ProjI \chi)/\mI^n$,
for $n$ a positive integer.
Let $X''$ denote the largest semisimple
subrepresentation of $W_{\chi,3} $ that
does not contain $\chi$ as a
Jordan-H\"older constituent,
and define 
$\Wbar \defeq \W{3}/X''$,
cf.\ \cite[(3.1)]{HW22}.
By \cite[Corollary~3.3]{HW22},
$\Wbar$ is a $\Gt$-module.

Notice that $\tau = \cosocG (\IndI (\chi))$
by construction,
in particular $\tau$ occurs in the 
$\Gt$-cosocle of $\W{2}$.
We choose a projective envelope 
$\ProjGt \tau$ of $\tau$
in the category of $\Gt$-modules,
we let $\phi_{\tau,2}
\colon \ProjGt \tau \twoheadrightarrow 
\IndI (\W{2})$
be any nonzero $\Gt$-module homomorphism 
(they all differ by a scalar
since $\IndI (\W{2})$ is multiplicity free,
by \cite[Lemma~3.4]{HW22}),
and we let
${\phi}_{\tau} \colon \ProjGt \tau \to \IW$
be a lift of $\phi_{\tau,2}$,
making the following diagram commutative
\begin{equation}
	\label{eq:phi-tau-def}
\begin{tikzcd}
\ProjGt \tau \ar[r, dashed, " \phi_\tau"] 
\ar[dr, "\phi_{\tau,2}"' ]
& \IndI \Wbar \ar[d, two heads]\\
 & \IndI \W{2}.
\end{tikzcd}
\end{equation}
This construction recovers the morphism $\phi_\tau$ 
of \cite[Proposition~3.10(i)]{HW22},
and by \cite[Proposition~3.10(ii)]{HW22},
it depends to choice only up to a scalar.
In particular, the image $\im(\phi_{\tau} )$
is well-defined.

Let $\alpha \colon H\to \F^\times $ be the
character 
$\left( \begin{smallmatrix} a& \\ & d \end{smallmatrix} \right)
\mapsto ad^{-1}$
for $a,d\in \F_q^\times $.
For $j\in \{0, \dots, f-1\}$, let
$\alpha_j \defeq \alpha^{p^j}$.
We collect here some vanishing results.
\begin{lemma}
	\label{lemma:vanishing-lemmas}
Assume that $\repr$ is $2$-generic.
\begin{enumerate}[(i)]
\item If $\tau' \notin W(\repr)$, then
\[
	\Ext^1_{\Gt} (\tau', \tilDZero)=0.
\]
In particular,
$ \Ext^1_{\Gamma} (\tau', \tilDZeroEll)=0$
for all $0 \le \ell \le f$.
\item 
Assume that $\tau \in W(\repr)$.
Then, $\JH(\coker(\phi_\tau)) \cap
W(\repr) = \emptyset $, and
\[
\Ext^{1}_{\GLring}(\tau', \tau)=0
\]
for any $\tau' \in \JH(\coker(\phi_\tau))$.
\item 
Assume that $\tau \in W(\repr)$.
Then, $\coker(\phi_\tau)$ admits a direct sum decomposition
$\coker(\phi_\tau) \cong 
\bigoplus_{j=0}^{f-1} \coker(\phi_\tau)_j$,
where $\coker(\phi_\tau)_j$
is a certain quotient of 
$\IndI (\chi \alpha_j)$ for $0 \le j \le f-1$.
In particular, $\coker(\phi_\tau)$ 
is a $\Gamma$-module.

Moreover, 
\[
\begin{cases}
  \Ext^{1}_{\GLring/Z_1}(\coker(\phi_\tau)_j, \DZero)=0 &  
\text{if }\chi \alpha_j \in \JH(\pi^{I_1}),\\
\Ext^{1}_{\GLring/Z_1}(\coker(\phi_\tau)_j, \pi)=0 & 
\text{if }\chi \alpha_j \notin \JH(\pi^{I_1}).
\end{cases}
\]
\end{enumerate}
\end{lemma}
\begin{proof}
(i)
The first statement follows from the maximality of $\tilDZero$,
as stated in \cite[Proposition~4.1]{HW22}.
The second statement follows from
the first one, and from the decomposition
$\tilDZero = 
\bigoplus_{ 0 \le \ell \le f} \tilDZeroEll$
of $\GLring$-representations.

(ii)
This is \cite[Lemma~3.1.2]{BHHMS5}.

(iii)
This is \cite[Lemma~3.1.4]{BHHMS5}.
We should note that (ii) of \emph{loc.\ cit.}
rests upon \cite[Lemma~3.1.1]{BHHMS5},
where $\rnull(\pi)$ is assumed to be $1$.
However, their proof does not make use of this assumption.
\end{proof}

Finally, recall the  $\Gt$-module $\Theta_\tau$, 
constructed in \cite[Proposition~3.12]{HW22}
as a quotient of $\im(\phi_{\tau})$.

Let $\radGt$ denote the radical of a
$\Gt$-module of finite length.
If we assume that $\tau$ is $1$-generic,
then we know from \emph{loc.\ cit.}\ 
and \cite[Corollary~3.16]{HW22} that:
\begin{equation}
	\label{eq:Theta-specs}
\cosocGt(\Theta_\tau) \cong \tau,\quad
\socGt(\Theta_\tau) \cong \tau^{\oplus 2f},\quad
\radGt(\Theta_\tau) \cong 
\bigoplus_{\tau' \in \mathscr{E}(\tau)}
E_{\tau,\tau'},
\end{equation}
in particular $\Theta_\tau/\tau^{\oplus 2f}$
fits into a (necessarily nonsplit) 
short exact sequence of $\Gt$-modules
\begin{equation}
	\label{eq:Theta-ses}
0 \to \bigoplus_{\tau'\in \mathscr{E}(\tau)} \tau'
\to  \Theta_\tau/\tau^{\oplus 2f} 
\to \tau  \to 0. 
\end{equation}

The following is a slight generalisation of \cite[Corollary~3.14]{HW22}.
\begin{lemma}
	\label{lemma:Theta-quotient}
Suppose that $\tau$ is in $W(\repr)$ and is
$2$-generic. 
Let $Q$ be a quotient of $\ProjGt \tau$,
and assume that there exists an injection
$\tau^{\oplus n} \hookrightarrow Q$,
for some integer $n \ge 0$,
such that $Q/\tau^{\oplus n}$
fits into a short exact sequence
\[
	0\to S \to Q/\tau^{\oplus n}\to \tau\to 0,
\]
where $S$ is a subrepresentation of
$\bigoplus_{\tau'\in \mathscr{E}(\tau)} 
\tau^{\prime \oplus m}$,
for some integer $m \ge 1$.
Then, $Q$ 
is a quotient of $\Theta_\tau$,
and in particular a quotient of 
$\im(\phi_\tau)$.
\end{lemma}
\begin{proof}
The short exact sequence of the statement
implies that $Q/\tau^{\oplus n}$ has Loewy length $ \le 2$,
and since 
$\tau' \in \JH(\socGt(Q/\tau^{\oplus n}))
\implies
 \tau' \in \mathscr{E}(\tau)$
by hypothesis
(in which case $ \Ext^{1}_{\GLring}(\tau, \tau')\cong \Ext^{1}_{\Gamma}(\tau, \tau')$
by \Cref{prop:E-tau}(ii)),
we conclude that 
$Q/\tau^{\oplus n}$ is a $\Gamma$-module.

We claim that 
every $\Gamma$-representation of Loewy length $2$
and irreducible cosocle $\tau$
is multiplicity free.
Let $M$ be such a $\Gamma$-representation,
and notice that every quotient of $M$ 
with Loewy length $2$ has the same property,
so we can assume that $\socG(M) = \tau^{\prime \oplus  d}$
has only one isotypic component, 
for some Serre weight 
$\tau' \in \mathscr{E}(\tau)$ and some
positive integer $d \ge 1$.
However,
$\dim_{\F}\Ext^{1}_{\Gamma}(\tau, \tau') =1$
by \Cref{prop:E-tau}(i),
which implies that $d \le 1$.

In particular, we can assume $m=1$
without loss of generality.
Since
$\bigoplus_{\tau'\in \mathscr{E}(\tau)} \tau'$
is semisimple, $S$ is also a quotient
of
$\bigoplus_{\tau'\in \mathscr{E}(\tau)} \tau'$,
and we can consider the pushout
\[
\begin{tikzcd}[cramped]
0 \ar[r] & \bigoplus_{\tau'\in \mathscr{E}(\tau)} \tau'
\ar[d, two heads] 
\ar[r] &  \Theta_\tau/\tau^{\oplus 2f} \ar[d, two heads] 
\ar[r]& \tau \ar[r] \ar[d, equals] & 0 \\
0 \ar[r] & S \ar[r] &  Q'
\ar[ul, phantom, "\square" pos=0.6]\ar[r]& \tau \ar[r] & 0,
\end{tikzcd}
\]
where the exact sequence at the top is
\eqref{eq:Theta-ses}.

By \emph{loc.\ cit.} we have
$\dim_{\F} \Ext^{1}_{\Gt}(\tau, \tau')=1$
for each $\tau' \in \JH(S)$,
from which we deduce that $Q'$
is isomorphic to $Q/\tau^{\oplus n}$,
because they are both the unique (up to isomorphism)
$\Gt$-extension of $\tau$ by $S$
whose cosocle is $\tau$.

We know from \cite[Corollary~3.13]{HW22}
that $\Ext^{1}_{\Gt}(\Theta_\tau, \tau)$
vanishes, so we can lift the composition
$\Theta_\tau \twoheadrightarrow \Theta_\tau/\mK
\twoheadrightarrow Q/\tau^{\oplus n}$
to a $\Gt$-module homomorphism
$\Theta_\tau \to Q$,
which is surjective because it is so on cosocles (both being equal to $\tau$).
\end{proof}

The following lemma is a slight generalisation of \cite[Lemma~3.1.6]{BHHMS5}.
\begin{lemma}
	\label{lemma:Proj-quotient}
Suppose that $\tau$ is in $W(\repr)$ and
that $\repr$ is
$2$-generic. 
Let 
$\beta \colon \ProjGt \tau \twoheadrightarrow Q$
be a quotient of $\ProjGt \tau$
such that $\radGt Q \subseteq \V \otimes_{\F} \tilDZero $.
Then, $Q$ is a quotient of $\Theta_\tau$,
and in particular a quotient of $\im(\phi_\tau)$.
\end{lemma}
\begin{proof}
By \Cref{lemma:line-intersection}(ii)
applied to $S= \Gt$, $D= \tilDx$,
$D_\sigma= \tilDx[\sigma]$, $W=\V$, 
$M=\radGt Q$
(the assumptions on $D$ are satisfied by
\cite[Theorem~4.6]{HW22} and by
\eqref{eq:til-direct-sum})
we can decompose 
$\radGt Q=\bigoplus_{\sigma\in W(\repr)}S_{\sigma},$
for some sub-$\Gt$-modules
$S_{\sigma} \subseteq \V \otimes_{\F}\tilDx$,
say with socle $\socGt S_\sigma = V_\sigma \otimes_{\F} \sigma$,
where $V_\sigma \subseteq \V$
is an $\F$-vector subspace.
Moreover, we have
$S_\sigma \subseteq V_\sigma \otimes_{\F} \tilDx[\sigma]$.

\paragraph{Step 1.}
We claim that, if $\sigma\neq \tau$,
then $S_{\sigma} =\socGt S_\sigma (= V_\sigma \otimes_{\F}\sigma)$.
Moreover, if $S_\sigma \neq 0$ then
$\sigma \in \mathscr{E}(\tau)$.

The natural inclusions 
$V_\sigma \otimes_{\F}\sigma \hookrightarrow
S_\sigma \hookrightarrow 
V_\sigma \otimes_{\F}\tilDx[\sigma]$
induce homomorphisms
\begin{equation}
	\label{eq:ext-comp-iso} 
\Ext^{1}_{\Gt}(\tau, V_\sigma \otimes_{\F}{\sigma} ) \hookrightarrow 
\Ext^{1}_{\Gt}(\tau,S_\sigma) \hookrightarrow  
\Ext^{1}_{\Gt}(\tau, V_\sigma \otimes_{\F}\tilDx[\sigma]),
\end{equation}
which are injective because
$\tilDZero$ multiplicity free 
and $\sigma \neq \tau$ imply
$[\tilDx[\sigma]:\tau]=0$,
which in turn implies
$\Hom_{\Gt} (\tau, (V_\sigma \otimes_{\F}
\tilDx[\sigma])/S_\sigma)=
\Hom_{\Gt} (\tau, S_\sigma/(V_\sigma \otimes_{\F}\sigma))=0$.
The composition of \eqref{eq:ext-comp-iso}
is an isomorphism by
\cite[Lemma~4.10]{HW22},
from which it follows that both homomorphisms
in \eqref{eq:ext-comp-iso} are bijective.

Taking pushforwards, this
 implies the splitting of the bottom row of the
following commutative diagram with exact rows:
\begin{equation}
	\label{eq:radQ-Q-tau}
\begin{tikzcd}[row sep=small]
0 \ar[r] & \radGt( Q) \ar[r]\ar[d, two heads] 
& Q \ar[r]\ar[d, two heads] 
& \tau \ar[r]\ar[d, equals] & 0 \\
0 \ar[r] & S_\sigma \ar[d, two heads] \ar[r]& Q' \ar[ul, phantom, " \square", pos=0.65]
\ar[r] \ar[d, two heads]
& \tau \ar[r] \ar[d, equals] & 0 \\
0 \ar[r] & S_\sigma/(V_\sigma \otimes_{\F}\sigma) \ar[r] 
& Q'' \ar[r] \ar[ul, phantom, " \square", pos=0.48]
& \tau \ar[r] & 0.
\end{tikzcd}\end{equation}
Now, $\cosocGt Q = \cosocGt( \ProjGt \tau)=\tau$
implies $\cosocGt Q''= \tau$,
but since the bottom row of \eqref{eq:radQ-Q-tau} 
splits we can also compute 
$\cosocGt Q'' = \tau \oplus \cosocGt(S_\sigma/(V_\sigma \otimes_{\F}\sigma))$,
and this forces 
$\cosocGt(S_\sigma/(V_\sigma \otimes_{\F}\sigma))=0$,
or equivalently
$V_{\sigma} \otimes_{\F}\sigma = S_\sigma$.
Finally, notice that if $S_\sigma \neq 0$
then the middle row of \eqref{eq:radQ-Q-tau}
does not split (because $\cosocGt Q'
= \cosocGt Q=\tau$).
In particular we have
$ \Ext^{1}_{\GLring}(\tau, \sigma)\neq 0$,
which by \Cref{prop:E-tau}(ii)
is equivalent to 
$ \Ext^{1}_{\Gamma}(\tau, \sigma)\neq 0$,
i.e.\ to $\sigma \in \mathscr{E}(\tau)$.

\paragraph{Step 2.}
We show that $S_\tau$ has Loewy length two.

Consider a variant of
\eqref{eq:radQ-Q-tau} for $\sigma=\tau$:
\begin{equation}
	\label{eq:cosoc-Q-tau}
\begin{tikzcd}[row sep=small]
0 \ar[r] & \radGt Q \ar[r]\ar[d, two heads] 
& Q \ar[r]\ar[d, two heads] 
& \tau \ar[r]\ar[d, equals] & 0 \\
0 \ar[r] & S_\tau \ar[d, two heads] \ar[r]& Q' \ar[ul, phantom, " \square", pos=0.6]\ar[r] \ar[d, two heads]
& \tau \ar[r] \ar[d, equals] & 0 \\
0 \ar[r] & \cosocGt S_\tau \ar[r] 
& Q'' \ar[r] \ar[ul, phantom, " \square", pos=0.45]
& \tau \ar[r] & 0.
\end{tikzcd}\end{equation}
Again, $\cosocGt Q = \tau$ implies 
$\cosocGt Q'' = \tau$,
and so the bottom row does not split,
which implies that for all Serre weights 
$\tau'$ occurring in $\cosocGt S_{\tau}$
we have 
$\Ext^{1}_{\GLring}(\tau, \tau')\neq 0$,
which by \Cref{prop:E-tau}(ii)
is equivalent to 
$ \Ext^{1}_{\Gamma}(\tau, \tau')\neq 0$,
i.e.\ to $\tau' \in \mathscr{E}(\tau)$.

On the other hand, we claim 
that all Serre weights 
$\tau' \in \JH(\tilDZero) \cap \mathscr{E}(\tau)$
occur in $\soc_2(\tilDZero)$
(and nowhere else, since $\tilDZero$
is multiplicity free),
where $\soc_2$ denotes the second term of
the $\Gt$-socle filtration.

Indeed, let $\tau' \in \mathscr{E}(\tau)$
which does not appear in $\soc(\DZero)$.
Notice that
$\tau \in \mathscr{E}(\tau')$
by \Cref{prop:E-tau}(i), so we can consider
the $\Gamma$-module $E_{\tau,\tau'}$, 
and by the maximality of $\tilDZero$
(cf.\ the discussion preceding
\Cref{prop:iota-tld})
we can embed $E_{\tau,\tau'}$ into $\tilDZero$.
Hence, 
$E_{\tau,\tau'}= \soc_2(E_{\tau,\tau'})
\subseteq \soc_2(\tilDZero)$.
It follows that $S_\tau \subseteq \soc_2(V_\tau \otimes_{\F}\tilDx[\tau])$, 
in particular it has Loewy length $\le 2$,
with socle $V_\tau \otimes_{\F} \tau$.

By \Cref{lemma:Theta-quotient}
(with $n = \dim_{\F}V_\tau$
and $m = r = \dim_{\F} V$)
we have the desired $\widetilde{\Gamma}$-module factorisation
$ \ProjGt \tau \twoheadrightarrow
\Theta_\tau\twoheadrightarrow Q.$
\end{proof}
The following lemma generalises
\cite[Proposition~4.18]{HW22} to the case
$r \ge 1$.
\begin{lemma}
	\label{lemma:through-Theta}
Suppose that $\tau$ is in $W(\repr)$,
and that $\repr$ is $2$-generic.
The following two conditions are equivalent:
\begin{enumerate}[(i)]
\item 
every $\GLring$-equivariant homomorphism
$\ProjGt \tau \to \pi$ factors through
$\socR \pi$;
\item 
every $\GLring$-equivariant homomorphism
$\Theta_\tau\to \pi$ factors through
$\socR \pi$.
\end{enumerate}
\end{lemma}
\begin{proof}
$(i) \Rightarrow (ii)$ is clear,
$\Theta_\tau$ being a quotient of $\ProjGt \tau$.
For the reverse implication, 
suppose by contradiction that
there exists a
$\GLring$-equivariant homomorphism
$h \colon\ProjGt \tau \to \pi$
that does not factor through the socle.
Moreover, we can choose $h$ 
in such a way that 
$[\im(h) : \tau]$ is minimal.
If we let $Q \defeq \im(h) \subseteq \pi$,
then we show that $Q$ is a quotient of $\Theta_\tau$,
which will provide a contradiction.

For $\sigma \in W(\repr)$,
let $S_\sigma$ be the $\sigma$-isotypic component
of $\socGt(Q)$.
Since
\[
\socGt(Q) \subseteq
\socR(\pi) \xrightarrow[\sim]{\iota_\pi}  
\V \otimes_{\F} \bigoplus_{\sigma \in W(\repr)} \sigma,
\]
we can decompose $\socGt(Q)$ into
a direct sum
$\socGt(Q)= \bigoplus _{\sigma \in W(\repr)} S_\sigma$.
Moreover,
$\iota_\pi(S_\sigma) \subseteq 
V \otimes_{\F} \sigma$
is equal to $V_\sigma \otimes_{\F}\sigma$
for a unique $\F$-vector subspace $V_\sigma \subseteq \V$.
Notice that we have $[Q/S_\tau:\tau]=1$,
by the minimality of $h$ and the projectivity of $\ProjGt \tau$.
In particular, $Q/S_\tau$ is multiplicity free
by \cite[Corollary~2.26]{HW22}
applied with $\sigma =\tau$, 
$V=(Q/S_\tau)^{\vee}$, $s=1$.

For a fixed $\sigma \in W(\repr)$,
$\sigma \neq \tau$,
we consider the pushout extension
\begin{equation}
	\label{eq:socQ-Q-tau-1}
\begin{tikzcd}
0 \ar[r] & \socGt(Q) \ar[r]\ar[d, two heads] 
& Q \ar[r]\ar[d, two heads] 
& Q/\socGt(Q) \ar[r]\ar[d, equals] & 0 \\
0 \ar[r] & S_\sigma 
\ar[r]& Q'_{\sigma}
\ar[ul, phantom, " \square", pos=0.55]\ar[r]
&Q/\socGt(Q) \ar[r] &0,
\end{tikzcd}\end{equation}
and we let $\overline{Q}_\sigma$
be the unique quotient of $Q'_\sigma$
with socle $\sigma$ 
(note that we are using 
$Q'_\sigma$ multiplicity free,
since $Q/S_\tau$ is multiplicity free).

For $\sigma=\tau$ and for a fixed 
$\F$-linear form 
$\psi \colon V_\sigma \twoheadrightarrow \F$,
remember that 
$\iota_\pi (S_\tau)=V_\tau \otimes_{\F}\tau$,
and consider the pushout extensions
\begin{equation}
	\label{eq:socQ-Q-tau-2}
\begin{tikzcd}
0 \ar[r] & \socGt(Q) \ar[r]\ar[d, two heads] 
& Q \ar[r]\ar[d, two heads] 
& Q/\socGt(Q) \ar[r]\ar[d, equals] & 0 \\
0 \ar[r] & S_\tau \ar[d, " (\psi \otimes \id) \circ \iota_\pi"', two heads] \ar[r]& 
Q'_{\tau,\psi}  \ar[ul, phantom, " \square", pos=0.54]
\ar[r] \ar[d, two heads]
& Q/\socGt(Q) \ar[r] \ar[d, equals] & 0 \\
0 \ar[r] & \F \otimes_{\F}\tau \ar[r] 
& Q''_{\tau,\psi}  \ar[r] \ar[ul, phantom, " \square", pos=0.45]
& Q/\socGt(Q) \ar[r] & 0.
\end{tikzcd}\end{equation}
We claim that there is a unique
quotient $\overline{Q}_{\tau,\psi}$
of $Q''_{\tau,\psi}$ with socle
$\tau$.
Indeed, for every $\Gt$-submodule 
$K \subseteq Q''_{\tau,\psi}$
we have $Q''_{\tau,\psi}/K \neq 0$
if and only if 
$K \subseteq \radGt(Q''_{\tau,\psi})$
(because $\cosocGt(Q''_{\tau,\psi})=\tau$
is irreducible),
so it is enough to show the claim for
$\radGt(Q''_{\tau,\psi})$,
where it holds
because $\radGt(Q''_{\tau,\psi})$
is multiplicity free.
Indeed, notice first that 
the quotient $Q \twoheadrightarrow 
Q''_{\tau,\psi} \twoheadrightarrow
Q''_{\tau,\psi} /\tau$
factors through $Q/S_\tau$
by construction 
(cf.\ \eqref{eq:socQ-Q-tau-2}),
hence 
$ [Q''_{\tau,\psi}/\tau:\tau] \le 
[Q/S_\tau :\tau]=1$.
However, $\tau$ occurs in the cosocle of
$Q''_{\tau,\psi}$, so equality holds.
Therefore, we can compute 
$[Q''_{\tau,\psi}:\tau]=
[Q''_{\tau,\psi}/\tau:\tau]+1=
[Q/S_\tau :\tau]+1=2$,
which implies
$[\radGt (Q''_{\tau,\psi} ):\tau]=
[Q''_{\tau,\psi}:\tau]-1= 1$.

We set $d\defeq \dim_{\F} V_\tau$,
we pick a basis 
$(\psi_{1}, \dots, \psi_{d})$
of the $\F$-linear dual $V_\tau^{\vee}$,
and we set $\overline{Q}_\tau
\defeq \bigoplus_{j=1} ^d \overline{Q}_{\tau,\psi_j}$.

\paragraph{Step 1.}
For $\sigma \in W(\repr)$,
we claim that $\radGt \overline{Q}_{\sigma}$ 
is $\mK$-torsion.

If $\sigma \neq \tau$ then we can apply
\cite[Corollary~2.31]{HW22} to 
$Q= \overline{Q}_\sigma$
to show that $\overline{Q}_\sigma$
is $\mK$-torsion,
so $\radGt\overline{Q}_\sigma$ is 
$\mK$-torsion \emph{a fortiori}.
Condition (a) of
\emph{loc.\ cit.}\ follows from
$\overline{Q}_\sigma$ multiplicity free,
while condition (b) of \emph{loc.\ cit.}\ holds
by \cite[Proposition~2.24]{HW18}
(i.e.\ because the representation 
$I(\sigma,\tau)$
of \Cref{deflem:I} is nonzero).

If $\sigma = \tau$, then 
it is enough to show that 
$\radGt( \overline{Q}_{\tau,\psi_j})$
is $\mK$-torsion for $j \in \{1, \dots, d\}$.
We wish to apply
\cite[Corollary~2.31]{HW22} 
to (the dual of) 
$\radGt ( \overline{Q}_{\tau,\psi_j})$.
Condition (a) of \emph{loc.\ cit.}
follows from the computation
following \eqref{eq:socQ-Q-tau-2}.
For condition (b) of \emph{loc.\ cit.}\ 
it is enough to show that, 
for every Serre weight $\tau'$ occurring in 
$\cosocGt(\radGt(\overline{Q}_{\tau,\psi_j}))$,
we have $\tau' \in \mathscr{E}(\tau)$
(because there is a surjection
$\ProjG \tau \twoheadrightarrow E_{\tau',\tau} $).
Indeed, consider the following 
pushout extension:
\[
\label{eq:cosocradQ-Q-tau}
\begin{tikzcd}
0 \ar[r] & \radGt (\overline{Q}_{\tau,\psi_j} )
\ar[r]\ar[d, two heads] 
& \overline{Q}_{\tau,\psi_j} \ar[r]\ar[d, two heads] 
& \tau \ar[r]\ar[d, equals] & 0 \\
0 \ar[r] & 
\cosocGt(\radGt(\overline{Q}_{\tau,\psi_j}))
\ar[r]& \overline{Q}_{\tau,\psi_j}'
\ar[ul, phantom, " \square", pos=0.55]\ar[r]
& \tau \ar[r] & 0.
\end{tikzcd}
\]
Since $\cosocGt(\overline{Q}_{\tau,\psi_j} )=\tau$
we have
$\cosocGt(\overline{Q}_{\tau,\psi_j}')=\tau$,
i.e.\ the bottom row does not split.
Hence, for every Serre weight
$\tau'$ occurring in 
$\cosocGt(\radGt(\overline{Q}_{\tau,\psi_j}))$
we have 
$\Ext^{1}_{\GLring}(\tau, \tau')\neq 0$,
which by \Cref{prop:E-tau}(ii)
is equivalent to 
$ \Ext^{1}_{\Gamma}(\tau, \tau')\neq 0$,
i.e.\ to $\tau' \in \mathscr{E}(\tau)$,
which proves the claim.

\paragraph{Step 2.}
We claim that $\radGt Q$ 
is $\mK$-torsion.

Notice that the natural map
$Q \to \bigoplus _{\sigma \in W(\repr)}
\overline{Q}_{\sigma}$
is injective, because it is so on the socles.
Taking $\radGt(-)$, we see that
$\radGt Q$ embeds into 
$\bigoplus _{\sigma \in W(\repr)}\radGt \overline{Q}_{\sigma}$,
which is $\mK$-torsion by Step 1.
Hence, $\radGt Q$ is also $\mK$-torsion,
and in particular $\radGt Q \subseteq 
\pi^{K_1} \cong \V \otimes_{\F}\DZero$.

We apply \Cref{lemma:Proj-quotient}
to $Q$ to conclude that $Q$ is a quotient of 
$\Theta_\tau$, contradiction.
\end{proof}

The following lemma generalises
\cite[Lemma~3.1.7]{BHHMS5} to the case
$r \ge 1$.
\begin{lemma}
	\label{lemma:through-soc}
Suppose that $\tau$ is in $W(\repr)$,
and that $\repr$ is $2$-generic.
Then, every $\GLring$-equivariant homomorphism
$\ProjGt \tau \to \pi$ factors through
$\soc_{\GLring} \pi$.
\end{lemma}
\begin{proof}
We follow the proof of
\cite[Lemma~3.1.7]{BHHMS5}.
\paragraph{Step 1.}
We prove that $\Hom_{\GLring} (\tau, \pi/\pi^{K_1})=0$.

Suppose by contradiction that 
a nonzero $\GLring$-equivariant homomorphism 
$\tau \hookrightarrow \pi/\pi^{K_1}$
exists.
Then, for a fixed nonzero element
in $ \Hom_{\GLring}(\tau, \pi/\pi^{K_1})$
the inverse image of 
$\tau \hookrightarrow \pi/\pi^{K_1}$ 
through the projection 
$\pi \twoheadrightarrow \pi/\pi^{K_1}$
is a $\GLring$-subrepresentation 
$E \subseteq \pi$, which
is a $\Gt$-representation but not a 
$\Gamma $-representation.
Using the isomorphism 
$\iota_\pi \colon \pi^{K_1} 
\xrightarrow{\sim} V \otimes_{\F}\DZero$,
we see that $E$ fits into 
a nonsplit extension
\begin{equation}
	\label{eq:E-nonsplit}
0 \to V \otimes_{\F}\DZero\to E \to\tau\to 0.
\end{equation}
By the projectivity of $\ProjGt \tau$,
there exists a $\Gt$-module homomorphism
$\beta \colon \ProjGt \tau \to E$
fitting into the following commutative diagram:
\[
\begin{tikzcd}
0 \ar[r] & V \otimes_{\F} \DZero  \ar[r] & E \ar[r] & \tau \ar[r] & 0 .\\
	 && \ProjGt \tau \ar[u, "\beta"] \ar[ur, two heads, bend right] 
\end{tikzcd}
\]
If we let $Q \defeq \im \beta$,
then we can apply
\Cref{lemma:Proj-quotient}
to $Q$, 
in particular $\beta$ factors as
\begin{equation}
	\label{eq:1st-beta-prime-def}
\begin{tikzcd}
\beta \colon \ProjGt \tau \ar[r, two heads,
" \phi_\tau"]& 
\im(\phi_\tau) \ar[r, two heads, " \widetilde{\beta}"]
& Q \ar[r, hook]& E.
\end{tikzcd}
\end{equation}
If $S \defeq \{0 \le j \le f-1
\mid \chi \alpha_j \in \JH(\pi^{I_1})\}$,
then we construct a $3$-step filtration on
$M \defeq \IW$ following
Step 1 of the proof of
\cite[Lemma~3.1.7]{BHHMS5}.
Set $M_2 \defeq \im(\phi_\tau) \subseteq M$,
and set $M_1 \defeq \ker
\left( M \twoheadrightarrow 
\bigoplus_{j \notin S}  
\coker(\phi_\tau)_j \right)$,
where $\coker(\phi_\tau)_j$
is the quotient of $\IndI (\chi \alpha_j)$
that appears in \Cref{lemma:vanishing-lemmas}(iii).
Then, $0 \subseteq M_2 \subseteq M_1 \subseteq M$ with
\begin{equation}
	\label{eq:M-filtration}
M_1/M_2 \cong \bigoplus _{j \in S} \coker(\phi_\tau)_j,
\quad
M/M_1 \cong \bigoplus _{j \notin S} \coker(\phi_\tau)_j. 
\end{equation}
We claim that $ \Ext^{1}_{\GLring/Z_1}(M_1/M_2, E) 
\overset{\eqref{eq:M-filtration}}{\cong}
\bigoplus _{j \in S}
\Ext^{1}_{\GLring/Z_1}(\coker(\phi_\tau)_j, E)$
vanishes.
By \eqref{eq:E-nonsplit}
it is enough to show that,
for all $j \in S$,
both
$\Ext^{1}_{\GLring/Z_1}(\coker(\phi_\tau)_j, \DZero)$
and
$\Ext^{1}_{\GLring/Z_1}(\coker(\phi_\tau)_j, \tau)$
vanish.
We have
$\Ext^{1}_{\GLring/Z_1}(\coker(\phi_\tau)_j, \DZero)=0$
by
\Cref{lemma:vanishing-lemmas}(iii),
while
$\Ext^{1}_{\GLring/Z_1}(\coker(\phi_\tau)_j, \tau)=0$
follows from
\Cref{lemma:vanishing-lemmas}(ii)
after a dévissage 
on the Jordan-H\"older constituents of
$\coker(\phi_\tau)_j$.
As a consequence, the restriction homomorphism
$\Hom_{\GLring/Z_1} (M_1, E) \to 
\Hom_{\GLring/Z_1} (M_2, E)$
is surjective,
and we can extend the composition
$\im(\phi_\tau)
\overset{\widetilde{\beta}}{\twoheadrightarrow}Q
\hookrightarrow E$ 
of \eqref{eq:1st-beta-prime-def} 
to a $\Gt$-module homomorphism
$\beta' \colon M_1 \to E$.

We also have $\Ext^{1}_{\GLring/Z_1}(M/M_1, \pi)
\overset{\eqref{eq:M-filtration}}{\cong}
\bigoplus _{j \notin S}
\Ext^{1}_{\GLring/Z_1}(\coker(\phi_\tau)_j, \pi)=0$
by \Cref{lemma:vanishing-lemmas}(iii),
so we can extend the composition 
$M_1 \xrightarrow{\beta'} Q \subseteq E \subseteq \pi$
to a $\GLring$-equivariant homomorphism 
$\beta'' \colon M \to \pi$,
which corresponds to an $I$-equivariant
homomorphism $\Wbar \to \pi$ under Frobenius reciprocity,
which in turn factors through 
$\Wbar \twoheadrightarrow \chi \hookrightarrow \pi$
by hypothesis \ref{hypothesis:ii} of
\Cref{sec:hypotheses}
(indeed, $\chi$ only occurs in the $I_1$-invariants
of $\pi[\mI^{3}]$, 
and $\pi^{I_1}$ is $I$-semisimple,
so $\Wbar \to \pi$ factors through
$\cosoc_I(\Wbar) = \chi$).
Correspondingly, $\beta''$ factors through
$M \twoheadrightarrow \IndI \chi \to \pi$,
in particular 
$\im(\beta'') \subseteq \pi^{K_1}$,
since $\IndI \chi$ is $K_1$-invariant.
But $\tau = \cosocR (\IndI \chi)$ 
only occurs in $\socR \pi$ 
and not elsewhere in $\pi^{K_1}$
(using that $\DZero$ is multiplicity free, and hypothesis \ref{hypothesis:i} 
of \Cref{sec:hypotheses}),
so the image of $\beta''$ is just $\tau$.

Finally, we have 
$0 \neq \im(\beta) \overset{\eqref{eq:1st-beta-prime-def}}{=}
\im(\widetilde{\beta})
\subseteq
\im(\beta'') = \tau$.
The inclusion follows from the fact that
$\beta''$ extends $\beta'$,
which extends 
$\im(\phi_\tau)
\overset{\widetilde{\beta}}{\twoheadrightarrow}Q
\hookrightarrow E$.
This forces $\im(\beta)=\tau$,
contradicting $E$ nonsplit.

\paragraph{Step 2.}
By \Cref{lemma:through-Theta},
it is enough to show that every
$\GLring$-equivariant homomorphism
$\Theta_\tau \to \pi$ factors through
$\socR \pi$.

Suppose by contradiction that
there exists a $\GLring$-equivariant homomorphism 
$\gamma \colon \Theta_\tau \to \pi$
that does not factor through $\socR \pi$.
By hypothesis \ref{hypothesis:i}
of \Cref{sec:hypotheses},
and since $\DZero$ is multiplicity free,
$\tau$ only occurs in the socle of $\pi^{K_1}$,
hence $\im(\gamma)$ is not contained 
in $\pi^{K_1}$.

However, $\radGt\Theta_\tau$ is 
$\mK$-torsion by \eqref{eq:Theta-specs},
and so we must have 
$\radGt\Theta_\tau=\Theta_\tau^{K_1}$.
Indeed, the containment
$\radGt\Theta_\tau\subseteq \Theta_\tau^{K_1}$
is clear, if inclusion were strict we would have
$\Theta_\tau^{K_1}= \Theta_\tau$
(because $\cosocGt\Theta_\tau\cong \tau$
is irreducible),
which contradicts $\im(\gamma) \not \subseteq  \pi^{K_1}$.

In particular, $\gamma \colon \Theta_\tau \to \pi $ 
factors through a $\GLring$-equivariant 
homomorphism
$\overline{\gamma} \colon \tau 
\cong  \Theta_\tau/\radGt (\Theta_\tau)
\to \pi/\pi^{K_1}$.
Step 1 implies that $\overline{\gamma}=0$,
which contradicts
$\im(\gamma) \not \subseteq \pi^{K_1}$.
\end{proof}

We are ready to prove  \Cref{prop:iota-tld}.
\begin{proof}[Proof of \Cref{prop:iota-tld}]
(i)
Fix an injective envelope
$\DZero \hookrightarrow \injGt(\DZero)$
in the category of $\Gt$-modules.
Since $\tilDZero^{K_1} \cong \DZero$
by \cite[Theorem~4.6]{HW22},
we can fix once and for all a copy of
$\tilDZero$ inside $\injGt(\DZero)$.
Notice that such a copy is unique,
cf.\ the proof of
\cite[Proposition~13.1]{BP12}.

By the injectivity of $\injGt(\DZero)$,
there exists a $\Gt$-module homomorphism
$\widetilde{\iota} \colon \pi[\mK^{2}]
\to \V \otimes_{\F}\injGt(\DZero)$
fitting into the following commutative diagram:
\begin{equation}
	\label{eq:oops-every-iota}
\begin{tikzcd}
 & \V \otimes_{\F}\injGt(\DZero) \\
\V \otimes_{\F}\DZero \ar[ur, hook]
\ar[r, hook, " \iota_\pi^{-1}"'] & \pi[\mK^{2}] \ar[u, hook, " \widetilde{\iota}"'].
\end{tikzcd}
\end{equation}
Note that $\widetilde{\iota}$ is injective because it is so on the socles.
Our goal is to show that $\im(\widetilde{\iota})=
\V \otimes_{\F} \tilDZero$,
so that we can set $\widetilde{\iota}_\pi
\defeq \widetilde{\iota}$.
It extends $\iota_\pi$ by \eqref{eq:oops-every-iota}.

Fix a line $L \subseteq \V$,
and set $\widetilde{D}_L \defeq \pi[\mK^{2}] \cap 
\widetilde{\iota}^{\;-1}
\left(L \otimes_{\F}\injGt(\DZero)\right)$.
The following are pullbacks of $\Gt$-modules:
\begin{equation}
	\label{eq:tilde-pullbacks}
\begin{tikzcd}
\V \otimes_{\F}\socG(\DZero) \ar[r, hook, " \iota_\pi^{-1} "] & \pi[\mK^{2}] 
\ar[r, hook, " \widetilde{\iota}"]& \V \otimes_{\F}\injGt(\DZero) \\
L \otimes_{\F}\socG(\DZero)
 \ar[r, hook, " \iota_\pi^{-1} "]\ar[u, hook]& \widetilde{D}_L
\ar[ul, phantom, " \square", pos=0.55]
\ar[r, hook, " \widetilde{\iota}"]\ar[u, hook]
& L \otimes_{\F}\injGt(\DZero)
\ar[ul, phantom, " \square"]
\ar[u, hook]. 
\end{tikzcd}
\end{equation}
Moreover, the chain of inequalities
$L \otimes_{\F}\socG(\DZero)
\subseteq \widetilde{\iota}(\socGt(\widetilde{D}_L))
\subseteq L \otimes_{\F}\socGt(\injGt(\DZero))
= L \otimes_{\F}\socG(\DZero)$
implies that 
$\widetilde{\iota}(\socGt(\widetilde{D}_L))=
L \otimes_{\F}\socG(\DZero)$.

By the maximality of $\tilDZero$ 
(cf.\ the discussion preceding
\Cref{prop:iota-tld}),
$\widetilde{\iota}$ restricts to 
$\widetilde{\iota} \colon\widetilde{D}_L
\hookrightarrow  L \otimes_{\F} \tilDZero$.
Indeed, (i) of \emph{loc.\ cit.}\ was shown
in the previous paragraph,
and (ii) follows from \Cref{lemma:through-soc}.

\paragraph{Step 1.}
We claim that $\widetilde{\iota}(\widetilde{D}_L)
=L \otimes_{\F} \tilDZero$
for all lines $L \subseteq \V$.

We follow the proof of
\cite[Proposition~3.1.8]{BHHMS5}.
Suppose that the inclusion 
$\widetilde{\iota}(\widetilde{D}_L)
 \subseteq L \otimes_{\F} \tilDZero$
is strict, 
and choose a Serre weight $\sigma$
together with a $\Gt$-module injection
$\sigma \hookrightarrow \left(L \otimes_{\F}\tilDZero\right)\!\!\Big/
\widetilde{\iota}(\widetilde{D}_L)$.
Let $S_\sigma \subseteq L \otimes_{\F}\tilDZero$
be a subrepresentation with cosocle $\sigma$,
and such that the composition
$ S_\sigma  \hookrightarrow L \otimes_{\F}\tilDZero \twoheadrightarrow 
\left(L \otimes_{\F}\tilDZero\right)\!\!\Big/ \widetilde{\iota}(\widetilde{D}_L)$ 
coincides with the chosen injection
$\sigma \hookrightarrow \left(L \otimes_{\F}\tilDZero\right)\!\!\Big/ \widetilde{\iota}(\widetilde{D}_L)$.
Such a subrepresentation exists because
$\tilDZero$ is multiplicity free.

Since $L \otimes_{\F}\DZero \subseteq \widetilde{D}_L$
and since $\tilDZero$ is multiplicity free,
we have $\sigma \in \JH(\tilDZero)
\setminus \JH(\DZero)$. 
In particular, if $\chi_\sigma$ is the character
of $I$ acting on $\sigma^{I_1}$,
we have $\chi_\sigma\notin \JH(\pi^{I_1})$.
Notice that $\radGt S_\sigma \subseteq 
\widetilde{\iota}(\widetilde{D}_L)$
by construction,
in particular there is a $\Gt$-module
injection $\gamma_0 \colon \radGt S_\sigma 
\overset{\widetilde{\iota}^{\;-1} } {\hookrightarrow}
\widetilde{D}_L 
\hookrightarrow  \pi[\mK^{2}] \subseteq \pi$.
By \cite[Lemma~3.1.1]{BHHMS5}
with $\chi = \chi_\sigma$
(note that \emph{loc.\ cit.}\ requires
$r=1$, but the proof does not use it),
when we apply $ \Hom_{\GLring/Z_1}(-, \pi)$
to $0 \to \radGt S_\sigma \to S_\sigma \to \sigma \to 0$
we obtain an isomorphism
\[
 \Hom_{\GLring}(S_\sigma, \pi) \xrightarrow{\sim}
\Hom_{\GLring}(\radGt S_\sigma, \pi).
\]
Thus, the injection $\gamma_0 \colon 
\radGt S_\sigma \hookrightarrow \pi$
lifts to an injection $\gamma \colon S_\sigma \hookrightarrow \pi$,
whose image is contained in $\pi[\mK^{2}]$,
because $S_\sigma$ is $\mK^{2}$-torsion.

We now claim that
$\widetilde{\iota}(\im(\gamma))
\subseteq V \otimes_{\F} \injGt(\DZero)$
is contained in 
$L \otimes_{\F} \injGt(\DZero)$,
and we prove this by showing that the
composition
$\delta \colon\widetilde{\iota}(\im(\gamma))
\subseteq V \otimes_{\F} \injGt(\DZero)
\twoheadrightarrow 
(V/L) \otimes_{\F} \injGt(\DZero)$
vanishes.
Notice that, by the construction
of $\gamma_0$
in the previous paragraph, we have 
$\widetilde{\iota}(\im(\gamma_0))
\subseteq 
\widetilde{\iota}(\widetilde{D}_L)
\subseteq L \otimes_{\F}\tilDZero$.
In particular, the composition
$\widetilde{\iota}(\im(\gamma_0))
\subseteq V \otimes_{\F} \injGt(\DZero)
\twoheadrightarrow 
(V/L) \otimes_{\F} \injGt(\DZero)$
vanishes, and so $\delta$ factors through
$\overline{\delta} \colon
\sigma = \im(\gamma)/\im(\gamma_0)
\to (V/L) \otimes_{\F} \injGt(\DZero)$
(where the first equality
follows from $\cosocGt S_\sigma=\sigma$).
Recall from the previous paragraph that
$\sigma \notin \JH(\DZero)
=\socGt(\injGt(\DZero))$,
hence $\overline{\delta}=0$,
hence $\delta=0$,
proving the claim.

In conclusion, $\im(\gamma)
\subseteq \pi[\mK^{2}]$
and $\widetilde{\iota}(\im(\gamma))
\subseteq L \otimes_{\F}\injGt(\DZero)$, so 
$\gamma \colon S_\sigma \hookrightarrow \pi$
factors through $\gamma \colon S_\sigma 
\to \widetilde{D}_L \hookrightarrow \pi[\mK^{2}]$
by definition of
$\widetilde{D}_L$
(cf.\ \eqref{eq:tilde-pullbacks}),
in particular $\sigma \in \JH(\widetilde{D}_L)$.
However, by construction we also have
$\sigma \in  \JH\left(\left(L \otimes_{\F} \tilDZero\right)\!\!\Big/
\widetilde{D}_L\right)$,
which contradicts $L \otimes_{\F}\tilDZero$ multiplicity free.
We conclude that 
$\widetilde{\iota}(\widetilde{D}_L)=
L \otimes_{\F}\tilDZero$
for all lines $L \subseteq \V$.

By \Cref{lemma:line-intersection}(i)
applied to $S= \Gt$, $D= \tilDZero$, 
$W=\V$, $M_1= \widetilde{\iota}(\pi[\mK^{2}]) \cap  
(\V \otimes_{\F}\tilDZero)$,
$M_2=V \otimes_{\F} \tilDZero$
we deduce that
$V \otimes_{\F}\tilDZero \subseteq 
\widetilde{\iota}(\pi[\mK^{2}])$.

\paragraph{Step 2.}
We claim that
$V \otimes_{\F}\tilDZero = 
\widetilde{\iota}(\pi[\mK^{2}])$.

Suppose for the sake of contradiction
that the inclusion 
$V \otimes_{\F}\tilDZero \subseteq 
\widetilde{\iota}(\pi[\mK^{2}])$
is strict,
and pick a Serre weight 
\[
\sigma \in
\JH\left(
\socGt\left( \widetilde{\iota}(\pi[\mK^{2}])/(V \otimes_{\F}\tilDZero)\right) \right)
\subseteq 
\JH\left(\socGt\left(V \otimes_{\F}\left(\injGt(\DZero)/
 \tilDZero\right)\right)\right),
\]
where the inclusion comes from the 
$\Gt$-module injection
\[
\frac{\widetilde{\iota}(\pi[\mK^{2}])}
{V \otimes_{\F}\tilDZero}
\subseteq 
\frac{V \otimes_{\F}\injGt(\DZero)}
{V \otimes_{\F}\tilDZero}\cong
V \otimes_{\F}
\frac{\injGt(\DZero)}{\tilDZero},
\]
after taking $\Gt$-socles.
We see that
$\sigma' \in \JH \left( \socGt\left(
\injGt(\DZero)/\tilDZero\right) \right)
\implies \sigma' \in W(\repr)$
by the maximality of $\tilDZero$
(cf.\ the discussion preceding
\Cref{prop:iota-tld}),
in particular $\sigma \in W(\repr)$.
If $\beta \colon\ProjGt \sigma \hookrightarrow
\widetilde{\iota}(\pi[\mK^{2}])/(V \otimes_{\F}\tilDZero)$
is any nonzero $\Gt$-module homomorphism
(which exists by the choice of $\sigma$),
we can lift it to a homomorphism
$ \widetilde{\beta} \colon\ProjGt \sigma \to 
\widetilde{\iota}(\pi[\mK^{2}])$
by the projectivity of $\ProjGt \sigma$.

By \Cref{lemma:through-soc}
applied to $\tau = \sigma$,
the composition
$\ProjGt \sigma \xrightarrow{\widetilde{\beta}}
\widetilde{\iota}(\pi[\mK^{2}])
\xrightarrow[\sim]{\widetilde{\iota}^{\;-1}}
\pi[\mK^{2}] \subseteq \pi$
factors through $\socR (\pi)$,
in particular $\widetilde{\beta}$
factors through $\widetilde{\iota}(\socR (\pi))$.
However, we see from \eqref{eq:oops-every-iota}
that $\widetilde{\iota}(\socR (\pi))
\subseteq V \otimes_{\F}\tilDZero$,
and since $\widetilde{\beta}$ lifts $\beta$
we must have $\beta=0$, contradiction.

(ii)
Suppose that $\repr$ is irreducible.

We know from (i) that 
$\widetilde{\iota}(\pi'[\mK^{2}])
\subseteq \V \otimes_{\F}\tilDZero$.
By the second part of
\Cref{lemma:line-intersection}(ii)
applied to $S= \Gt$, $D= \tilDZero$, 
$W=\V$, 
$M= \widetilde{\iota}(\pi'[\mK^{2}])$
we deduce that
\begin{equation}
	\label{eq:tld-inclusion}
\widetilde{\iota}(\pi'[\mK^{2}])
\subseteq
\Vnull[\pi'] \otimes_{\F}\tilDZero
\end{equation}
(note that, in \emph{loc.\ cit.}\ we have,
for all $\sigma' \in W(\repr)$,
$W_{\sigma'} =\Vwt[\pi',\sigma']$
precisely by construction of
$\Vwt[\pi',\sigma']$, cf.\ \eqref{eq:V-def}),
and $\Vwt[\pi',\sigma'] = \Vnull[\pi']$
since $\repr$ is irreducible
(cf.\ the paragraph after
the proof of \Cref{lemma:main}).

We conclude following the proof of
\cite[Proposition~5.1.1(i)]{BHHMS5}.
Denote by $Q$ the quotient 
$\left(\Vnull[\pi'] \otimes_{\F}\tilDZero\right)\!\!\Big/
\widetilde{\iota}(\pi'[\mK^{2}])$,
we want to prove $Q=0$.
Since
$\tilDZero$ is multiplicity free,
we have $\JH(Q) \cap W(\repr)= \emptyset$.
The natural inclusions
\[
Q \cong 
\widetilde{\iota}^{\;-1} \left(\Vnull[\pi'] \otimes_{\F}\tilDZero\right)\!\!\Big/
\pi'[\mK^{2}]
 \hookrightarrow \pi[\mK^{2}]/\pi'[\mK^{2}]
\hookrightarrow (\pi/\pi')[\mK^{2}]
\hookrightarrow \pi/\pi',
\]
induce an embedding
$\socR(Q) \hookrightarrow \socR(\pi/\pi')$.
But $\sigma' \in \JH(\socR(\pi/\pi'))
\implies \sigma' \in  W(\repr)$
by \Cref{lemma:soc-ex},
so we must have $\socR(Q)=0$,
or equivalently $Q=0$.

The case $\repr$ reducible split is analogous,
we highlight the differences:
in this case we have $\Vwt[\pi',\sigma'] =
\Vell[\pi', \ell(\sigma')]$,
so the right-hand side of 
\eqref{eq:tld-inclusion} becomes
$\bigoplus_{\ell =0}^f 
\left( \Vell[\pi', \ell] \otimes_{\F}
\tilDZeroEll \right)$.
Accordingly, we denote by $Q$ the quotient 
$\left(
\bigoplus_{\ell=0}^f 
\left( \Vell[\pi', \ell] \otimes_{\F}
\tilDZeroEll \right) \right)\!\!\Big/
\widetilde{\iota}(\pi'[\mK^{2}])$,
and we prove $Q=0$ with the same argument.
\end{proof}

The following is a generalisation of 
\cite[Proposition~5.1.1]{BHHMS5}
for $r \ge 1$.
\begin{prop}
	\label{prop:K1-exact}
Suppose that $\repr$ is $\max \{9, 2f+1\} $-generic.
Let $ \pi _1 \subseteq  \pi _2 $ be
subrepresentations of $ \pi  $,
and let
$ \pi' \defeq  \pi_2/\pi_1 $.
Then, the induced sequence of $\GLring$-modules
\[
0 \to
\pi _1[\mK^{2}] \to 
\pi _2[\mK^{2}] \to 
\pi'[\mK^{2}] \to 0
\]
is split exact.
In particular, the induced sequence of 
$\Gamma$-representations
\[
0 \to
\pi _1^{K_1} \to 
\pi _2^{K_1} \to 
\pi^{\prime K_1} \to 0
\]
is split exact.
\end{prop}
\begin{proof}
Assume first that $\repr$ is reducible split. 
Following the proof of
\Cref{lemma:soc-ex},
for $0\le \ell\le f$
and for $i=1,2$ we
pick a complementary subspace $V'(\ell)$
of $\Vell[\pi_1,\ell]$ in $\Vell[\pi_2,\ell]$.
For $i=1,2$, we set
\begin{equation}
	\label{eq:tilDi-D'-def}
\begin{array}{ccc}
	\widetilde{D}_i \defeq \bigoplus_{\ell=0} ^f 
\left( 	\Vell[\pi_i,\ell]\otimes_{\F}  \tilDZeroEll \right)
& \text{and}& 
\widetilde{D}' \defeq \bigoplus_{\ell=0} ^f \left( V'(\ell)\otimes_{\F}  \tilDZeroEll \right).
\end{array}
\end{equation}
Then, we have a direct sum
decomposition
$\widetilde{D}_2= \widetilde{D}_1 \oplus \widetilde{D}'$, which identifies
$\widetilde{D}_2/ \widetilde{D}_1 \cong \widetilde{D}'$.
Consider
the following commutative
diagram of $\GLring$-representations
with exact rows
\begin{equation}
	\label{diag:til-D-pi-split}
	\begin{tikzcd}
		0 \ar[r] &  \widetilde{D}_1 \ar[r]&  
		\widetilde{D}_2 \ar[r] &  
		\widetilde{D}'\ar[r] &  0\\
		0 \ar[r] &  \pi_1[\mK^{2}] \ar[u, " \widetilde{\iota}_\pi", "\wr"']\ar[r] &  
		\pi_2[\mK^{2}] \ar[r]\ar[u, " \widetilde{\iota}_\pi", "\wr"'] &
{\pi_2}[\mK^{2}]/{\pi_1 [\mK^{2}]} \ar[r] \ar[u, dashed, "\wr"']&  0,
	\end{tikzcd}
\end{equation}
coming from 
\Cref{prop:iota-tld}(ii).
We know that the top row splits
($\widetilde{D}'$ is naturally a subrepresentation of $\widetilde{D}_2$),
which implies that the bottom row also splits.

Notice that, by reasoning as in the proof of 
\Cref{cor:pi1-pi2-exact}(ii),
it is enough to only consider the case $\pi_2=\pi$.
Consider the $\GLring$-equivariant injection
$\widetilde{D}' \cong 
\pi_2[\mK^{2}]/\pi_1[\mK^{2}]\hookrightarrow
(\pi _2/ \pi _1)[\mK^{2}]$.
If it is an isomorphism,
then we can conclude 
by \eqref{diag:til-D-pi-split}.
Assume for the sake of contradiction that 
this is not the case,
and pick a $\Gt$-submodule
$E \subseteq (\pi/\pi_1)[\mK^2]$
strictly containing $\widetilde{D}' \cong \pi[\mK^{2}]/\pi_1[\mK^{2}]$,
and minimal for this property:
it gives rise to a short exact sequence
\begin{equation}
	\label{eq:1st-nonsplit}
0 \to \widetilde{D}' \to E \to \tau \to 0
\end{equation}
of $\Gamma$-representations,
for some Serre weight 
$\tau \in \JH\left((\pi/\pi_1)[\mK^{2}]\right)$.
Notice that, by \Cref{lemma:soc-ex},
the inclusions 
$\widetilde{D}' \hookrightarrow E \subseteq (\pi/\pi_1)[\mK^{2}]$
induce isomorphisms on the socles,
and this implies that \eqref{eq:1st-nonsplit}
is nonsplit.
By \Cref{lemma:vanishing-lemmas}(i)
and by \eqref{eq:tilDi-D'-def}
we must have $\tau\in W(\repr)$. 
Let $\chi \defeq \tau^{I_1}$.
By the projectivity of $\ProjGt \tau$,
we can lift the natural
$\Gt$-equivariant projection 
$\ProjGt \tau \twoheadrightarrow \tau$ to $E$:
\[
\begin{tikzcd}
0 \ar[r] &  \widetilde{D}'  \ar[r] & E \ar[r] & \tau \ar[r] & 0 .\\
	 && \ProjGt \tau \ar[u, "\beta"] \ar[ur, two heads, bend right] 
\end{tikzcd}
\]
If we let $Q\defeq \im \beta$,
then we are in the conditions to apply
\Cref{lemma:Proj-quotient} to $Q$, 
in particular $\beta$ fits into the
following commutative diagram
of $\Gt$-modules:
\begin{equation}
	\label{eq:beta-tld-def}
\begin{tikzcd}
\beta \colon \ProjGt \tau \ar[r, two heads,
" \phi_\tau"]& 
\im(\phi_\tau) \ar[r, two heads, " \widetilde{\beta}"]
& Q \ar[r, hook]& E.
\end{tikzcd}
\end{equation}

\paragraph{Step 1.}
We claim that the short exact sequence
\begin{equation}
	\label{eq:Wbar-ses}
0\to \im (\phi_\tau) \to \IW\to 
\coker (\phi_\tau)\to 0
\end{equation}
remains exact after applying
$\Hom_{\Gt} (-,E)$.

Consider the exact sequence
\begin{equation}
	\label{eq:ses-mK}
\Ext^{1}_{\Gt}(\coker (\phi_\tau), \widetilde{D}')
\to
\Ext^{1}_{\Gt}(\coker (\phi_\tau), E)\to
\Ext^{1}_{\Gt}(\coker (\phi_\tau), \tau)
\end{equation}
coming from \eqref{eq:1st-nonsplit}.
First, we show that the left-hand term of \eqref{eq:ses-mK}
vanishes.
By \eqref{eq:tilDi-D'-def},
it is enough to show that
$\Ext^{1}_{\Gt}(\coker (\phi_\tau), \tilDZeroEll)=0$
for $\ell \in \{0, \dots, f\}$.
With a dévissage, we reduce to showing that
$\Ext^{1}_{\Gt}(\tau', \tilDZeroEll)=0$
for all
$\tau' \in \JH(\coker(\phi_\tau))$,
and this follows from 
\Cref{lemma:vanishing-lemmas}(i)
and (ii).
Then, we show that the right-hand term 
of \eqref{eq:ses-mK} also vanishes.
This follows again from a dévissage on the
Jordan-H\"older constituents of 
$\coker(\phi_\tau)$, 
together with
\Cref{lemma:vanishing-lemmas}(ii)
(notice that $ \Ext^{1}_{\GLring/Z_1}(\tau', \tau)=0$
implies $ \Ext^{1}_{\Gt}(\tau', \tau)=0$).

We conclude from \eqref{eq:ses-mK} 
that $\Ext^{1}_{\Gamma}(\coker (\phi_\tau), E) =0$,
which proves our claim.
In particular, the $\Gt$-module homomorphism
$\widetilde{\beta}$ of \eqref{eq:beta-tld-def}
fits into a commutative diagram
of $\Gt$-modules:
\begin{equation}
	\label{eq:beta'-def}
\begin{tikzcd}
\im(\phi_\tau)
\ar[d, hook] \ar[r, two heads,
" \widetilde{\beta}"] & Q\ar[d, hook]  \\
\IW \ar[r, dashed, " \beta'"']& E.
\end{tikzcd}
\end{equation}
\paragraph{Step 2.}
We claim that there exists a
$\Gt$-module homomorphism
$\beta'' \colon \IW\to \pi$ 
that makes the following diagram commute: 
\begin{equation}
	\label{eq:beta''-def}
\begin{tikzcd}
&&&
\pi \ar[d, two heads] \\
\IW\ar[urrr, " \beta''", bend left = 8] \ar[r, " \beta'"']&
E \ar[r, " \subseteq ", phantom]&
(\pi/\pi_1)^{K_1} \ar[r,"\subseteq", phantom]&
 \pi/\pi_1.
\end{tikzcd}
\end{equation}
We construct $\beta''$ by showing that 
$ \pi \twoheadrightarrow \pi/\pi_1 $
remains surjective after applying 
the functor
$\Hom_{\GLring}(\IW,-) $.
Consider
the following chain of natural isomorphisms:
\[
\begin{tikzcd}
	\Hom_{\GLring}(\IW,\pi ) \ar[r] \ar[d, "\wr"']&   
	\Hom_{\GLring}(\IW,\pi/\pi_1) \ar[d, "\wr"]\\
	\Hom_{I}(\Wbar,\pi ) \ar[r] \ar[d,equal]&   
	\Hom_{I}(\Wbar,\pi/\pi_1) \ar[d,equal]\\
	\Hom_{I/Z_1}(\Wbar,\pi ) \ar[r] \ar[d, "\wr"', swap]&   
	\Hom_{I/Z_1}(\Wbar,\pi/\pi_1) \ar[d, "\wr", swap]\\
	\Hom_{I/Z_1}(\Wbar,\pi[\mI^3] ) \ar[r] &   
	\Hom_{I/Z_1}(\Wbar,(\pi/\pi_1)[\mI^3]).
\end{tikzcd}
\]
On the first square we have used Frobenius reciprocity,
while on the last one we have used
that $\Wbar$ is $\mI^{3}$-torsion,
being a quotient of 
$\W{3} = (\Proj_{I/Z_1} \chi)/\mI^{3}$.
The claim then reduces to the fact that,
by \Cref{cor:pi1-pi2-exact}
with $n=3$,
\[
	0\to \pi_1[\mI^3] \to
	 \pi[\mI^3] \to
	 (\pi/\pi_1)[\mI^3] \to 0
\]
is split exact, so it remains exact 
when we apply $\Hom_{I/Z_1} (\Wbar,-)$.

By hypothesis \ref{hypothesis:ii}
of \Cref{sec:hypotheses}
and Frobenius reciprocity,
we see that the $\Gt$-module homomorphism
$\beta''$ of \eqref{eq:beta''-def} factors as
\[
\beta'' \colon \IW \twoheadrightarrow \IndI( \chi )\to \pi.
\]
Indeed, any $I$-equivariant map $\Wbar \to \pi$
factors through $\pi^{I_1}$, which is
$I$-semisimple, so it factors though
the cosocle $\cosoc_I \Wbar=\chi$.
In particular, $\im (\beta'')$ is contained in
$\pi^{K_1}$ and has cosocle $\tau
= \cosoc_{\GLring}(\IndI (\chi))$.
But $\tau$ only occurs in $\socR \pi$ 
and not elsewhere in $\pi^{K_1}$
(using that $\DZero$ is multiplicity free, and hypothesis \ref{hypothesis:i} 
of \Cref{sec:hypotheses}),
so the image of $\beta''$ is just $\tau$.

Finally, we have
$0 \neq Q = \im(\beta) \overset{\eqref{eq:beta-tld-def}}{=}
\im(\widetilde{\beta}) \overset{\eqref{eq:beta'-def}}{\subseteq}
\im(\beta') \overset{\eqref{eq:beta''-def}}{
\twoheadleftarrow} 
\im(\beta'') = \tau$,
which forces $Q = \tau$,
contradicting $E$ nonsplit.
\end{proof}
\begin{remark}
	\label{rem:sadly-K1}
Even though 
in \Cref{prop:K1-exact}
we are mainly interested in
the exactness at the level of 
$K_1$-invariants, in the proof
we are forced to pass
through the $\mK^{2}$-torsion.
The reason is roughly as follows: 
if we take the $K_1$-coinvariants of
\eqref{eq:Wbar-ses},
we lose exactness on the left.
Indeed, one can compute that $\tau$ occurs in
the $\Gamma$-socle of $\im(\phi_\tau)/\mK$,
but not in the $\Gamma$-socle of
$\im(\phi_\tau/\mK)$.

However, (even assuming that $\widetilde{D}'$
is $K_1$-invariant!)
we cannot rule out that $\tau$ occurs
in the $\Gamma$-socle of $\widetilde{D}'$.
Hence, even though we can construct
a $\Gamma$-equivariant map 
$\widetilde{\beta} \colon\im(\phi_\tau)/\mK
\twoheadrightarrow Q$
playing the role of \eqref{eq:beta-tld-def},
this map does not factor through
the quotient
$\im(\phi_\tau)/\mK \twoheadrightarrow 
\im(\phi_\tau/\mK)$,
which precludes us from using
homological arguments on
the short exact sequence 
that plays the role of \eqref{eq:Wbar-ses}.
\end{remark}
\begin{remark}
	\label{rem:sum-K1}
As an immediate consequence of 
\Cref{prop:K1-exact},
given $\pi_1,\pi_2$ two subrepresentations
of $\pi$, then 
$(\pi_1+\pi_2)^{K_1}=\pi_1^{K_1}+\pi_2^{K_1}$.
Indeed, we always have
$(\pi_1 \cap \pi_2)^{K_1}=\pi_1^{K_1}\cap \pi_2^{K_1}$,
so when we take $K_1$-invariants
in the following isomorphism
$(\pi_1+\pi_2)/\pi_1 \cong \pi_2/(\pi_1 \cap \pi_2)$
we obtain
\[
	(\pi_1+\pi_2)^{K_1}/\pi_1^{K_1} \cong
	\pi_2^{K_1}/(\pi_1^{K_1} \cap \pi_2^{K_1})
	\cong 
	(\pi_1^{K_1}+\pi_2^{K_1})/\pi_1^{K_1}.
\]
For dimension reasons, the inclusion
$\pi_1^{K_1}+\pi_2^{K_1} \subseteq
(\pi_1+\pi_2)^{K_1}$
must then be an equality.
\end{remark}

The following corollary is a (partial)
generalisation of \cite[Corollary~3.2.7]{BHHMS4}
to the case $r\ge 1$.
\begin{corollary}[]
	\label{cor:final-corollary}	
Assume that $\repr$ is 
$\max \{9,2f+1\} $-generic.
Let $\pi_1 \subseteq \pi_2$ be two 
subrepresentations of $\pi$,  and
let $\pi ' \defeq \pi_1/\pi_2 $.
If $\repr$ is reducible,
we let $V_i(\ell)\defeq \Vell[\pi_i,\ell]$, for
$i=1,2$;
if $\repr$ is irreducible, 
we let $V_i \defeq  \Vnull[\pi_i]$, for
$i=1,2$.
\begin{enumerate}[(i)]
\item We have $\dim_{\F}  \Dvee[\pi^{\prime\vee}] =
	\lgR(\socR (\pi'))$.
In particular, $\pi'\neq 0 $ implies that
$\Dvee[\pi'] $ 
is nonzero.
\item 
If $\repr$ is reducible,
	there exists a unique set $0\le r(\pi', \ell) \le r$ of integers, for $0\le{\ell}\le f$,
such that 
\[
	\pi^{\prime K_1} \cong \bigoplus_{\ell=0} ^f 
	\DZeroEll[\ell]^{\oplus r(\pi',\ell)}.
\]
Namely, $r(\pi',\ell)= \dim_{\F}  V_2(\ell)-\dim_{\F}  V_1(\ell).$

If $\repr$ is irreducible,
there exists a unique integer $0\le r(\pi')\le r$
such that
\[
	\pi^{\prime K_1} \cong \bigoplus_{\ell=0} ^f 
	\DZero^{\oplus r(\pi')}.
\]
Namely, $r(\pi')= \dim_{\F}  V_2-\dim_{\F}  V_1.$

\item 
Let $N'\defeq \bigoplus_{\ell=0} ^f \frac{V_2(\ell)}{V_1(\ell)} \otimes_{\F} \Nell$
if $\repr$ is reducible, and
$N'\defeq  \frac{V_2}{V_1} \otimes_{\F} \Nnull$
if $\repr$ is irreducible.
The surjection
	\[
		\theta' \colon N' \twoheadrightarrow \grm \pi^{\prime \vee}
	\]
of \eqref{diagr:thetas-ses}
is an isomorphism of graded $\grL $-modules with compatible $H $-action.
In particular, $\grm(\pi^{\prime\vee})$ and
$\pi^{\prime\vee}$ are both Cohen-Macaulay
of grade $2f$.
\item The subquotient $\pi'$ is generated 
	by its $\GLring$-socle.
\item The subquotient $\pi'$ has finite length, 
bounded above by $\sum_{ \ell=0} ^f r(\pi',\ell)$
if $\repr$ is reducible, 
and by $r(\pi')$ is $\repr$ is irreducible.
\item For $i=1,2$, let 
$\widetilde{\pi}_i$ be the subrepresentation
of $\pi$ of \Cref{def:conjugate-srep}
for $\pi'=\pi_i$.
Then, we have
$\widetilde{\pi}_2 \subseteq \widetilde{\pi}_1$,
and if we set
$\widetilde{\pi}'\defeq \widetilde{\pi}_1/\widetilde{\pi}_2,$
there is an isomorphism 
$\E (\pi^{\prime\vee })\cong  \widetilde{\pi}^{\prime\vee }\otimes \mytwist $
of $\Lambda$-modules with compatible
actions of $\GLfield$.
\end{enumerate}
\end{corollary}
Notice that the integers $r(\pi',\ell)$
and $r(\pi')$ of 
\Cref{cor:final-corollary}(ii)
extend the definition 
given in the paragraph after the proof of
\Cref{lemma:main}.
\begin{proof}
(i)
By \Cref{prop:main-memoire} the assertion
holds for $\pi_1$ and $\pi_2$,
we conclude by the exactness of $\Dvee[-]$
and \Cref{lemma:soc-ex}.

(ii)
This is \Cref{prop:K1-exact},
together with \Cref{cor:direction}.

(iii)
Consider the commutative diagram
\eqref{diagr:thetas-ses}
of \Cref{lemma:gr}(ii).
We know from \Cref{prop:main}(i)
that $\theta_1$ and $\theta_2$ 
are isomorphisms
(in \emph{loc.\ cit.}\ we take $\pi_1=\pi_1$ 
in the first case,
and $\pi_1=\pi_2$ in the second case),
so we deduce that 
$\theta' \colon N' \twoheadrightarrow 
\gr_F \pi ^{\prime \vee}$ is an isomorphism,
where $F$ is the submodule filtration
on $\pi^{\prime \vee }$
induced from $\pi^\vee$.
Finally, 
by reasoning as in the proof of
\Cref{prop:main}(iii),
we can show that $F$ is the
$\mathfrak{m}$-adic filtration.

By the proof of 
\cite[Theorem~2.1.2]{BHHMS4}
$\Nnull$ is Cohen-Macaulay of grade $2f$,
so its direct summand $\Nell$
is also Cohen-Macaulay.
Hence, 
$ \frac{V_2}{V_1} \otimes_{\F} \Nnull
 =: N' \cong \grm(\pi^{\prime\vee}) $ 
is Cohen-Macauley, and by
\cite[Proposition~III.2.2.4]{LvO}
$\pi^{\prime\vee}$ 
is also Cohen-Macaulay.

(iv)
By \Cref{thm:finite-length-yay}(i),
$\pi_2$ is generated by $\socR(\pi_2)$,
hence $\pi'$ is generated by the image of
$\socR(\pi_2)$ in $\pi'$,
which is equal to $\socR(\pi')$ by
\Cref{lemma:soc-ex}.

(v)
This is \Cref{thm:finite-length-yay}(ii).

(vi)
For $i=1,2$, 
we can consider the commutative diagram
\eqref{eq:conjugate-when-CM} of
\Cref{rem:conjugate-when-CM} for
$\pi' = \pi_i$
(which takes a simpler form
because
$\pi_i$ is Cohen-Macaulay 
by (iii)).
Then, the factorisation 
$\pi \twoheadrightarrow \pi/\pi_1
\twoheadrightarrow \pi/\pi_2$
gives rise to the factorisation
\begin{equation}
	\label{diagr:E-factorisation}
	\begin{tikzcd}
 \E(\pi^\vee)  \ar[r, two heads]\ar[d,"{ \kappa_\pi^\vee }"', "{\wr}"] 
\ar[rr, two heads, bend left=10]
& \E((\pi/\pi_1)^\vee ) 
\ar[d,"{ \kappa_\pi^\vee }", "{\wr}"'] \ar[r,dashed]
& \E((\pi/\pi_2)^\vee ) 
\ar[d,"{ \kappa_\pi^\vee }", "{\wr}"'] 
\\
\pi^\vee \otimes \mytwist  \ar[r, two heads]
\ar[rr, two heads, bend right=10]
&\widetilde{\pi}_1^\vee \otimes \mytwist \ar[r, dashed]
&\widetilde{\pi}_2^\vee \otimes \mytwist, 
\end{tikzcd}\end{equation}
where the leftmost square and the
outermost square both come from
\eqref{eq:conjugate-when-CM}.
In particular, we have the containment
$\widetilde{\pi}_2
\subseteq \widetilde{\pi}_1 $.

Now, consider the long exact sequence
induced by the short exact sequence
$
0 \to (\pi/\pi_2)^\vee  \to (\pi/\pi_1)^\vee  \to \pi^{\prime\vee }\to 0
$
after taking $ \Hom_{\Lambda}(-, \Lambda)$.
Since all the terms of the sequence
are Cohen-Macaulay of grade $2f$,
we obtain the following commutative
diagram with exact rows
\[\begin{tikzcd}
	0 \ar[r] &  	\E(\pi^{\prime\vee})\ar[r]
& \E((\pi/\pi_1)^\vee ) 
\ar[r]\ar[d,"{ \kappa_\pi^\vee \otimes (\mytwist) }"', "{\wr}"] 
\ar[dr, phantom, " \eqref{diagr:E-factorisation}"]
& \E((\pi/\pi_2)^\vee ) 
\ar[d,"{ \kappa_\pi^\vee \otimes (\mytwist)}", "{\wr}"'] \ar[r]
& 0 \\
       0 \ar[r] 
&\widetilde{\pi}^{\prime\vee}\otimes (\mytwist) \ar[r]
\ar[r]
&\widetilde{\pi}_1^\vee \otimes (\mytwist) \ar[r]
&\widetilde{\pi}_2^\vee \otimes (\mytwist)
	\ar[r]& 0,
\end{tikzcd}\]
and so ${ \kappa_\pi^\vee \otimes (\mytwist)}$ restricts
to an isomorphism 
$ \E(\pi^{\prime\vee}) \cong 
\widetilde{\pi}^{\prime\vee}\otimes (\mytwist) $.
\end{proof}

The natural generalisation of (v) 
of \cite[Corollary~3.2.7]{BHHMS4},
 i.e.\ the fact that 
for every irreducible subquotient $\overline{\pi}$
of $\pi$ we have $[\pi : \overline{\pi}]=r$,
is out of our reach.
Also out of our reach is the optimistic
expectation that $\overline{\pi}$
should only depend on $\repr$ 
(and not on $\glrepr$)
when $\pi$ is the representation \eqref{eq:piShi-dev}.
What we can show, however, is that 
$\overline{\pi}$ is always supersingular
 when we expect it to be.

\begin{corollary}[]
	\label{cor:supersingularity-for-free}
Assume that $\repr$ is 
$\max \{9,2f+1\} $-generic.
\begin{enumerate}[(i)]
\item 
Assume $\repr$ reducible split,
and consider the subrepresentation 
$\pi' \subseteq \pi$
of \Cref{thm:principal-series}.
Then, every Jordan-H\"older constituent
$\overline{\pi}$ of $\pi'$
is supersingular.
\item 
Assume $\repr$ irreducible.
Then, every Jordan-H\"older constituent
$\overline{\pi}$ of $\pi$
is supersingular.
\end{enumerate}
\end{corollary}
\begin{proof}
Suppose by contradiction that
$\overline{\pi}$ is not supersingular.
Then, by the characterisation of 
smooth irreducible representations
given in \cite[Theorem~33]{BL94},
$\overline{\pi}$ is isomorphic to
a character, a principal series,
or a special series.
Consider
$\lg_{\GLring} ( \socR \overline{\pi})$:
we know that it is $1$ in all cases except
for a special series, 
where it is possibly $2$
(cf.\ for example \cite[Corollary~2.38]{AWS25}),
and in that case we have 
$\dim_{\F}\overline{\pi}^{I_1}=2$
(for example, because
$|B(K) \backslash \!\GL_2(K)/I_1| = 2$).

(i)
Remember that,
by \Cref{thm:principal-series},
we have a direct sum decomposition
$\pi = \pi_0^{\oplus r } \oplus 
\pi_f^{\oplus r} \oplus \pi'$,
with
$\pi_i^{K_1} \cong \DZeroEll[i]$,
for $i \in \{0,f\}$.
Let $\overline{\pi}$ be a Jordan-H\"older
constituent of $\pi'$.
By \Cref{cor:final-corollary}(ii)
we can write
$\overline{\pi}^{K_1}$ as
\begin{equation}
	\label{eq:red-pi-bar-K1}
\overline{\pi}^{K_1} \cong
 \bigoplus_{\ell=1} ^{f-1} 
\DZeroEll[\ell]^{\oplus \rell[\overline{\pi},\ell]},
\end{equation}
for some integers 
$0 \le \rell[\overline{\pi},\ell] \le r$,
not all zero
as $\ell$ runs over $1 \le \ell \le f-1$
(in particular, $f \ge  2$).
Fix an integer $1 \le \ell_0 \le f-1$
such that $\rell[\overline{\pi},\ell_{0}]\neq 0$,
and recall the set 
$W(\repr)_{\ell}$ of \Cref{def:J-and-ell}.
Then,
$\lg_{\GLring}( \socR \overline{\pi}^{K_1})
\ge |W(\repr)_{\ell_0} |= \binom {f}{\ell_0}$,
which is strictly greater than $2$ unless
$f=2$, $\ell_0=1$, where we have
$\binom f {\ell_0} =2$. 
In this case, use that 
$\dim_{\F}(\overline{\pi}^{I_1})
\ge |\mathscr{P}_{\ell_0} |=8>2$,
by \cite[§16]{BP12},
$f=2$, case~(ii).
In both cases, we have reached
a contradiction.

(ii)
The proof goes along the same lines.
Let $\overline{\pi}$ be a Jordan-H\"older
constituent of $\pi$.
By \Cref{cor:final-corollary}(ii)
we can write
$\overline{\pi}^{K_1}$ as
\begin{equation}
	\label{eq:irr-pi-bar-K1}
\overline{\pi}^{K_1} \cong
\DZero^{\oplus \rnull[\overline{\pi}]},
\end{equation}
for some nonzero integer
$1 \le \rnull[\overline{\pi}] \le r$,
and
$\lg_{\GLring} \left(\socR\overline{\pi}\right)
\ge |W(\repr)|= 2^{f}$,
which is strictly greater than $2$
unless $f=1$, in which case we use that 
$\dim_{\F}(\overline{\pi}^{I_1})
\ge |\mathscr{P}|=4>2$,
by \cite[§16]{BP12}, $f=1$, case~(iii).
In both cases, we have reached
a contradiction.
\end{proof}

\section{Global results}
	\label{sec:global}
In this section we recall our global setup
(which coincides with the ``indefinite
case'' 
of \cite[§2.6]{BHHMS4}) in more detail,
and then prove
\Cref{thm:global-part-fl},
which states that the representation
$\pi$ of \eqref{eq:piShi-def}
has finite length.

We fix a totally real number field $F$, with
ring of integers $\mathcal{O}_F$,
and we let $S_p$ denote the set of places above $p$.
We assume that $F$ is unramified
at $S_p$.
For a place $w$, we denote by $F_w$
the completion of $F$ at $w$, and
by $\mathcal{O}_{F_w} $ its ring of integers.
We let 
$\Frob_w$
be a geometric Frobenius element at $w$.
We fix a quaternion algebra
$D$, with center $F$, which is split at all places
above $p$ and at most one infinite place,
and we let $S_D$  denote the set of places 
where $D$ ramifies.
Consequently, if $\mathcal{O}_D$
is a fixed maximal order in $D$,
we can also fix isomorphisms
$(\mathcal{O}_D)_w \defeq
\mathcal{O}_D \otimes_{\mathcal{O}_F}\mathcal{O}_{F_w}
\xrightarrow{\sim} \Mring[F_w]$
for $w \notin S_D$.

We fix a continuous representation
$\glrepr \colon
\agal[F] \to \GL_2(\F)$,
and we set 
$\glrepr_w \defeq \glrepr|_{\agal[F_w]} $
for a finite place $w$ of $F$.
Let $S_{\glrepr} $ be the set of places
where $\glrepr$ ramifies. 
Then, we make the following assumptions 
on $\glrepr$:
\begin{itemize}
\item the restriction
$\glrepr|_{\Gal(\overline{F}/F(\mu_p))} $
is absolutely irreducible;
\item for all places $w \in  S_p$,
$\glrepr_w$ is $0$-generic, 
in particular $S_p \subseteq S_{\glrepr} $;
\item for all places
$w \in (S_D \cup S_{\glrepr} ) \setminus 
 S_p$,
the universal framed deformation ring
of $\glrepr_w$ is formally smooth over 
$W(\F)$.
\end{itemize}
If $D$ splits at exactly one infinite place
(the ``indefinite case''), 
we make the following choices.

We denote by $\mathbb{A}_F^{ \infty }$
the ring of finite adèles of $F$.
Given a compact open subgroup $V$ of 
$(D \otimes_{F}\mathbb{A}_F^{\infty})^{\times}$, 
then we can consider the associated smooth
projective Shimura curve $X_{V}$, 
which is defined over $F$, with the conventions of
\cite[§3.1]{BD14}.
We choose:
\begin{enumerate}[(i)]
\item
	\label{condition:place-i}
a finite place $w_1 \notin S_D \cup S_{\glrepr}$
such that:
\begin{enumerate}[]
\item then norm $\Norm(w_1)$ is not
congruent to $1\bmod p$;
\item the ratio of the eigenvalues
of $\glrepr (\Frob_{w_1} )$ is not
in $\{ 1, \Norm(w_1), \Norm(w_1)^{-1} \} $;
\item for any nontrivial root of unity
$\zeta$ in a quadratic extension of $F$,
we have $w_1 \nmid (\zeta+ \zeta^{-1} -2)$;
\end{enumerate}
\item
	\label{condition:place-ii}
 a finite set $S$ of finite places of $F$
such that: 
\begin{enumerate}[]
\item we have
$ S_D \cup S_{\glrepr} \subseteq S$
and $w_1 \notin S$;
\item for all $w \in S \setminus S_p$, 
the framed deformation ring 
$R_{\glrepr_w ^\vee } $
of $\glrepr_w ^\vee $ is formally smooth
over $W(\F)$;
\end{enumerate}
\item 
	\label{condition:place-iii}
compact open subgroups
$V \defeq  \prod _w V_w$,
with $V_w$
subgroups of 
$(\mathcal{O}_D)_w ^{\times}$
as $w$ runs over the set of places of $F$,
such that we have:
\begin{enumerate}[]
\item an equality
$V_w = (\mathcal{O}_D)_w^{\times}$
for $w \notin S \cup \{w_1\}$;
\item 
	\label{condition:place-iii-b}
the nonvanishing of
\begin{equation}
	\label{eq:pre-pi-nonvanishing}
 \Hom_{\agal[F]}(\glrepr, 
H^{1}_{ \text{ét}}(X_V \times_{F} \overline{F}, \F) ) \neq 0.
\end{equation}
\end{enumerate}
\end{enumerate}

If $D$ splits at no infinite places 
(the ``definite'' case)
we make the same choices as 
\ref{condition:place-i}---
\ref{condition:place-iii} above,
replacing \eqref{eq:pre-pi-nonvanishing}
by the condition $S(V,\F)[\mgl] \neq 0$,
where
$S(V,\F) \defeq 
\{f \colon D^{\times} \backslash\!
(D \otimes_{F}\mathbb{A}_F^{\infty })^{\times}
/V \to \F\}$,
and where $\mgl$ is generated by 
$T_w-S_w \tr(\glrepr(\Frob_w))$,
$\Norm(w)-S_w \det(\glrepr(\Frob_w))$,
for every $w \notin S \cup \{w_1\}$ such that
$V_w = (\mathcal{O}_D)_w^{\times}$,
with $T_w$, $S_w$ acting on $S(V,\F)$
via right translation on functions by
$V \left( \begin{smallmatrix}
\varpi_w & 0 \\
0 & 1 \\
\end{smallmatrix} \right)V$,
$V \left( \begin{smallmatrix}
\varpi_w & 0 \\
0 & \varpi_w \\
\end{smallmatrix} \right)V$
respectively (where $\varpi_w$
is any choosen uniformiser of $F_w$).

For convenience, set
\begin{equation}
	\label{eq:small-global-pi-def}
\pi(V) \defeq \begin{cases}
\Hom_{\Gal(\overline{F}/F)}\left(\overline{r}, H^1_{ \text{ét}} (X_{V} \times_{F}\overline{F}, \F)\right)
 & \text{in the indefinite case,} \\
S(V, \F)[\mgl]
 & \text{in the definite case,} \\
\end{cases}
\end{equation}
so that \ref{condition:place-iii}(b)
becomes $\pi(V) \neq 0$.

Now, fix a finite place $v \in  S_p$,
and let $K \defeq F_v$,
$\repr \defeq  \glrepr^\vee$,
which we assume to be semisimple.
If $V^{v} \defeq \prod _{w \neq v} V_w$,
then we want to study
the admissible smooth representation
of $\GLfield$ over $\F$
\begin{equation}
	\label{eq:global-pi-def}
\pi(V^{v}) \defeq \begin{cases}
\varinjlim_{V_v} \pi(V^{v}V_v)
 & \text{in the indefinite case,} \\
\varinjlim_{V_v} \pi(V^{v}V_v)
 & \text{in the definite case,} \\
\end{cases}
\end{equation}
where the limit is taken over all
compact open subgroups 
$V_v \unlhd (\mathcal{O}_D)_v^{\times}
\cong \GLring$
contained inside
$1+ p \Mring$.
It follows from \eqref{eq:pre-pi-nonvanishing}
that \eqref{eq:global-pi-def} is nonzero.

If we let
$\psi \colon \agal[F] \to W(\F)^{\times}$
be the Teichm\"uller lift
of $\det (\glrepr) \omega$,
and denote by $\psi_w$ its restriction to
$\agal(F_w)$,
and by $\overline{\psi}_w$ 
the reduction modulo $p$ of $\psi_w$,
then notice that \eqref{eq:global-pi-def}
has central character 
$\overline{\psi}_v^{-1} = \mytwist$.

The following is the main result of this section,
and of the article.
\begin{newthm}[]
	\label{thm:global-part-fl}
Suppose that $\repr$ is
$\max \{12,2f+1\}$-generic.
The representation $\pi(V^{v})$ of
\eqref{eq:global-pi-def}
has finite length as a 
$\GLfield$-representation. 
\end{newthm}
\begin{newremark}
	\label{rmk:gl-conditions-for-free}
Up to enlarging $S$, we can always assume that
\ref{condition:place-iii}(a) is satisfied
(this, of course, makes 
\ref{condition:place-ii}(b) more restrictive).
Moreover, as was noted in
\cite[§6.2]{EGS}, a finite place $w_1$ satisfying
\ref{condition:place-i} always exists,
by \cite[Lemma~4.11]{DDT97}.
\end{newremark}
\begin{newremark}
	\label{rmk:global-first-reduction}
Notice that, for a smaller
 $V^{\prime v} \subseteq V^{v}$
we have 
$\pi(V^v) \subseteq \pi(V^{\prime v})$.
In particular, if $\pi(V^{\prime v})$ has finite length,
then so does $\pi(V^v)$,
which means that
we can assume, without loss of generality,
that
$V_w$ is contained in $1+ p\Mring[F_w]$
and is normal in $\GL_2(\mathcal{O}_{F_w} )$
for all $w \in S_p \setminus  \{v\},$
and that $V_{w_1} $ is contained in 
the subgroup of 
$(\mathcal{O}_D)_{w_1}^{\times}$
of matrices that are upper-triangular unipotent
modulo $w_1$.
\end{newremark}

To prove this theorem, 
we consider a larger auxiliary 
prime-to-$v$ level
$U^{v} \defeq  \prod _{w \neq v} U_w$, where 
\begin{equation}
	\label{eq:U-lvl-def}
\begin{cases}
U_w = V_w & \text{for } w \notin S_p, \\
U_w = (\mathcal{O}_D)_w^{\times}
\cong \GLring[F_w]
 &  \text{for } w \in S_p \setminus \{v\}.
\end{cases}
\end{equation}
By \Cref{rmk:global-first-reduction},
$U^{v}$ normalises $V^{v}$.
We also consider
an auxiliary
$\GLfield$-representation:
for a choice
of a Serre weight $\sigma_w$ of $\GLring[F_w])$
for each $w \in S_p \setminus \{v\} $,
set 
$\sigma_p^{v} \defeq 
\bigotimes_{w \in S_p \setminus \{v\} } \sigma_w $,
and define $\pi(V^{v},\sigma_p^{v})$ to be
\begin{equation}
	\label{def:global-small-pi-def}
\pi(V^{v},\sigma_p^{v}) \defeq 
\varinjlim _{V_v} \Hom_{U^{v}/V^{v}}
\left(\sigma_p^{v}, 
\pi(V^{v}V_v)
\right).
\end{equation}
Then, we have the following result,
 coming from \cite[§5.5]{GK14}:
\begin{equation}
	\label{eq:small-pi-and-modular-weights}
\pi(V^{v},\sigma_p^{v}) \neq 0
\iff \sigma_w \in W(\overline{r}_w ^\vee )
\ 
\forall w \in S_p \setminus \{v\} ,
\end{equation}
where $W(\overline{r}_w ^\vee )$
is defined as in 
\cite[§3]{BDJ10}.

Here is a compilation of results 
concerning the representation 
\eqref{def:global-small-pi-def}.
\begin{newthm}[]
	\label{thm:compilation}
For $w \in  S_p \setminus \{v\}$,
we let $\sigma_w$ be a Serre weight in
$W(\glrepr_w^\vee)$,
and set
$\sigma_p^{v} \defeq 
\bigotimes_{w \in S_p \setminus \{v\} } \sigma_w $.
Assume that $\repr$ is $12$-generic.
\begin{enumerate}[(i)]
\item 
(\cite[Theorem~8.4.2]{BHHMS1})
the hypothesis \ref{hypothesis:i}
of \Cref{sec:hypotheses}
holds for
$\pi(V^{v}, \sigma_p^{v})$;
	\item 
(\cite[Proposition~6.4.6]{BHHMS1} and \cite[Theorem~8.4.2]{BHHMS1})
the hypothesis \ref{hypothesis:ii}
of \Cref{sec:hypotheses}
 holds for
$\pi(V^{v}, \sigma_p^{v})$;
	\item
	(\cite[Theorem~8.2]{HW22} 
and \cite[Theorem~8.4.1]{BHHMS1})
the hypothesis \ref{hypothesis:iii}
of \Cref{sec:hypotheses}
holds for
$\pi(V^{v}, \sigma_p^{v})$;
\item 
(\cite[Proposition~2.6.2]{BHHMS4})
the hypothesis \ref{hypothesis:iv}
of \Cref{sec:hypotheses}
 holds for
$\pi(V^{v}, \sigma_p^{v})$.
\end{enumerate}
\end{newthm}

We return to \Cref{thm:global-part-fl}.
\begin{proof}[Proof of 
\Cref{thm:global-part-fl}]
Notice that 
$ H^{1}_{\text{ét}}(X_{V^{v}V_v} \times_{F} \overline{F},\F )$
is a $U^{v}/V^{v}$-representation,
in particular it is $V^{v}$-invariant,
and we can rewrite $\pi(V^{v})$ as
\begin{align}
	\label{eq:pi-rewrite-dumb}
\pi(V^{v}) \cong &
\varinjlim _{V_v} \Hom_{V^{v}}\left(1, 
\pi(V^{v}V_v)
\right)\\
	\label{eq:pi-rewrite-Frobenius}
\cong& 
\varinjlim _{V_v} \Hom_{U^{v}/V^{v}}\left(
 \bigotimes _{w \in S_p \setminus \{v\} } 
\Ind_{V_w} ^{U_w}1, 
\pi(V^{v}V_v)
\right).
\end{align}
In \eqref{eq:pi-rewrite-dumb},
$1$ denotes the trivial representation
of $V^{v}$,
and in \eqref{eq:pi-rewrite-Frobenius}
we have used Frobenius reciprocity
from $V^{v}$ to $U^{v}$.
Since the two levels only differ at 
places in
$S_p \setminus \{v\}$,
$\Ind_{V^v} ^{U^v}1$
coincides with 
$
\Ind_{\prod_{w \in S_p \setminus \{v\}} V_w} 
^{\prod_{w \in S_p \setminus \{v\}} U_w}1 
\cong 
\bigotimes _{w \in S_p \setminus \{v\} } 
\Ind_{V_w} ^{U_w}1
\overset{\eqref{eq:U-lvl-def}}{\cong} 
\bigotimes _{w \in S_p \setminus \{v\} } 
\Ind_{V_w} ^{\GLring[F_w]}1 $.

Recall that 
the centre of $\GLring[F_w]$
acts by $\overline{\psi}^{-1} |_{\inertia[F_w]} $
on \eqref{eq:global-pi-def}.
If
$(\Ind_{V_w} ^{\GLring[F_w]}1)_Z $
denotes the maximal quotient of 
$\Ind_{V_w} ^{\GLring[F_w]}1$
on which 
the centre of $\GLring[F_w]$
acts by the character $\overline{\psi}^{-1} |_{\inertia[F_w]} $,
then
\begin{equation}
	\label{eq:pi-rewrite}
\pi(V^{v}) \cong
\varinjlim _{V_v} 
\Hom_{U^{v}/V^{v}}\!\!\left(
\bigotimes _{w \in S_p \setminus \{v\} } 
(\Ind_{V_w} ^{\GLring[F_w]}1)_Z, 
\pi(V^{v}V_v)
\right)
.
\end{equation}
If we set
$\JH_{w}
\defeq \JH((\Ind_{V_w} ^{\GLring[F_w]}1)_Z )$,
then a dévissage 
on the Jordan-H\"older constituents of
$ \bigotimes _{w \in S_p \setminus \{v\} } 
(\Ind_{V_w} ^{\GLring[F_w]}1)_Z$
(which are all of the form
$\bigotimes_{S_p \setminus \{v\} }  \sigma_w$,
for $\sigma_w \in \JH_{w}$)
gives
\begin{equation}
	\label{eq:length-devissage}
\lg_{\GLfield}  \left(\pi(V^{v})\right)
\le 
\sum _{w \in S_p \setminus \{v\} } 
\sum _{\sigma_w \in \JH_w} 
\lg_{\GLfield}  \left(
\pi\left(V^{v}, \textstyle\bigotimes_{w \in S_p \setminus \{v\} }  \sigma_w\right)
\right).
\end{equation}
Finally, notice that
each term
$ \pi\left(V^{v}, \bigotimes_{w \in S_p \setminus \{v\} }  \sigma_w\right) $
is either zero
(by \eqref{eq:small-pi-and-modular-weights})
or satisfies hypotheses
\ref{hypothesis:i}
to 
\ref{hypothesis:iv}
 of \Cref{sec:hypotheses}
(by \Cref{thm:compilation}),
 and hence has finite length 
by 
\Cref{thm:finite-length-yay}(ii),
using that $\glrepr_v$ is
$\max \{12,2f+1\}$-generic.
Since the sum in the right-hand side
of \eqref{eq:length-devissage}
is finite, we conclude
that $\pi(V^{v})$ has finite length.
\end{proof}

\printbibliography
\end{document}

\typeout{get arXiv to do 4 passes: Label(s) may have changed. Rerun}